\documentclass[reqno,centertags]{amsart}
\usepackage{amsmath,amsthm,amscd,amssymb,latexsym,upref,enumerate}

\usepackage{bbm}
\usepackage{nicefrac}
\usepackage{hyperref}

\newcommand*{\mailto}[1]{\href{mailto:#1}{\nolinkurl{#1}}}
\newcommand{\arxiv}[1]{\href{http://arxiv.org/abs/#1}{arXiv:#1}}
\newcommand{\doi}[1]{\href{http://dx.doi.org/#1}{DOI:#1}}

\newtheorem{theorem}{Theorem}[section]

\newtheorem{lemma}[theorem]{Lemma}
\newtheorem{corollary}[theorem]{Corollary}
\newtheorem{hypothesis}[theorem]{Hypothesis}
\theoremstyle{definition}
\newtheorem{definition}[theorem]{Definition}
\newtheorem{remark}[theorem]{Remark}

\newtheorem{example}[theorem]{Example}

\newcommand{\R}{{\mathbb R}}
\newcommand{\N}{{\mathbb N}}

\newcommand{\C}{{\mathbb C}}

\newcommand{\bbC}{{\mathbb{C}}}

\newcommand{\bbN}{{\mathbb{N}}}

\newcommand{\bbR}{{\mathbb{R}}}
\newcommand{\bbT}{{\mathbb{T}}}

\newcommand{\cB}{{\mathcal B}}

\newcommand{\cD}{{\mathcal D}}

\newcommand{\cH}{{\mathcal H}}

\newcommand{\cM}{{\mathcal M}}

\newcommand{\cS}{{\mathcal S}}

\newcommand{\gq}{\mathfrak q}

\newcommand{\be}{\begin{equation}}
\newcommand{\ee}{\end{equation}} 
\newcommand{\spr}[2]{\langle #1 , #2 \rangle}

\newcommand{\E}{\mathrm{e}}
\newcommand{\I}{\mathrm{i}}
\newcommand{\sgn}{\mathrm{sgn}}
\newcommand{\tr}{\mathrm{tr}}
\newcommand{\im}{\mathrm{Im}}
\newcommand{\re}{\mathrm{Re}}
\newcommand{\dom}[1]{\mathrm{dom}\left(#1\right)}

\newcommand{\no}{\notag}
\newcommand{\lb}{\label}
\newcommand{\f}{\frac}

\newcommand{\ol}{\overline}

\newcommand{\wti}{\widetilde}

\newcommand{\oh}{o}

\newcommand{\bi}{\bibitem}
\newcommand{\hatt}{\widehat}

\newcommand{\prodm}{\mu \otimes \mu} 

\DeclareMathOperator{\ran}{ran}

\DeclareMathOperator{\rang}{rank}

\DeclareMathOperator{\reg}{r}
\DeclareMathOperator{\AC}{AC}
\DeclareMathOperator{\SL}{SL}
\DeclareMathOperator{\loc}{loc}
\newcommand{\sigdis}{\sig_d}
\newcommand{\sigess}{\sig_e}

\DeclareMathOperator{\linspan}{span}
\DeclareMathOperator{\supp}{supp}

\newcommand{\foco}{s}
\newcommand{\qd}{{[1]}}
\newcommand{\Tpre}{T_0}
\newcommand{\Tmin}{T_{\mathrm{min}}}
\newcommand{\Tmax}{T_{\mathrm{max}}}
\newcommand{\Llocr}{L^1_{\mathrm{loc}}((a,b);r(x)dx)}
\newcommand{\Lr}{L^2((a,b);r(x)dx)}
\newcommand{\Lrmu}{L^2(\R;d\mu)}
\newcommand{\Deftau}{\mathfrak{D}_\tau}
\newcommand{\BCa}{BC_{a}}
\newcommand{\BCb}{BC_{b}}
\newcommand{\indik}{\chi}
\newcommand{\M}{\mathrm{M}}

\renewcommand{\Re}{\text{\rm Re}}
\renewcommand{\Im}{\text{\rm Im}}

\renewcommand{\max}{\text{\rm max}}
\renewcommand{\min}{\text{\rm min}}

\renewcommand{\le}{\leqslant}

\newcommand{\eps}{\varepsilon}
\newcommand{\vphi}{\varphi}
\newcommand{\sig}{\sigma}
\newcommand{\lam}{\lambda}

\numberwithin{equation}{section}

\begin{document}

\title[Weyl--Titchmarsh Theory]{Weyl--Titchmarsh Theory for Sturm--Liouville Operators with 
Distributional Potentials}

\author[J.\ Eckhardt]{Jonathan Eckhardt}
\address{Faculty of Mathematics\\ University of Vienna\\
Nordbergstrasse 15\\ 1090 Wien\\ Austria}
\email{\mailto{jonathan.eckhardt@univie.ac.at}}
\urladdr{\url{http://homepage.univie.ac.at/jonathan.eckhardt/}}

\author[F.\ Gesztesy]{Fritz Gesztesy}
\address{Department of Mathematics,
University of Missouri,
Columbia, MO 65211, USA}
\email{\mailto{gesztesyf@missouri.edu}}
\urladdr{\url{http://www.math.missouri.edu/personnel/faculty/gesztesyf.html}}

\author[R.\ Nichols]{Roger Nichols}
\address{Mathematics Department, The University of Tennessee at Chattanooga, 
415 EMCS Building, Dept. 6956, 615 McCallie Ave, Chattanooga, TN 37403, USA}
\email{\mailto{Roger-Nichols@utc.edu}}

\author[G.\ Teschl]{Gerald Teschl}
\address{Faculty of Mathematics\\ University of Vienna\\
Nordbergstrasse 15\\ 1090 Wien\\ Austria\\ and International
Erwin Schr\"odinger
Institute for Mathematical Physics\\ Boltzmanngasse 9\\ 1090 Wien\\ Austria}
\email{\mailto{Gerald.Teschl@univie.ac.at}}
\urladdr{\url{http://www.mat.univie.ac.at/~gerald/}}

\thanks{{\it Opuscula Math.} {\bf 33}, 467--563 (2013).}
\thanks{{\it Research supported by the Austrian Science Fund (FWF) 
under Grant No.\ Y330}} 


\begin{abstract}
We systematically develop Weyl--Titchmarsh theory for singular differential operators on arbitrary 
intervals $(a,b) \subseteq \mathbb{R}$ associated with rather general differential expressions of the type 
\[
 \tau f = \frac{1}{r} \left( - \big(p[f' + s f]\big)' + s p[f' + s f] + qf\right),    
\]
where the coefficients $p$, $q$, $r$, $s$ are real-valued and Lebesgue measurable on $(a,b)$, with 
$p\neq 0$, $r>0$ a.e.\ on $(a,b)$, and $p^{-1}$, $q$, $r$, $s \in L^1_{\text{loc}}((a,b); dx)$,  
and $f$ is supposed to satisfy
\[
f \in AC_{\text{loc}}((a,b)), \; p[f' + s f] \in AC_{\text{loc}}((a,b)).  
\]
In particular, this setup implies that $\tau$ permits a distributional potential coefficient, 
including potentials in $H^{-1}_{\text{loc}}((a,b))$. 

We study maximal and minimal Sturm--Liouville operators, all  self-adjoint restrictions of the maximal operator $T_{\text{max}}$, or equivalently, all self-adjoint extensions of the minimal operator $T_{\text{min}}$, 
all self-adjoint boundary conditions (separated and coupled ones), and describe the resolvent of any 
self-adjoint extension of $T_{\text{min}}$. In addition, we characterize the principal object of this paper, the singular Weyl--Titchmarsh--Kodaira $m$-function corresponding to any self-adjoint extension with 
separated boundary conditions and derive the corresponding spectral transformation, including a characterization of spectral multiplicities and minimal supports of standard subsets of the spectrum. 
We also deal with principal solutions and characterize the Friedrichs extension of $T_{\text{min}}$.

Finally, in the special case where $\tau$ is regular, we characterize the Krein--von Neumann extension 
of $T_{\text{min}}$ and also characterize all boundary conditions that lead to positivity preserving,    
equivalently, improving, resolvents (and hence semigroups).  
\end{abstract}

\maketitle

{\scriptsize \tableofcontents}

\section{Introduction}
\label{s1}

The prime motivation behind this paper is to develop Weyl--Titchmarsh theory for singular Sturm--Liouville operators on an arbitrary interval $(a,b) \subseteq \bbR$ associated with rather general differential expressions of the type 
\be
 \tau f = \frac{1}{r} \left( - \big(p[f' + \foco f]\big)' + \foco p[f' + \foco f] + qf\right).    \lb{1.1}
\ee
Here the coefficients $p$, $q$, $r$, $\foco$ are real-valued and Lebesgue measurable on $(a,b)$, with 
$p\not=0$, $r>0$ a.e.\ on $(a,b)$, and $p^{-1}$, $q$, $r$, $\foco \in L^1_{\loc}((a,b); dx)$,  
and $f$ is supposed to satisfy
\begin{equation} 
f \in AC_{\text{loc}}((a,b)), \; p[f' + \foco f] \in AC_{\text{loc}}((a,b)), 
\end{equation}
with $AC_{\loc}((a,b))$ denoting the set of locally absolutely continuous functions on $(a,b)$. (The 
expression $f^{[1]} = p[f'+\foco f]$ will subsequently be called the {\it first quasi-derivative} of $f$.)

One notes that in the general case \eqref{1.1}, the differential expression is formally given by
\be
\tau f = \frac{1}{r} \left( - \big(pf'\big)' + \big[- (p \foco)' + p \foco^2 + q\big]f \right).
\ee
Moreover, in the special case $\foco\equiv 0$ this approach reduces to the standard one, that is, 
one obtains,
\be
\tau f = \frac{1}{r} \left( - \big(pf'\big)' + q f \right).
\ee
 
In particular, in the case $p=r=1$ our approach is sufficiently general to include arbitrary distributional potential coefficients from $H^{-1}_{\loc}((a,b)) = W^{-1,2}_{\loc}((a,b))$ (as the term 
$\foco^2$ can be absorbed in $q$), and thus even in this special case our setup is slightly more general than the approach pioneered by Savchuk and Shkalikov \cite{SS99}, who defined the differential expression as
\be
\tau f = - \big([f' + \foco f]\big)' + \foco [f' + \foco f] - \foco^2 f, \quad f, [f' + \foco f] \in \AC_{\loc}((a,b)). 
\ee
One observes that in this case $q$ can be absorbed in $\foco$ by virtue of the transformation 
$\foco \to \foco - \int^x q$. Their approach requires the additional condition $\foco^2 \in L^1_{\loc}((a,b); dx)$.  Moreover, since there are distributions in $H^{-1}_{\loc}((a,b))$ which are not measures, the operators discussed here are not a special case of Sturm--Liouville operators with measure-valued coefficients as discussed, for instance, in \cite{ET12}.

We emphasize that similar differential expressions have already been studied by 
Bennewitz and Everitt \cite{BE83} in 1983 (see also \cite[Sect.\ I.2]{EM99}). While some of their discussion is more general, they restrict their considerations to compact intervals and focus on the special case of a left-definite setting. An extremely thorough and systematic investigation, including even and odd higher-order operators defined in terms of appropriate quasi-derivatives, and in the general case of matrix-valued coefficients (including distributional potential coefficients in the context of Schr\"odinger-type operators) was presented by Weidmann \cite{We87} in 1987. In fact, the general approach in \cite{BE83} and \cite{We87} draws on earlier discussions of quasi-derivatives in Shin \cite{Sh38}--\cite{Sh43}, 
Naimark \cite[Ch.\ V]{Na68}, and Zettl \cite{Ze75}. 
Still, it appears that the distributional coefficients treated in \cite{BE83} did not catch on and subsequent authors referring to this paper mostly focused on the various left and right-definite aspects developed therein. Similarly, it seems likely that the extraordinary generality exerted by Weidmann \cite{We87} in his treatment of higher-order differential operators obscured the fact that he already dealt with distributional potential coefficients back in 1987.   

There were actually earlier papers dealing with Schr\"odinger operators involving strongly singular and oscillating potentials which should be mentioned in this context, such as, 
Baeteman and Chadan \cite{BC75}, \cite{BC76}, Combescure \cite{Co80}, 
Combescure and Ginibre \cite{CG76}, Pearson \cite{Pe79}, 
Rofe-Beketov and Hristov \cite{RH66}, \cite{RH69}, and a more recent contribution treating distributional potentials by Herczy\'nski \cite{He89}. 

In addition, the case of point interactions as particular distributional potential coefficients in Schr\"odinger operators received enormous attention, too numerous to be mentioned here in detail. Hence, we only refer to the standard monographs by Albeverio, Gesztesy, H{\o}egh-Krohn, and Holden \cite{AGHKH05} and  
Albeverio and Kurasov \cite{AK01}, and some of the more recent developments in 
Albeverio, Kostenko, and Malamud \cite{AKM10}, Kostenko and Malamud \cite{KM10}, \cite{KM10a}.
We also mention the case of discontinuous Schr\"odinger operators originally considered by Hald \cite{Ha84},
motivated by the inverse problem for the torsional modes of the earth. For recent development in this direction we refer to Shahriari, Jodayree Akbarfam, and Teschl \cite{SJT12}. 

It was not until 1999 that Savchuk and Shkalikov \cite{SS99} started a new development for Sturm--Liouville (resp., Schr\"odinger) operators with distributional potential coefficients
in connection with 
areas such as, self-adjointness proofs, spectral and inverse spectral theory, oscillation properties, spectral properties in the non-self-adjoint context, etc. In addition to the important series of papers by Savchuk and Shkalikov \cite{SS99}--\cite{SS10}, we also mention other groups such as Albeverio, Hryniv, and Mykytyuk \cite{AHM08}, Bak and Shkalikov \cite{BS02},  Ben Amara and Shkalikov \cite{BS09}, Ben Amor and Remling \cite{BR05}, Davies \cite{Da13}, 
Djakov and Mityagin \cite{DM09}--\cite{DM12}, Eckhardt and Teschl \cite{ET12}, 
Frayer, Hryniv, Mykytyuk, and Perry \cite{FHMP09}, Gesztesy and Weikard \cite{GW12}, 
Goriunov and Mikhailets \cite{GM10}, \cite{GM10a}, Goriunov, Mikhailets, and Pankrashkin \cite{GMP12}, Hryniv \cite{Hr11}, 
Kappeler and M\"ohr \cite{KM01}, Kappeler, Perry, Shubin, and Topalov \cite{KPST05}, 
Kappeler and Topalov \cite{KT04}, 
Hryniv and Mykytyuk \cite{HM01}--\cite{HM12}, Hryniv, Mykytyuk, and Perry \cite{HMP11}--\cite{HMP11a}, 
Kato \cite{Ka10}, Korotyaev \cite{Ko03}, \cite{Ko12}, Maz'ya and Shaposhnikova \cite[Ch.\ 11]{MS09}, 
Maz'ya and Verbitsky \cite{MV02}--\cite{MV06}, Mikhailets and Molyboga \cite{MM04}--\cite{MM09}, 
Mirzoev and Safanova \cite{MS11}, Mykytyuk and Trush \cite{MT10}, 
Sadovnichaya \cite{Sa10}, \cite{Sa11}.  

It should be mentioned that some of the attraction in connection with distributional potential coefficients 
in the Schr\"odinger operator clearly stems from the low-regularity investigations of solutions of the 
Korteweg--de Vries (KdV) equation. We mention, for instance, Buckmaster and Koch \cite{BK12}, 
Grudsky and Rybkin \cite{GR12}, Kappeler and M\"ohr \cite{KM01}, Kappeler and Topalov \cite{KT05}, 
\cite{KT06}, and Rybkin \cite{Ry10}. 

The case of strongly singular potentials at an endpoint and the associated Weyl--Titchmarsh--Kodaira theory 
for Schr\"odinger operators can already be found in the seminal paper by Kodaira \cite{Ko49}. A gap in
Kodaira's approach was later circumvented by Kac \cite{Ka67}. The theory did not receive much further attention
until it was independently rediscovered and further developed by Gesztesy and Zinchenko \cite{GZ06}.
This soon led to a systematic development of Weyl--Titchmarsh theory for strongly singular potentials and we mention, for instance, Eckhardt \cite{E12}, Eckhardt and Teschl \cite{ET12a}, Fulton \cite{Fu08}, Fulton and Langer \cite{FL10}, 
Fulton, Langer, and Luger \cite{FLL12}, 
Kostenko, Sakhnovich, and Teschl \cite{KST12}, \cite{KST12a}, \cite{KT11}, 
\cite{KT12}, and Kurasov and Luger \cite{KL11}. 

In contrast, Weyl--Titchmarsh theory in the presence of distributional potential coefficients, especially, in connection with \eqref{1.1} (resp., \eqref{2.2}) has not yet been developed in the literature, and it is precisely the purpose of this paper to accomplish just that under the full generality of Hypothesis \ref{h2.1}. Applications to inverse spectral theory will be given in \cite{EGNT12a}.

It remains to briefly describe the content of this paper: Section \ref{s2} develops the basics of 
Sturm--Liouville equations under our general hypotheses on $p$, $q$, $r$, $\foco$, including the Lagrange identity and unique solvability of initial value problems. Maximal and minimal Sturm--Liouville operators are introduced in Section \ref{s3}, and Weyl's alternative is described in Section \ref{s4}. Self-adjoint restrictions of the maximal operator, or equivalently, self-adjoint extensions of the minimal operator, are the principal subject of 
Section \ref{s5}, and all self-adjoint boundary conditions (separated and coupled ones) are described in 
Section \ref{s6}. The resolvent of all self-adjoint extensions and some of their spectral properties are 
discussed in Section \ref{s7}. The singular Weyl--Titchmarsh--Kodaira $m$-function corresponding to 
any self-adjoint extension with separated boundary conditions is introduced and studied in 
Section \ref{s8}, and the corresponding spectral transformation is derived in Section \ref{s9}. Classical  spectral multiplicity results for Schr\"odinger operators due to Kac \cite{Ka62}, \cite{Ka63} (see also Gilbert \cite{Gi98} and Simon \cite{Si05}) are extended to our general situation in Section \ref{s10}. 
Section \ref{s11} deals with various applications of the abstract theory developed in this paper.  More specifically, we prove a simple analogue of the classic Sturm separation theorem on the separation of zeros of two real-valued solutions to the distributional Sturm--Liouville equation $(\tau - \lambda)u=0$, $\lambda \in \bbR$, and show the existence of {\it principal solutions} under certain sign-definiteness assumptions on the coefficient $p$ near an endpoint of the basic interval $(a,b)$.\ When $\tau-\lambda$ is non-oscillatory at an endpoint, we present a sufficient criterion on $r$ and $p$ for $\tau$ to be in the limit-point case at that endpoint. This condition dates back to Hartman \cite{Ha48} (in the special case $p=r=1$, $\foco=0$), and was subsequently studied by Rellich \cite{Re51} (in the case $\foco=0$). This section concludes with a detailed characterization of the Friedrichs extension of $T_0$ in terms of (non-)principal solutions, closely following a seminal paper by Kalf \cite{Ka78} (also in the case $\foco=0$). In Section \ref{s12}  we characterize the 
Krein--von Neumann self-adjoint extension of $\Tmin$ by explicitly determining the boundary conditions associated to it.  In our final Section \ref{s13}, we derive the quadratic form associated to each self-adjoint extension of $\Tmin$, assuming $\tau$ is regular on $(a,b)$. We then combine this with the 
Beurling--Deny criterion to present a characterization of all positivity preserving resolvents (and hence semigroups) associated with self-adjoint extensions of $\Tmin$ in the regular case. In particular, this result confirms that the Krein--von Neumann extension does not generate a positivity preserving resolvent or semigroup. We actually go a step further and prove that the notions of positivity preserving and positivity improving are equivalent in the regular case. 

We also mention that an entirely different approach to Schr\"odinger operators (assumed to be bounded from below) with matrix-valued distributional  potentials, based on supersymmetric considerations, has been developed simultaneously in \cite{EGNT12}.

Finally, we briefly summarize some of the notation used in this paper: The Hilbert spaces used in this 
paper  are typically of the form $\Lr$ with scalar product denoted by $\spr{\cdot\,}{\cdot}_{r}$ (linear 
in the first factor), associated norm $\|\cdot\|_{2,r}$, and corresponding identity operator denoted 
by $I_r$. Moreover, $L^2_c((a,b);r(x)dx)$ denotes the space of square integrable functions with compact support. In addition, we use the Hilbert space $\Lrmu$ for an appropriate Borel 
measure $\mu$ on $\bbR$ with scalar product and norm abbreviated by $\spr{\cdot\,}{\cdot}_\mu$ 
and $\|\cdot\|_{2,\mu}$, respectively. 

Next, let $T$ be a linear operator mapping (a subspace of) a
Hilbert space into another, with $\dom{T}$, $\ran(T)$, and $\ker(T)$ denoting the
domain, range, and kernel (i.e., null space) of $T$. The closure of a closable 
operator $S$ is denoted by $\ol S$. The spectrum, essential spectrum, point spectrum, discrete spectrum, absolutely continuous spectrum, and resolvent set of a closed linear operator in the underlying Hilbert 
space  will be denoted by $\sigma(\cdot)$, $\sigma_{ess}(\cdot)$, $\sigma_{p}(\cdot)$, 
$\sigma_{d}(\cdot)$, $\sigma_{ac}(\cdot)$, 
and $\rho(\cdot)$, respectively. The Banach spaces of linear bounded, compact, and Hilbert--Schmidt  operators in a separable complex Hilbert space are denoted by $\cB(\cdot)$, $\cB_\infty(\cdot)$, and 
$\cB_2(\cdot)$, respectively.  The orthogonal complement of a subspace $\cS$ of the Hilbert space $\cH$ will be denoted by $\cS^{\perp}$.  

The symbol $\SL_2(\bbR)$ will be used to denote the special linear group of order two over $\bbR$, that is, the set of all $2\times 2$ matrices with real entries and determinant equal to one.

At last, we will use the abbreviations ``iff'' for ``if and only if'', ``a.e.'' for ``almost everywhere'', and ``$\supp$'' for the support of functions throughout this paper.

\section{The Basics on Sturm--Liouville Equations} \lb{s2}

In this section we provide the basics of Sturm--Liouville equations with distributional potential coefficients.

Throughout this paper we make the following set of assumptions:

\begin{hypothesis} \lb{h2.1}
Suppose $(a,b) \subseteq \bbR$ and assume that 
$p$, $q$, $r$, $\foco$ are Lebesgue measurable on $(a,b)$ with $p^{-1}$, $q$, $r$, $s\in L^1_{\loc}((a,b); dx)$ and real-valued a.e.\ on $(a,b)$ with $r>0$ and $p \neq 0$  a.e.\ on $(a,b)$.
\end{hypothesis}

Assuming Hypothesis \ref{h2.1} and introducing the set, 
\begin{equation}
\Deftau=\big\{g\in AC_{\text{loc}}((a,b))\, \big|\, g^{[1]}=p [g' + \foco g] \in AC_{\text{loc}}((a,b))\big\}, 
\end{equation}
the differential expression $\tau$ considered in this paper is of the type,
\be
 \tau f = \frac{1}{r} \left( - \big(f^\qd\big)' + \foco f^\qd + qf\right) \in\Llocr,\quad f\in \Deftau.  \lb{2.2}
\ee
The expression
\begin{equation}
f^{[1]}=p [f' + \foco f], \quad f \in \Deftau,    \lb{2.3}
\end{equation}
will be called the {\it first quasi-derivative} of $f$.

Given some $g\in\Llocr$, the equation $(\tau-z) f=g$ is equivalent to the system of ordinary differential equations
\be\label{eqn:system}
 \begin{pmatrix} f \\ f^\qd \end{pmatrix}' 
  = \begin{pmatrix}  - \foco & p^{-1} \\ q - zr & \foco  \end{pmatrix}  \begin{pmatrix} f \\ f^\qd  \end{pmatrix} - \begin{pmatrix} 0 \\ rg \end{pmatrix}.
\ee
From this, we immediately get the following existence and uniqueness result.

\begin{theorem}\label{thm:exisuniq}
 For each $g\in \Llocr$, $z\in\C$, $c\in(a,b)$, and $d_1$, $d_2\in\C$ there is a unique solution  $f\in\Deftau$ of $(\tau - z) f = g$ with $f(c)=d_1$ and $f^\qd(c)=d_2$. If, in addition, $g$, $d_1$, $d_2$, and $z$ are real-valued, then the solution $f$ is real-valued.
\end{theorem}

For each $f, g\in\Deftau$ we define the modified Wronski determinant
\be
 W(f,g)(x) = f(x)g^\qd(x) - f^\qd(x)g(x), \quad x\in (a,b).
\ee
The Wronskian is locally absolutely continuous with derivative
\be
  W(f,g)'(x) = \left[g(x) (\tau f)(x) - f(x) (\tau g)(x)\right] r(x), \quad x\in(a,b). 
\ee
Indeed, this is a consequence of the following Lagrange identity, which is readily proved using integration by parts.

\begin{lemma}\label{propLagrange}
 For each $f$, $g\in\Deftau$ and $\alpha, \beta\in(a,b)$ we have 
\be\label{eqn:lagrange}
 \int_\alpha^\beta \left[g(x) (\tau f)(x)  - f(x) (\tau g)(x)\right] \, r(x) dx = W(f,g)(\beta) - W(f,g)(\alpha).
\ee
\end{lemma}

As a consequence, one verifies that the Wronskian $W(u_1,u_2)$ of two 
solutions $u_1$, $u_2\in\Deftau$ of $(\tau-z)u=0$ is constant.
Furthermore, $W(u_1,u_2)  \not=0$ if and only if  $u_1$,  $u_2$ are linearly independent. In fact, the Wronskian of two linearly dependent solutions vanishes obviously. 
Conversely, $W(u_1,u_2) = 0$ means that for $c\in(a,b)$ there is a $K\in\C$ such that 
\be
 K u_1(c) = u_2(c) \, \text{ and } \, K u_1^\qd(c) = u_2^\qd(c),
\ee
where we assume, without loss of generality, that $u_1$ is a nontrivial solution (i.e., not vanishing 
identically). Now by uniqueness of solutions this implies the linear dependence of $u_1$ and $u_2$.

\begin{lemma}\label{prop:repsol}
 Let $z\in\C$, $u_1$, $u_2$ be two linearly independent solutions of $(\tau-z)u =0$ and $c\in(a,b)$, $d_1, d_2\in\C$, $g\in \Llocr$.
 Then there exist $c_1$, $c_2\in\C$ such that the solution $u$ of $(\tau-z)f = g$  with $f(c)=d_1$ and $f^\qd(c)=d_2$,
 is given for each $x\in(a,b)$ by
\begin{align}
 \begin{split}
 f(x) &= c_1 u_1(x) + c_2 u_2(x) + \frac{u_1(x)}{W(u_1,u_2)} \int_c^x u_2(t) g(t) \, r(t) dt   \\
 & \quad - \frac{u_2(x)}{W(u_1,u_2)} \int_c^x u_1(t) g(t) \, r(t) dt,  \end{split} \\
 \begin{split}
 f^\qd(x) & = c_1 u_1^\qd(x) + c_2 u_2^\qd(x) + \frac{u_1^\qd(x)}{W(u_1,u_2)} \int_c^x u_2(t) g(t) \, r(t) dt \\
 & \quad - \frac{u_2^\qd(x)}{W(u_1,u_2)} \int_c^x u_1(t) g(t) \, r(t) dt. \end{split}
\end{align}
If $u_1$, $u_2$ is the fundamental system of solutions of $(\tau -z) u =0$ satisfying $u_1(c)=u_2^\qd(c) = 1$ and $u_1^\qd(c)=u_2(c)=0$, then $c_1=d_1$ and $c_2=d_2$.
\end{lemma} 

We omit the straightforward calculations underlying the proof of Lemma \ref{prop:repsol}.
 Another important identity for the Wronskian is the well-known Pl\"{u}cker identity:

\begin{lemma}
 For all $f_1, f_2, f_3, f_4\in\Deftau$ one has
 \begin{equation}
  0 = W(f_1,f_2)W(f_3,f_4) + W(f_1,f_3)W(f_4,f_2) + W(f_1,f_4)W(f_2,f_3).  \lb{2.12}
 \end{equation}
\end{lemma}

We say $\tau$ is {\it regular} at $a$, if $p^{-1}$, $q$, $r$, and $\foco$ are integrable near $a$.
Similarly, we say $\tau$ is regular at $b$ if these functions are integrable near $b$.
Furthermore, we say $\tau$ is regular on $(a,b)$ if $\tau$ is regular at both endpoints $a$ and $b$.

\begin{theorem}\label{thm:EEreg}
 Let $\tau$ be regular at $a$, $z\in\C$, and $g\in L^1((a,c);r(x)dx)$ for each $c\in(a,b)$. 
 Then for every solution $f$ of $(\tau-z)f=g$ the limits
\begin{align}
 f(a) = \lim_{x\downarrow a} f(x) \, \text{ and } \, f^\qd(a) = \lim_{x\downarrow a} f^\qd(x)
\end{align}
exist and are finite.
For each $d_1$, $d_2\in\C$ there is a unique solution of $(\tau-z)f=g$ with $f(a)=d_1$ and $f^\qd(a)=d_2$. 
 Furthermore, if $g$, $d_1$, $d_2$, and $z$ are real, then the solution is real.
Similar results hold for the right endpoint $b$.
\end{theorem}
\begin{proof} 
This theorem is an immediate consequence of the corresponding result for the equivalent system \eqref{eqn:system}.
\end{proof}

Under the assumptions of Theorem \ref{thm:EEreg} one also infers that Lemma \ref{prop:repsol} remains 
 valid even in the case when $c=a$ (resp., $c=b$).

We now turn to analytic dependence of solutions on the spectral parameter $z\in\C$.

\begin{theorem}\label{thmMSLEanaly}
  Let $g\in \Llocr$, $c\in(a,b)$, $d_1, d_2\in\C$ and for each $z\in\C$ let $f_z$ be the unique solution of  $(\tau - z) f = g$  with $f(c)=d_1$ and  $f^\qd(c)=d_2$.
  Then $f_z(x)$ and $f_z^\qd(x)$ are entire
  functions of order $\nicefrac{1}{2}$ in $z$ for each $x\in(a,b)$.
  Moreover, for each $\alpha, \beta\in(a,b)$ with $\alpha<\beta$ we have
  \begin{align}
   |f_z(x)| + |f_z^\qd(x) | \leq C \E^{B\sqrt{|z|}}, \quad x\in[\alpha,\beta],~z\in\C,
  \end{align}
  for some constants $C$, $B\in\R$.
\end{theorem}
\begin{proof}
 The analyticity part follows from the corresponding result for the equivalent system. 
 For the remaining part, first note that because of Lemma \ref{prop:repsol} it suffices to consider the case when $g$ vanishes identically.
 Now if we set for each $z\in\C$ with $|z|\geq 1$  
 \begin{align} 
  v_z(x) = |z| |f_z(x)|^2 + |f_z^\qd(x)|^2, \quad x\in(a,b),
 \end{align}
 an integration by parts shows that for each $x\in(a,b)$
 \begin{align}
 \begin{split}
  v_z(x) = v_z(c) & - 
      \int_c^x 2 \big[|z| |f_z(t)|^2 -  |f_z^\qd(t)|^2 \big] \foco(t) \, dt \\
         & + \int_c^x 2\,\re\Big(f_z(t) \ol{f_z^\qd(t)}\Big) \big[|z| p(t)^{-1} + q(t) \big] dt \\
         & - \int_c^x  2\,\re\Big(zf_z(t) \ol{f_z^\qd(t)}\Big) \, r(t) dt.
 \end{split}
 \end{align}
Employing the elementary estimate
 \begin{align}
  2| f_z(x) f_z^\qd(x)| \leq \frac{ |z| |f_z(x)|^2 + |f_z^\qd(x)|^2}{\sqrt{|z|}} = \frac{v_z(x)}{\sqrt{|z|}}, \quad x\in(a,b),
 \end{align}
 we obtain an upper bound for $v_z$:
 \begin{align}
  v_z(x) & \leq v_z(c) + 2 \left| \int_c^x v_z(t) \sqrt{|z|} \omega(t) dt\right|, \quad x\in(a,b), 
 \end{align}
 where $\omega=|p^{-1}|+ |q| + |r| + |\foco|$.
 Now an application of the Gronwall lemma yields
 \begin{align}
  v_z(x) \leq v_z(c) \E^{2 \sqrt{|z|} \left|\int_c^x \, \omega(t) dt\right|}, \quad x\in(a,b).
 \end{align}
\end{proof}

If, in addition to the assumptions of Theorem \ref{thmMSLEanaly},
 $\tau$ is regular at $a$ and $g$ is integrable near $a$, then the limits $f_z(a)$ and $f_z^\qd(a)$ are entire 
 functions of order $\nicefrac{1}{2}$ and the bound in Theorem \ref{thmMSLEanaly} holds for all $x\in[a,\beta]$. 
 Indeed, this follows since the entire functions $f_z(x)$ and $f_z^\qd(x)$, $x\in(a,c)$ are locally bounded, uniformly in $x\in(a,c)$. 
Moreover, in this case the assertions of Theorem \ref{thmMSLEanaly} are valid even if we take $c=a$ and/or $\alpha=a$.

\section{Sturm--Liouville Operators}  \lb{s3}

In this section, we will introduce operators associated with our differential expression $\tau$ in the Hilbert space $\Lr$ with scalar product
\be
 \spr{f}{g}_{r} = \int_{a}^b  f(x) \ol{g(x)} \, r(x) dx, \quad f,\, g\in\Lr.
\ee
First, we define the maximal operator $\Tmax$ in $\Lr$ by
\begin{align}
& \Tmax f = \tau f,   \\
& f \in \dom{\Tmax} = \left\lbrace g\in\Lr \,\big|\, g\in\Deftau,~ \tau g\in\Lr \right\rbrace.  \no
\end{align}
In order to obtain a symmetric operator, we restrict the maximal operator $\Tmax$ to functions with compact support by
\begin{align}
\begin{split}
& \Tpre f = \tau f,    \\
& f \in \dom{\Tpre} = \left\lbrace g\in\dom{\Tmax} \,|\, g \text{ has compact support in }(a,b) \right\rbrace. 
\end{split}\lb{3.3}
\end{align} 
 Since $\tau$ is a real differential expression, the operators $\Tpre$ and $\Tmax$ are real with respect 
 to the natural conjugation in $\Lr$.

We say some measurable function $f$ lies in $\Lr$ near $a$ (resp., near $b$) if $f$ lies in $L^2((a,c);r(x)dx)$ (resp., in $L^2((c,b);r(x)dx)$) for each $c\in(a,b)$. 
 Furthermore, we say some $f\in\Deftau$ lies in $\dom{\Tmax}$ near $a$ (resp., near $b$) if $f$ and $\tau f$ both 
 lie in $\Lr$ near $a$ (resp., near $b$).
One readily verifies that some $f\in\Deftau$ lies in $\dom{\Tmax}$ near $a$ (resp., $b$) if and only if
$\ol{f}$ lies in $\dom{\Tmax}$ near $a$ (resp., $b$).

\begin{lemma}
 If $\tau$ is regular at $a$ and $f$ lies in $\dom{\Tmax}$ near $a$, then the limits
\begin{align}
 f(a) = \lim_{x\downarrow a} f(x) \, \text{ and } \, f^\qd(a) = \lim_{x\downarrow a} f^\qd(x)
\end{align}
exist and are finite. Similar results hold at $b$.
\end{lemma}
\begin{proof}
 Under the assumptions of the lemma, $\tau f$ lies in $\Lr$ near $a$ and since $r(x)dx$ is
 a finite measure near $a$ we have $\tau f\in L^1((a,c);r(x)dx)$ for each $c\in(a,b)$. 
 Hence, the claim follows from Theorem \ref{thm:EEreg}.
\end{proof}

The following lemma is a consequence of the Lagrange identity.

\begin{lemma}\label{lem:l2lagrange}
If $f$ and $g$ lie in $\dom{\Tmax}$ near $a$, then the limit
\be
 W(f,\ol{g})(a) = \lim_{\alpha\downarrow a} W(f,\ol{g})(\alpha)
\ee
exists and is finite. A similar result holds at the endpoint $b$.
If $f$, $g\in\dom{\Tmax}$, then
\be
 \spr{\tau f}{g}_{r} - \spr{f}{\tau g}_{r} = W(f,\ol{g})(b) - W(f,\ol{g})(a) =: W_a^b(f,\ol{g}).
\ee
\end{lemma}
\begin{proof}
 If $f$ and $g$ lie in $\dom{\Tmax}$ near $a$, the limit $\alpha\downarrow a$ of the left-hand side in equation \eqref{eqn:lagrange} exists.
 Hence, the limit in the claim exists as well.
 Now the remaining part follows by taking the limits $\alpha\downarrow a$ and $\beta\uparrow b$.
\end{proof}

If $\tau$ is regular at $a$ and $f$ and $g$ lie in $\dom{\Tmax}$ near $a$, then we clearly have
\be
 W(f,\ol{g})(a) = f(a) \ol{g^\qd(a)} - f^\qd(a) \ol{g(a)}.
\ee
In order to determine the adjoint of $\Tpre$ we will rely on the following lemma (see, e.g., 
\cite[Lemma\ 9.3]{Te09} or \cite[Theorem\ 4.1]{We80}).

\begin{lemma}\label{lem:hilfslemmaadjoint}
 Let $V$ be a vector space over $\C$ and $F_1,\ldots,F_n,F$ linear functionals defined on $V$. Then
\be
 F\in\linspan\left\lbrace F_1,\ldots,F_n\right\rbrace \, \text{ iff } \,\,  
 \bigcap_{j=1}^n \ker (F_j) \subseteq \ker (F).
\ee
\end{lemma}

\begin{theorem}
 The operator $\Tpre$ is densely defined and $T^*_0 =\Tmax$.
\end{theorem}
\begin{proof}
 If we set
 \be
 \wti \Tpre^\ast = \big\{ (f_1,f_2)\in \Lr^2 \,\big|\, \forall g\in\dom{\Tpre}: \spr{f_1}{\Tpre g}_{r} = \spr{f_2}{g}_{r} \big\},
 \ee
 then from Lemma \ref{lem:l2lagrange} one immediately sees that the graph of $\Tmax$ is contained in $\wti \Tpre^\ast$. Indeed, for each 
  $f\in\dom{\Tmax}$ and $g\in\dom{\Tpre}$ we infer
  \begin{align}
    \spr{\tau f}{g}_{r} - \spr{f}{\tau g}_{r} & = \lim_{\beta\uparrow b} W(f,\ol{g})(\beta) - \lim_{\alpha\downarrow a}W(f,\ol{g})(\alpha) = 0,  
  \end{align}
  since $W(f,\ol{g})$ has compact support.
 Conversely, let $f_1$, $f_2\in\Lr$ such that $\spr{f_1}{\Tpre g}_{r}=\spr{f_2}{g}_{r}$ for each $g\in\dom{\Tpre}$ and $f$ be a solution of $\tau f = f_2$.
 In order to prove that $f_1-f$ is a solution of $\tau u = 0$, we will invoke Lemma \ref{lem:hilfslemmaadjoint}.
 Therefore, consider the linear functionals 
\begin{align}
 \ell(g)   & = \int_{a}^b \ol{\left(f_1(x)-f(x)\right)} g(x) \, r(x) dx, \quad g\in L^2_c((a,b);r(x)dx), \\
 \ell_j(g) & = \int_{a}^b \ol{u_j(x)} g(x) \, r(x) dx, \quad g\in L^2_c((a,b);r(x)dx), \; j=1,2,
\end{align}
where $u_j$ are two solutions of $\tau u=0$ with $W(u_1,u_2)=1$ and $L^2_c((a,b);r(x)dx)$ is the space of square integrable functions with compact support. 
For these functionals we have $\ker (\ell_1) \cap \ker (\ell_2) \subseteq \ker (\ell)$.
Indeed, let $g\in\ker(\ell_1) \cap \ker (\ell_2)$, then the function
\begin{align}
 u(x) = u_1(x)\int_a^x u_2(t) g(t) \, r(t) dt + u_2(x) \int_x^b u_1(t) g(t) \, r(t) dt, \quad x\in(a,b),
\end{align}
is a solution of $\tau u = g$ by Lemma \ref{prop:repsol} and has compact support since $g$ lies in 
the kernels of $\ell_1$ and $\ell_2$, hence $u\in\dom{\Tpre}$. 
Then the Lagrange identity and the property of $(f_1, f_2)$ yield 
\begin{align}
\begin{split}
 \int_a^b \ol{[f_1(x) - f(x)]} \tau u(x) \, r(x) dx & 
 = \spr{\tau u}{f_1}_{r} - \int_a^b \ol{f(x)} \tau u(x) \, r(x) dx \\
   & = \spr{u}{f_2}_{r} - \int_a^b \tau \ol{f(x)} u(x) \, r(x) dx = 0,
\end{split}
\end{align}
hence $g=\tau u\in\ker (\ell)$. 
Now applying Lemma \ref{lem:hilfslemmaadjoint} there are $c_1$, $c_2\in\C$ such that
\begin{align}\label{eqn:ftildfsol}
 \int_a^b \ol{[f_1(x)- f(x) + c_1 u_1(x) + c_2 u_2(x)]} g(x) \, r(x) dx = 0, 
\end{align}
for each $g\in L^2_c((a,b);r(x)dx)$. 
Hence, obviously $f_1\in\Deftau$ and $\tau f_1 = \tau f = f_2$, that is, $f_1\in\dom{\Tmax}$ and $\Tmax f_1 = f_2$.
But this shows that $\wti \Tpre^\ast$ actually is the graph of $\Tmax$, which shows that $\Tpre$ is densely defined with adjoint $\Tmax$.  Indeed, if $T_0$ were not densely defined, there would exist $0\neq h\in L^2((a,b);r(x)dx)\cap (\dom {T_0})^{\perp}$.  Consequently, if $(f_1,f_2)\in \wti \Tpre^\ast$, then $(f_1,f_2+h)\in \wti \Tpre^\ast$, contradicting the fact that $\wti \Tpre^\ast$ is the graph of an operator.
\end{proof}

The operator $\Tpre$ is symmetric by the preceding theorem. The closure $\Tmin$ of $\Tpre$ is called the minimal operator,
\begin{align}
 \Tmin = \overline{\Tpre} = \Tpre^{\ast\ast} = \Tmax^\ast.
\end{align}
In order to determine $\Tmin$ we need the following lemma on functions in $\dom{\Tmax}$.

\begin{lemma}\label{lem:funcdomtmax}
 If $f_a$ lies in $\dom{\Tmax}$ near $a$ and $f_b$ lies in $\dom{\Tmax}$ near $b$, 
 then there exists an $f\in\dom{\Tmax}$ such that $f=f_a$ near $a$ and $f=f_b$ near $b$.
\end{lemma}
\begin{proof}
Let $u_1$, $u_2$ be a fundamental system of $\tau u = 0$ with $W(u_1,u_2)=1$ and let 
$\alpha, \beta\in(a,b)$, $\alpha<\beta$ such that the functionals
\begin{align}
 F_j(g) = \int_\alpha^\beta u_j(x) g(x) \, r(x) dx, \quad g\in\Lr,~j=1,2,
\end{align}
are linearly independent.
First we will show that there is some $u\in\Deftau$ such that
\begin{align}
 u(\alpha) = f_a(\alpha), \quad u^\qd(\alpha) = f_a^\qd(\alpha), \quad u(\beta) = f_b(\beta), 
 \quad u^\qd(\beta) = f_b^\qd(\beta).
\end{align}
Indeed, let $g\in\Lr$ and consider the solution $u$ of $\tau u=g$ with initial conditions
\begin{align}
 u(\alpha) = f_a(\alpha) \, \text{ and }\, u^\qd(\alpha)=f_a^\qd(\alpha).
\end{align}
With Lemma \ref{prop:repsol} one sees that $u$ has the desired properties if 
\be
      \begin{pmatrix} F_2(g) \\ F_1(g) \end{pmatrix} 
     =  \begin{pmatrix} u_1(\beta) & -u_2(\beta) \\ u_1^\qd(\beta) & -u_2^\qd(\beta) \end{pmatrix}^{-1}
       \begin{pmatrix}
        f_b(\beta) - c_1 u_1(\beta) - c_2 u_2(\beta) \\ f_b^\qd(\beta) - c_1u_1^\qd(\beta) - c_2 u_2^\qd(\beta)
       \end{pmatrix},
\ee
where $c_1$, $c_2\in\C$ are the constants appearing in Lemma \ref{prop:repsol}.
But since the functionals $F_1$, $F_2$ are linearly independent, we may choose $g\in\Lr$ such that this equation is valid.
Now the function $f$ defined by
\be
 f(x) = \begin{cases}
          f_a(x),   & x\in(a,\alpha), \\
          u(x),     & x\in(\alpha,\beta), \\
          f_b(x),   & x\in(\beta,b),
        \end{cases}
\ee
has the claimed properties.
\end{proof}

\begin{theorem}\label{thm:Tmin}
 The minimal operator $\Tmin$ is given by
\begin{align}
\begin{split}
& \Tmin f = \tau f,  \quad  f \in \dom{\Tmin} = \{g\in \dom{\Tmax} \,|\, \forall h\in\dom{\Tmax}:  \\
& \hspace*{6.2cm}  W(g,h)(a) = W(g,h)(b) = 0 \}.    
\end{split}
\end{align}
\end{theorem}
\begin{proof}
 If $f\in\dom{\Tmin}=\dom{\Tmax^\ast}\subseteq\dom{\Tmax}$, then
\begin{align}
 0 = \spr{\tau f}{g}_{r} - \spr{f}{\tau g}_{r} = W(f,\ol{g})(b) - W(f,\ol{g})(a), \quad g\in\dom{\Tmax}.
\end{align}
Given some $g\in\dom{\Tmax}$, there is a $g_a\in\dom{\Tmax}$ such that $\ol{g_a}=g$ in a 
 vicinity of $a$ and $g_a=0$ in a vicinity of $b$. Therefore, $W(f,g)(a) = W(f,\ol{g_a})(a) - W(f,\ol{g_a})(a) = 0$.
 Similarly, one obtains $W(f,g)(b)=0$ for each $g\in\dom{\Tmax}$.

 Conversely, if $f\in\dom{\Tmax}$ such that for each $g\in\dom{\Tmax}$, $W(f,g)(a)=W(f,g)(b)=0$, then
\begin{align}
 \spr{\tau f}{g}_{r} - \spr{f}{\tau g}_{r} = W(f,\ol{g})(b) - W(f,\ol{g})(a) = 0,
\end{align}
 hence $f\in \dom{\Tmax^\ast} = \dom{\Tmin}$.
\end{proof}

For regular $\tau$ on $(a,b)$ we may characterize the minimal operator by the boundary values of the functions $f\in \dom{\Tmax}$ as follows:

\begin{corollary}
 If $\tau$ is regular at $a$ and $f\in \dom{\Tmax}$, then
\begin{align}
 f(a) = f^\qd(a) = 0 \, \text{ iff } \,  \forall g\in \dom{\Tmax}: W(f,g)(a) = 0.
\end{align}
 A similar result holds at $b$.
\end{corollary}
\begin{proof}
 The claim follows from $W(f,g)(a) = f(a)g^\qd(a) - f^\qd(a)g(a)$ and the fact that one finds $g\in \dom{\Tmax}$ with prescribed initial values at $a$.
 Indeed, one can take $g$ to coincide with some solution of $\tau u = 0$ near $a$.
\end{proof}

Next we will show that $\Tmin$ always has self-adjoint extensions.

\begin{theorem}
 The deficiency indices $n(\Tmin)$ of the minimal operator $\Tmin$ are equal and at most two, that is, 
\be
  n(\Tmin) = \dim \big(\ran\big(\left(\Tmin - \I\right)^\bot\big)\big) 
  = \dim \big(\ran \big(\left(\Tmin + \I\right)^\bot\big)\big) \leq 2.
\ee
\end{theorem}
\begin{proof}
 The fact that the dimensions are less than two follows from 
 \be
  \ran\big((\Tmin \pm \I)^\bot\big) = \ker((\Tmax \mp \I)),
 \ee
 because there are at most two linearly independent solutions of $(\tau\pm\I)u=0$.
 Moreover, equality is due to the fact that $\Tmin$ is real with respect to the natural conjugation in $\Lr$.
\end{proof}

\section{Weyl's Alternative}  \lb{s4} 

We say $\tau$ is in the limit-circle (l.c.) case at $a$, if for each $z\in\C$ all solutions of $(\tau-z)u=0$ 
lie in $\Lr$ near $a$. Furthermore, we say $\tau$ is in the limit-point (l.p.) case at $a$ if for each $z\in\C$ 
there is some solution of $(\tau-z)u=0$ which does not lie in $\Lr$ near $a$.
Similarly, one defines the l.c.\ and l.p.\ cases at the endpoint $b$.
It is clear that $\tau$ is only either in the l.c.\ or in the l.p.\ case at some boundary point.
The next lemma shows that $\tau$ indeed is in one of these cases at each endpoint, which is known 
as Weyl's alternative. 
 
\begin{lemma}\label{lemweylaltLC}
 If there is a $z_0\in\C$ such that all solutions of $(\tau-z_0) u = 0$ lie in $\Lr$ near $a$, 
 then $\tau$ is in the l.c.\ case at $a$.
 A similar result holds at the endpoint $b$.
\end{lemma}
 \begin{proof}
 Let $z\in\C$ and $u$ be a solution of $(\tau-z)u = 0$. 
 If $u_1$, $u_2$ are a fundamental system of $(\tau-z_0)u=0$ with $W(u_1,u_2)=1$, 
 then $u_1$ and $u_2$ lie in $\Lr$ near $a$ by assumption.
 Therefore, there is some $c\in(a,b)$ such that the function $v=|u_1| + |u_2|$ satisfies
\be
 |z-z_0| \int_a^c v(t)^2 \, r(t) dt \leq \frac{1}{2}.
\ee
 Since $u$ is a solution of $(\tau-z_0)u = (z-z_0)u$ we have for each $x\in(a,b)$,
\be
 u(x) = c_1u_1(x) + c_2u_2(x) + (z-z_0) \int_c^x \left( u_1(x) u_2(t) - u_1(t)u_2(x)\right) u(t)r(t)dt,
\ee
for some $c_1$, $c_2\in\C$ by Lemma \ref{prop:repsol}. Hence, with $C= \max(|c_1|,|c_2|)$, one estimates
\be
 |u(x)| \leq C v(x) + |z-z_0| v(x) \int_x^c v(t) |u(t)| \, r(t) dt, \quad x\in(a,c),
\ee
and furthermore, using Cauchy--Schwarz,
\be
 |u(x)|^2 \leq 2 C^2 v(x)^2 + 2 |z-z_0|^2 v(x)^2 \int_x^c v(t)^2 \, r(t) dt \int_x^c |u(t)|^2 \, r(t) dt.
\ee
Now an integration yields for each $s\in(a,c)$,
\begin{align}
& \int_s^c |u(t)|^2 r(t)dt  \nonumber\\
&\quad \leq 2C^2 \int_a^c v(t)^2 r(t)dt + 2|z-z_0|^2 \left(\int_a^c v(t)^2 r(t)dt \right)^2 \int_s^c |u(t)|^2 r(t)dt 
\no \\
&\quad \leq 2C^2 \int_a^c v(t)^2 r(t)dt + \frac{1}{2} \int_s^c |u(t)|^2 r(t)dt,
\end{align}
 and therefore,
\be
 \int_s^c |u(t)|^2 r(t)dt \leq 4C^2 \int_a^c v(t)^2 \, r(t) dt < \infty.
\ee
Since $s\in(a,c)$ was arbitrary, this yields the claim.
 \end{proof}

 In particular, if $\tau$ is regular at an endpoint, then $\tau$ is in the l.c.\ case there since each 
 solution of $(\tau-z)u=0$ has a continuous extension to this endpoint.

With $\reg(\Tmin)$ we denote the set of all points of regular type of $\Tmin$, that is, all $z\in\C$ such that
$(\Tmin-z)^{-1}$ is a bounded operator (not necessarily everywhere defined). Recall that $\dim \ran(\Tmin - z)^\bot$
is constant on every connected component of $\reg(\Tmin)$ (\cite[Theorem\ 8.1]{We80}) and thus
$\dim\big(\ran\big((\Tmin - z)^\bot\big)\big) = \dim(\ker(\Tmax  - \ol{z})) = n(\Tmin)$ for every 
$z\in \reg(\Tmin)$.

\begin{lemma}\label{lemWeylRegType}
 For each $z\in \reg(\Tmin)$ there is a nontrivial solution of $(\tau-z)u=0$ which lies in $\Lr$ near $a$.
 A similar result holds at the endpoint $b$.
\end{lemma}
\begin{proof}
 First assume that $\tau$ is regular at $b$. If there were no solution of $(\tau-z)u=0$ which lies in $\Lr$
 near $a$, we would have $\ker(\Tmax-z) = \lbrace 0\rbrace$ and hence $n(\Tmin)=0$, that is, 
 $\Tmin=\Tmax$.
 But since there is an $f\in\dom{\Tmax}$ with
 \begin{align}
  f(b) = 1 \, \text{ and } \, f^\qd(b) = 0,
 \end{align}
 this is a contradiction to Theorem \ref{thm:Tmin}.
 
 For the general case pick some $c\in(a,b)$ and consider the minimal operator $T_c$ in $L^2((a,c);r(x)dx)$ 
 induced by $\tau|_{(a,c)}$. Then $z$ is a point of regular type of $T_c$.
 Indeed, we can extend each $f_c\in\dom{T_c}$ with zero and obtain a function $f\in\dom{\Tmin}$.
 For these functions and some positive constant $C$, 
 \begin{align}
  \left\| (T_c - z)f_c \right\|_{L^2((a,c);r(x)dx)} = \left\| (\Tmin-z)f \right\|_{2,r} \geq C \left\| f\right\|_{2,r} 
  = C\left\| f_c\right\|_{L^2((a,c);r(x)dx)}. 
 \end{align}
Now since the solutions of $(\tau|_{(a,c)}-z)u=0$ are exactly the solutions of
 $(\tau-z)u=0$ restricted to $(a,c)$, the claim follows from what we already proved.
\end{proof}

\begin{corollary}\label{cor:regtypeuniqsol}
 If $z\in \reg(\Tmin)$ and $\tau$ is in the l.p.\ case at $a$, then there is a unique
 nontrivial solution of $(\tau-z)u=0$ $($up to scalar multiples\,$)$, which lies in $\Lr$ near $a$. A similar result holds at the endpoint $b$.
\end{corollary}
\begin{proof}
 If there were two linearly independent solutions in $\Lr$ near $a$, $\tau$ would be l.c.\ at $a$.
\end{proof}

\begin{lemma}\label{lem:lclpwronski}
 $\tau$ is in the l.p.\ case at $a$ if and only if
\be
 W(f,g)(a) = 0, \quad f,\,g\in \dom{\Tmax}.
\ee
$\tau$ is in the l.c.\ case at $a$ if and only if there is a $f\in \dom{\Tmax}$ such that
\be
 W(f,\ol{f})(a) = 0 \, \text{ and } \, W(f,g)(a)\not=0 \, \text{ for some } \, g\in \dom{\Tmax}.
\ee
Similar results hold at the endpoint $b$.
\end{lemma}
\begin{proof}
 Let $\tau$ be in the l.c.\ case at $a$ and $u_1$, $u_2$ be a real fundamental system of $\tau u=0$ with $W(u_1,u_2)=1$.
 Both, $u_1$ and $u_2$ lie in $\dom{\Tmax}$ near $a$. Hence, there are $f$, $g\in\dom{\Tmax}$ with $f=u_1$ and $g=u_2$ near $a$
 and $f=g=0$ near $b$. Consequently, we obtain
 \begin{align}
 W(f,g)(a) = W(u_1,u_2)(a) = 1 \, \text{ and } \, W(f,\ol{f})(a) = W(u_1,\ol{u_1})(a) = 0,
 \end{align}
 since $u_1$ is real.

Now assume $\tau$ is in the l.p.\ case at $a$ and regular at $b$. Then $\dom{\Tmax}$ is a 
two-dimensional extension of $\dom{\Tmin}$, since $\dim(\ker(\Tmax-\I)) = 1$ by 
Corollary \ref{cor:regtypeuniqsol}. Let $v$, $w\in\dom{\Tmax}$ with $v=w=0$ in a vicinity of 
$a$ and
 \be
  v(b) = w^\qd(b) = 1 \, \text{ and } \, v^\qd(b) = w(b) = 0.
 \ee
 Then
 \begin{align}
  \dom{\Tmax} = \dom{\Tmin} + \linspan\lbrace v,w\rbrace,
 \end{align}
 since $v$ and $w$ are linearly independent modulo $\dom{\Tmin}$ and they do not lie in $\dom{\Tmin}$. Then for each $f$, $g\in\dom{\Tmax}$ there are
 $f_0$, $g_0\in\dom{\Tmin}$ such that $f=f_0$ and $g=g_0$ in a vicinity of $a$ and therefore, 
 \be
  W(f,g)(a) = W(f_0,g_0)(a) = 0.
 \ee
 Now if $\tau$ is not regular at $b$ we pick some $c\in(a,b)$. Then for each $f\in\dom{\Tmax}$, 
 $f|_{(a,c)}$ lies in the domain of the maximal operator induced by $\tau|_{(a,c)}$ and the claim 
 follows from what we already proved.
\end{proof}

\begin{lemma}\label{lem:lplpnosol}
 Let $\tau$ be in the l.p.\ case at both endpoints and $z\in\C\backslash\R$. Then there is no nontrivial 
 solution of $(\tau-z)u=0$ in $\Lr$.
\end{lemma}
\begin{proof}
 If $u\in\Lr$ is a solution of $(\tau-z)v=0$, then $\ol{u}$ is a solution of $(\tau-\ol{z})w=0$ 
 and both $u$ and $\ol{u}$ lie in $\dom{\Tmax}$. Now the Lagrange identity yields
 \begin{align}
  W(u,\ol{u})(\beta) - W(u,\ol{u})(\alpha) = (z-\ol{z})\int_\alpha^\beta |u(t)|^2 \, r(t) dt = 2 \I \im(z) \int_\alpha^\beta |u(t)|^2 \, r(t) dt.
 \end{align}
 If $\alpha\rightarrow a$ and $\beta\rightarrow b$, the left-hand side converges to zero by Lemma \ref{lem:lclpwronski}
 and the right-hand side converges to $2 \I \, \im(z) \|u\|_{2,r}$, hence $\|u\|_{2,r}=0$.
\end{proof}

\begin{theorem}\label{thm:TminDefIndLCLP}
 The deficiency indices of the minimal operator $\Tmin$ are given by
\begin{align}
 n(\Tmin) = \begin{cases}
              0, & \text{if }\tau\text{ is l.c.\ at no boundary point}, \\
              1, & \text{if }\tau\text{ is l.c.\ at exactly one boundary point}, \\
              2, & \text{if }\tau\text{ is l.c.\ at both boundary points.} 
            \end{cases}
\end{align}
\end{theorem}
\begin{proof}
 If $\tau$ is in the l.c.\ case at both endpoints, all solutions of $(\tau-\I)u=0$ lie in $\Lr$ and hence in $\dom{\Tmax}$.
 Therefore, $n(\Tmin) = \dim(\ker(\Tmax - \I)) = 2$.
 In the case when $\tau$ is in the l.c.\ case at exactly one endpoint, there is (up to scalar multiples) 
 exactly one nontrivial solution of $(\tau-\I)u=0$ in $\Lr$, by Corollary \ref{cor:regtypeuniqsol}.
 Now if $\tau$ is in the l.p.\ case at both endpoints, we have $\ker(\Tmax-\I)=\lbrace 0\rbrace$ by Lemma \ref{lem:lplpnosol} 
 and hence $n(\Tmin)=0$.
\end{proof}

\section{Self-Adjoint Realizations}  \lb{s5} 

We are interested in the self-adjoint restrictions of $\Tmax$ (or equivalently 
the self-adjoint extensions of $\Tmin$).
To this end, recall that  we introduced the convenient short-hand notation
\be
W_a^b(f,g)= W(f,g)(b) -W(f,g)(a), \quad f,\, g\in\dom{\Tmax}.
\ee

\begin{theorem}
 Some operator $S$ is a self-adjoint restriction of $\Tmax$ if and only if
\begin{equation}
S f = \tau f, \quad f\in 
 \dom{S} = \left\lbrace f\in\dom{\Tmax} \,|\, \forall g\in \dom{S}: W_a^b(f,\ol{g}) = 0 \right\rbrace, 
 \label{5.2}
\end{equation}
\end{theorem}
\begin{proof}
 We denote the right-hand side of \eqref{5.2} by $\dom {S_0}$. First assume 
 $S$ is a self-adjoint restriction of $\Tmax$. If $f\in \dom S$ then
 \begin{align}
  0 = \spr{\tau f}{g}_{r} - \spr{f}{\tau g}_{r} = W_a^b(f,\ol{g})
 \end{align}
 for each $g\in \dom S$ so that $f\in \dom {S_0}$.
 Now if $f\in \dom {S_0}$, then 
 \begin{align}
  0 = W_a^b(f,\ol{g}) = \spr{\tau f}{g}_{r} - \spr{f}{\tau g}_{r}
 \end{align}
 for each $g\in \dom S$, hence $f\in \text{dom}\,(S^*)=\dom S$.
 
 Conversely, assume $\dom S=\dom {S_0}$.  Then $S$ is symmetric since $\spr{\tau f}{g}_{r}=\spr{f}{\tau g}_{r}$ for each $f$, $g\in \dom S$.
 Now let $f\in \text{dom}\,(S^\ast)\subseteq\text{dom}\,(\Tmin^\ast) =\text{dom}\,(\Tmax)$, then
 \begin{align}
  0 = \spr{\tau f}{g}_{r} - \spr{f}{\tau g}_{r} = W_a^b(f,\ol{g}),
 \end{align}
 for each $g\in \dom S$.  Hence, $f\in \dom {S_0}=\dom S$, and it follows that $S$ is self-adjoint.
\end{proof}

The aim of this section is to determine all self-adjoint restrictions of $\Tmax$. 
If both endpoints are in the l.p.\ case this is an immediate consequence of Theorem \ref{thm:TminDefIndLCLP}.

\begin{theorem}\lb{t5.2}
 If $\tau$ is in the l.p.\ case at both endpoints then $\Tmin=\Tmax$ is a self-adjoint operator.
\end{theorem}

Next we turn to the case when one endpoint is in the l.c.\ case and the other one is in the l.p.\ case.
But before we do this, we need some more properties of the Wronskian.

\begin{lemma}\label{lem:WronskPropSR}
 Let $v\in\dom{\Tmax}$ such that $W(v,\ol{v})(a) = 0$ and suppose there is an $h\in\dom{\Tmax}$ with
    $W(h,\ol{v})(a)\not=0$. Then for each $f$, $g\in\dom{\Tmax}$ we have
 \begin{align}\label{eqn:WronskLCconj}
  W(f,\ol{v})(a) = 0\ \ \text{if and only if}\ \ W(\ol{f},\ol{v})(a) = 0
 \end{align}
 and
 \begin{align}\label{eqn:WronskLC2}
  W(f,\ol{v})(a) = W(g,\ol{v})(a) = 0 \ \ \text{implies} \ \ W(f,g)(a) = 0.
 \end{align}
 Similar results hold at the endpoint $b$.
\end{lemma}
\begin{proof}
 Choosing $f_1=v$, $f_2=\ol{v}$, $f_3=h$ and $f_4=\ol{h}$ in the Pl\"{u}cker identity, we infer that also $W(h,v)(a)\not=0$.
 Now let $f_1=f$, $f_2=v$, $f_3=\ol{v}$ and $f_4=h$, then the Pl\"{u}cker identity yields \eqref{eqn:WronskLCconj},
 whereas $f_1=f$, $f_2=g$, $f_3=\ol{v}$ and $f_4=h$ yields \eqref{eqn:WronskLC2}.
\end{proof}

\begin{theorem}\label{thm:SRLCLPWronsk}
 Suppose $\tau$ is in the l.c.\ case at $a$ and in the l.p.\ case at $b$. Then some 
 operator $S$ is a self-adjoint restriction of $\Tmax$ if and only if there is a 
 $v\in\dom{\Tmax}\backslash\dom{\Tmin}$ with $W(v,\ol{v})(a)=0$ such that
 \begin{equation}
 S f = \tau f, \quad f \in \dom{S} = \left\lbrace g\in\dom{\Tmax} \,|\, W(g,\ol{v})(a)=0 \right\rbrace. 
 \end{equation}
 A similar result holds if $\tau$ is in the l.c.\ case at $b$ and in the l.p.\ case at $a$.
\end{theorem}
\begin{proof}
 Since $n(\Tmin)=1$, the self-adjoint extensions of $\Tmin$ are precisely the one-dimensional, 
  symmetric extensions of $\Tmin$. Hence some operator $S$ is a self-adjoint extension of $\Tmin$ if and only
  if there is a $v\in\dom{\Tmax}\backslash\dom{\Tmin}$ with $W(v,\ol{v})(a)=0$ such that
 \be
 S f= \tau f, \quad f\in \dom{S} =  \dom{\Tmin} + \linspan\left\lbrace v \right\rbrace.  
 \ee
 Hence, we have to prove that 
 \be
   \dom{\Tmin} + \linspan\left\lbrace v \right\rbrace = \left\lbrace g\in\dom{\Tmax} \,|\, W(g,\ol{v})(a) = 0 \right\rbrace.
 \ee
 The subspace on the left-hand side is included in the right one because of 
  Theorem \ref{thm:Tmin} and $W(v,\ol{v})(a)=0$.
 On the other hand, if the subspace on the right-hand side were larger, then it would coincide with  $\dom{\Tmax}$ and, hence,
  would imply $v\in\dom{\Tmin}$.
\end{proof}

Two self-adjoint restrictions are distinct if and only if the corresponding functions $v$ are linearly independent modulo $\Tmin$.
 Furthermore, $v$ can always be chosen such that $v$ is equal to some real solution of $(\tau-z)u=0$ with $z\in\R$ in some vicinity of $a$.

It remains to consider the case when both endpoints are in the l.c.\ case.

\begin{theorem}\label{thm:SRLCLCWronsk}
 Suppose $\tau$ is in the l.c.\ case at both endpoints. Then some operator $S$ is a self-adjoint restriction of
  $\Tmax$ if and only if there are $v$, $w\in\dom{\Tmax}$, linearly independent modulo $\dom{\Tmin}$, with
 \be\label{eqn:WronskLCLCvw}
  W_a^b(v,\ol{v}) = W_a^b(w,\ol{w}) = W_a^b(v,\ol{w}) = 0
 \ee
 such that
 \begin{equation}
 S f=\tau f, \quad 
f \in \dom{S} = \left\lbrace g\in\dom{\Tmax} \,|\, W_a^b(g,\ol{v}) = W_a^b(g,\ol{w})=0 \right\rbrace.
\end{equation}
\end{theorem}
\begin{proof}
 Since $n(\Tmin)=2$ the self-adjoint restrictions of $\Tmax$ are precisely the two-dimensional, symmetric
  extensions of $\Tmin$. Hence, an operator $S$ is a self-adjoint restriction of $\Tmax$ if and only if there 
  are $v$, $w\in\dom{\Tmax}$, linearly independent modulo $\dom{\Tmin}$, with \eqref{eqn:WronskLCLCvw} such that
 \begin{equation}
S f =\tau f,    \quad  f \in \dom{S} =  \dom{\Tmin} + \linspan\left\lbrace v,w \right\rbrace. 
 \end{equation} 
 Therefore, we have to prove that
 \be
    \dom{\Tmin} + \linspan\left\lbrace v,w \right\rbrace 
    = \left\lbrace f\in\dom{\Tmax} \,|\, W_a^b(f,\ol{v}) = W_a^b(f,\ol{w}) = 0 \right\rbrace := \cD.
 \ee
 Indeed, the subspace on the left-hand side is contained in $\cD$ by Theorem \ref{thm:Tmin} and \eqref{eqn:WronskLCLCvw}.
 In order to prove that it is also not larger, consider the linear functionals $F_v$, $F_w$ on $\dom{\Tmax}$ defined by
 \be
  F_v(f) = W_a^b(f,\ol{v}) \, \text{ and } \, F_w(f) = W_a^b(f,\ol{w}) \, \text{ for } \, f\in\dom{\Tmax}.
 \ee
 The intersection of the kernels of these functionals is precisely $\cD$. Furthermore, these functionals are 
  linearly independent. Indeed, assume $c_1$, $c_2\in\C$ and $c_1 F_v + c_2 F_w=0$, then for all $f\in\dom{\Tmax}$, 
 \be
  0 = c_1 F_v(f) + c_2 F_w(f) = c_1 W_a^b(f,\ol{v}) + c_2 W_a^b(f,\ol{w}) = W_a^b(f,c_1\ol{v} + c_2\ol{w}).
 \ee
 However, by Lemma \ref{lem:funcdomtmax} this yields
 \be
  W(f,c_1\ol{v} + c_2\ol{w})(a) = W(f,c_1\ol{v} + c_2\ol{w})(b) = 0
 \ee
 for all $f\in\dom{\Tmax}$ and consequently $c_1\ol{v} + c_2 \ol{w}\in\dom{\Tmin}$. Now since $v$, $w$ are linearly
  independent modulo $\dom{\Tmin}$ we infer that $c_1=c_2=0$ and Lemma \ref{lem:hilfslemmaadjoint} implies  that
 \be
  \ker (F_v) \not\subseteq \ker (F_w) \, \text{ and } \, \ker (F_w) \not\subseteq \ker (F_v).
 \ee
 Hence, there exist $f_v$, $f_w\in\dom{\Tmax}$ such that $W_a^b(f_v,\ol{v})=W_a^b(f_w,\ol{w})=0$, but for which
  $W_a^b(f_v,\ol{w})\not=0$ and $W_a^b(f_w,\ol{v}) \not=0$. Both $f_v$ and $f_w$ do not lie in $\cD$ and are linearly
  independent; hence, $\cD$ is at most a two-dimensional extension of $\dom{\Tmin}$.
\end{proof}

In the case when $\tau$ is in the l.c.\ case at both endpoints, we may divide the self-adjoint restrictions of
 $\Tmax$ into two classes. Indeed, we say some operator $S$ is a self-adjoint restriction of $\Tmax$ with
 {\it separated boundary conditions} if it is of the form
 \begin{equation}
S f=\tau f,   \quad 
f \in \dom{S} = \left\lbrace g\in\dom{\Tmax} \,|\, W(g,\ol{v})(a)=W(g,\ol{w})(b)=0 \right\rbrace, 
 \end{equation}
where $v$, $w\in\dom{\Tmax}$ such that $W(v,\ol{v})(a)=W(w,\ol{w})(b)=0$ but $W(h,\ol{v})(a)\not=0\not=W(h,\ol{w})(b)$ for some $h\in\dom{\Tmax}$.
Conversely, each operator of this form is a self-adjoint restriction of $\Tmax$ by 
Theorem \ref{thm:SRLCLCWronsk} and Lemma \ref{lem:funcdomtmax}.
The remaining self-adjoint restrictions are called self-adjoint restrictions of 
$\Tmax$ with {\it coupled boundary conditions}.

\section{Boundary Conditions}  \lb{s6}

In this section, let $w_1$, $w_2\in\dom{\Tmax}$ with
\begin{align}\label{eqn:ufuncBCa}
 W(w_1,\ol{w_2})(a) = 1 \, \text{ and } \, W(w_1,\ol{w_1})(a) = W(w_2,\ol{w_2})(a) = 0,
\end{align}
if $\tau$ is in the l.c.\ case at $a$ and 
\begin{align}\label{eqn:ufuncBCb}
 W(w_1,\ol{w_2})(b) = 1 \, \text{ and } \, W(w_1,\ol{w_1})(b) = W(w_2,\ol{w_2})(b) = 0,
\end{align}
if $\tau$ is in the l.c.\ case at $b$.
We will describe the self-adjoint restrictions of $\Tmax$ in terms of the linear functionals 
 $\BCa^1$, $\BCa^2$, $\BCb^1$ and $\BCb^2$ on $\dom{\Tmax}$, defined by
\begin{align}
 \BCa^1(f) = W(f,\ol{w_2})(a) \, \text{ and } \, \BCa^2(f) = W(\ol{w_1},f)(a) \, \text{ for } \, f\in\dom{\Tmax},
\end{align}
if $\tau$ is in the l.c.\ case at $a$ and 
\begin{align}
 \BCb^1(f) = W(f,\ol{w_2})(b) \, \text{ and } \, \BCb^2(f) = W(\ol{w_1},f)(b) 
 \, \text{ for } \, f\in\dom{\Tmax},
\end{align}
if $\tau$ is in the l.c.\ case at $b$. 

If $\tau$ is in the l.c.\ case at some endpoint, functions with \eqref{eqn:ufuncBCa} (resp., 
 with \eqref{eqn:ufuncBCb}) always exist. Indeed, one may take them to coincide near the endpoint with some 
 real solutions of $(\tau-z)u=0$ with $W(u_1,u_2)=1$ for some $z\in\R$ and use Lemma \ref{lem:funcdomtmax}.

In the regular case these functionals may take the form of point evaluations of the function
 and its quasi-derivative at the boundary point.

\begin{lemma}\label{prop:PointEvalBC}
 Suppose $\tau$ is regular at $a$. Then there are $w_1$, $w_2\in\dom{\Tmax}$ with 
 \eqref{eqn:ufuncBCa} such that
  the corresponding linear functionals $\BCa^1$ and $\BCa^2$ satisfy
 \begin{equation}
  \BCa^1(f) = f(a) \, \text{ and } \, \BCa^2(f) = f^\qd(a) \, \text{ for } \, f\in\dom{\Tmax}.    \lb{6BC}
 \end{equation}
 The analogous result holds at the endpoint $b$.
\end{lemma}
\begin{proof}
 Take $w_1$, $w_2\in\dom{\Tmax}$ to coincide near $a$ with the real solutions $u_1$, $u_2$ of $\tau u=0$ with 
\begin{align}
 u_1(a) = u_2^\qd(a) = 1 \, \text{ and } \, u_1^\qd(a) = u_2(a) = 0.
\end{align}
\end{proof}

Using the Pl\"{u}cker identity one easily obtains the equality
\begin{align}
 W(f,g)(a) = \BCa^1(f)\BCa^2(g) - \BCa^2(f)\BCa^1(g), \quad f,\,g\in\dom{\Tmax}.
\end{align}
Then for each $v\in\dom{\Tmax}\backslash\dom{\Tmin}$ with $W(v,\ol{v})(a)=0$ and $W(h,\ol{v})(a)\not=0$ for some $h\in\dom{\Tmax}$, one may show that 
 there is a $\varphi_a\in[0,\pi)$ such that
\begin{align}\label{eqnBCrelavtophi}
 W(f,\ol{v})(a) = 0 \, \text{ iff } \, \BCa^1(f) \cos(\varphi_a) - \BCa^2(f)\sin(\varphi_a)=0, \quad f\in\dom{\Tmax}.
\end{align}
Conversely, if some $\varphi_a\in[0,\pi)$ is given, then there exists a $v\in\dom{\Tmax}$, not belonging to $\dom{\Tmin}$, with 
 $W(v,\ol{v})(a) = 0$ and $W(h,\ol{v})(a)\not=0$ for some $h\in\dom{\Tmax}$ such that
\begin{align}\label{eqnBCrelaphitov}
 W(f,\ol{v})(a) = 0 \, \text{ iff } \, \BCa^1(f) \cos(\varphi_a) - \BCa^2(f)\sin(\varphi_a)=0, \quad f\in\dom{\Tmax}.
\end{align}
Using this, Theorem \ref{thm:SRLCLPWronsk} immediately yields the following characterization of the
 self-adjoint restrictions of $\Tmax$ in terms of the boundary functionals.

\begin{theorem}\label{thm:bcLCLP}
 Suppose $\tau$ is in the l.c.\ case at $a$ and in the l.p.\ case at $b$. Then some operator $S$ is a 
  self-adjoint restriction of $\Tmax$ if and only if there is some $\varphi_a\in[0,\pi)$ such that 
 \begin{align}
 \begin{split} 
 & S f = \tau f,   \\
& f \in \dom{S} = \left\lbrace g \in\dom{\Tmax} \left|\, \BCa^1(g) \cos(\varphi_a) 
- \BCa^2(g) \sin(\varphi_a) = 0\right.\right\rbrace. 
\end{split} 
 \end{align}
A similar result holds if $\tau$ is in the l.c.\ case at $b$ and in the l.p.\ case at $a$.
\end{theorem} 

Next we will give a characterization of the self-adjoint restrictions of $\Tmax$ if $\tau$ is in the l.c.\ case at both endpoints.

\begin{theorem}\label{thm:LCLCBC}
 Suppose $\tau$ is in the l.c.\ case at both endpoints. Then some operator $S$ is a self-adjoint restriction 
  of $\Tmax$ if and only if there are matrices $B_a$, $B_b\in\C^{2\times 2}$ with
 \begin{align}\label{eqnBCLCLCmatcond}
  \rang(B_a | B_b) = 2 \, \text{ and } \, B_aJB_a^\ast = B_bJB_b^\ast \, \text{ with } \,  
     J=\begin{pmatrix} 0 & -1 \\ 1 & 0 \end{pmatrix},
 \end{align}
 such that
 \begin{equation}
 S f= \tau f, \quad 
f \in \dom{S} = \left\lbrace g\in\dom{\Tmax} \left|\, B_a\begin{pmatrix} \BCa^1(g) \\ \BCa^2(g) \end{pmatrix} 
    = B_b \begin{pmatrix} \BCb^1(g) \\ \BCb^2(g) \end{pmatrix}\right. \right\rbrace.
 \end{equation}
\end{theorem}
\begin{proof}
If $S$ is a self-adjoint restriction of $\Tmax$, there exist $v$, $w\in\dom{\Tmax}$, linearly independent
 modulo $\dom{\Tmin}$, with 
 \begin{equation} 
 W_a^b(v,\ol{v})=W_a^b(w,\ol{w})=W_a^b(v,\ol{w})=0,
 \end{equation}
 such that 
  \begin{equation} 
 \dom{S} = \left\lbrace f\in\dom{\Tmax} \left|\, W_a^b(f,\ol{v})=W_a^b(f,\ol{w})=0\right.\right\rbrace. 
  \end{equation}
 Let $B_a$, $B_b\in\C^{2\times2}$ be defined by 
  \begin{equation} 
 B_a=\left(\begin{matrix} \BCa^2(\ol{v}) & -\BCa^1(\ol{v}) \\ \BCa^2(\ol{w}) & -\BCa^1(\ol{w}) \end{matrix}\right)\, \text{ and } \,  
   B_b=\left(\begin{matrix} \BCb^2(\ol{v}) & -\BCb^1(\ol{v}) \\ \BCb^2(\ol{w}) & -\BCb^1(\ol{w}) \end{matrix}\right). 
    \end{equation}
 Then a simple computation shows that 
  \begin{equation} 
 B_aJB_a^*=B_bJB_b^* \, \text{ iff } \, W_a^b(v,\ol{v})=W_a^b(w,\ol{w})=W_a^b(v,\ol{w})=0. 
  \end{equation}
 In order to prove $\rang{(B_a|B_b)}=2$, let $c_1$, $c_2\in\C$ and 
  \begin{equation} 
 0=c_1\begin{pmatrix} \BCa^2(\ol{v}) \\ -\BCa^1(\ol{v}) \\ \BCb^2(\ol{v}) \\ -\BCb^1(\ol{v}) \end{pmatrix} + c_2 \begin{pmatrix} \BCa^2(\ol{w}) \\ -\BCa^1(\ol{w}) \\ \BCb^2(\ol{w}) \\ -\BCb^1(\ol{w}) \end{pmatrix}=\begin{pmatrix} \BCa^2(c_1\ol{v}+c_2\ol{w}) \\ -\BCa^1(c_1\ol{v}+c_2\ol{w}) \\ \BCb^2(c_1\ol{v}+c_2\ol{w}) \\ -\BCb^1(c_1\ol{v}+c_2\ol{w}) \end{pmatrix}. 
  \end{equation}
 Hence, the function $c_1\ol{v}+c_2\ol{w}$ lies in the kernel of $\BCa^1$, $\BCa^2$, $\BCb^1$ and $\BCb^2$, and 
 therefore, $W(c_1\ol{v}+c_2\ol{w},f)(a)=0$ and $W(c_1\ol{v}+c_2\ol{w},f)(b)=0$ for each $f\in\dom{\Tmax}$. 
 This means that $c_1\ol{v}+c_2\ol{w}\in\dom{\Tmin}$ and hence $c_1=c_2=0$, since $\ol{v}$, $\ol{w}$ are 
 linearly independent modulo $\dom{\Tmin}$. This proves that $(B_a|B_b)$ has rank two. 
 Furthermore, a calculation yields that for $f\in\dom{\Tmax}$ 
 \begin{equation} 
 W_a^b(f,\ol{v})=W_a^b(f,\ol{w})=0 \, \text{ iff } \, 
 B_a\left(\begin{matrix}\BCa^1(f)\\ \BCa^2(f)\end{matrix}\right)
    =B_b\left(\begin{matrix} \BCb^1(f)\\ \BCb^2(f)\end{matrix}\right), 
 \end{equation}  
which proves that $S$ is given as in the claim.

Conversely, let $B_a$, $B_b\in\C^{2\times2}$ with the claimed properties be given.
Then there are $v$, $w\in\dom{\Tmax}$ such that 
\begin{equation} 
B_a=\left(\begin{matrix} \BCa^2(\ol{v}) & -\BCa^1(\ol{v}) \\ \BCa^2(\ol{w}) & -\BCa^1(\ol{w}) \end{matrix}\right)
   \, \text{ and } \, B_b=\left(\begin{matrix} \BCb^2(\ol{v}) & -\BCb^1(\ol{v}) \\ \BCb^2(\ol{w}) & -\BCb^1(\ol{w}) \end{matrix}\right). 
\end{equation}
In order to prove that $v$ and $w$ are linearly independent modulo $\dom{\Tmin}$, let $c_1$, $c_2\in\C$ 
 and $c_1v+c_2w\in\dom{\Tmin}$, then
\begin{equation} 
0=\begin{pmatrix} \BCa^2(\ol{c_1} \ol{v}+\ol{c_2} \ol{w}) \\ -\BCa^1(\ol{c_1} \ol{v}+\ol{c_2} \ol{w}) \\
       \BCb^2(\ol{c_1} \ol{v}+\ol{c_2} \ol{w}) \\ -\BCb^1(\ol{c_1} \ol{v}+\ol{c_2} \ol{w}) \end{pmatrix}
    = \ol{c_1} \begin{pmatrix} \BCa^2(\ol{v}) \\ -\BCa^1(\ol{v}) \\ \BCb^2(\ol{v}) \\ -\BCb^1(\ol{v}) \end{pmatrix} 
  + \ol{c_2} \begin{pmatrix} \BCa^2(\ol{w}) \\ -\BCa^1(\ol{w}) \\ \BCb^2(\ol{w}) \\ -\BCb^1(\ol{w}) \end{pmatrix}. 
\end{equation}
Now the rows of $(B_a|B_b)$ are linearly independent, hence $c_1=c_2=0$. 
Since again 
\begin{equation} 
B_aJB_a^*=B_bJB_b^* \, \text{ iff } \, W_a^b(v,\ol{v})=W_a^b(w,\ol{w})=W_a^b(v,\ol{w})=0, 
 \end{equation} 
the functions $v$, $w$ satisfy the assumptions of Theorem\ \ref{thm:SRLCLCWronsk}. 
As above, one infers once again that for $f\in\dom{\Tmax}$,
\begin{equation} 
B_a\left(\begin{matrix}\BCa^1(f)\\ \BCa^2(f)\end{matrix}\right)=B_b\left(\begin{matrix} \BCb^1(f)\\ \BCb^2(f)\end{matrix}\right)
   \, \text{ iff } \, W_a^b(f,\ol{w})=W_a^b(f,\ol{w})=0. 
 \end{equation}   
 Hence, $S$ is a self-adjoint restriction of $\Tmax$ by Theorem \ref{thm:SRLCLCWronsk}.
\end{proof}

As in the preceding section, if $\tau$ is in the l.c.\ case at both endpoints, we may divide the self-adjoint 
 restrictions of $\Tmax$ into two classes.
 
\begin{theorem}\label{thm:SRLCLCsepcoup}
Suppose $\tau$ is in the l.c.\ case at both endpoints. 
Then some operator $S$ is a self-adjoint restriction of $\Tmax$ with separated boundary conditions if and only 
 if there are $\varphi_a$, $\varphi_b\in[0,\pi)$ such that
 \begin{align}\label{eqn:SRLCLCsep}
& S f = \tau f,   \\
& f \in \dom{S} = \bigg\{g\in\dom{\Tmax} \, \bigg| \,  
\begin{array}{l} \BCa^1(g)\cos(\varphi_a)-\BCa^2(g)\sin(\varphi_a)=0, \\
\BCb^1(g)\cos(\varphi_b)-\BCb^2(g)\sin(\varphi_b)=0 \end{array}\bigg\}.   \no 
 \end{align} 
 Furthermore, $S$ is a self-adjoint restriction of $\Tmax$ with coupled boundary conditions if and only if
 there are $\phi\in[0,\pi)$ and $R\in\R^{2\times2}$ with $\det{(R)}=1$ $($i.e., $R \in \SL_2(\bbR)$$)$  
 such that
 \begin{align}\label{eqn:SRLCLCcoup}
 \begin{split} 
 & S f = \tau f,     \\
& f \in \dom{S} = \bigg\{g \in\dom{\Tmax} \, \bigg| \, \left(\begin{matrix}\BCb^1(g)\\\BCb^2(g)\end{matrix}\right) 
      = \E^{\I\phi}R\left(\begin{matrix}\BCa^1(g)\\\BCa^2(g)\end{matrix}\right)\bigg\}.    \\
\end{split} 
 \end{align} 
\end{theorem}
\begin{proof}
Using \eqref{eqnBCrelavtophi} and \eqref{eqnBCrelaphitov} one easily sees that the 
self-adjoint restrictions 
 of $\Tmax$ with separated boundary conditions are precisely the ones given in \eqref{eqn:SRLCLCsep}. Hence, we only have to prove the second claim.
Let $S$ be a self-adjoint restriction of $\Tmax$ with coupled boundary conditions and 
 $B_a$, $B_b\in\C^{2\times 2}$ matrices as in Theorem \ref{thm:LCLCBC}.
Then by \eqref{eqnBCLCLCmatcond} either both of them have rank one or both have rank two.
In the first case we have
\begin{align}
 B_a z = c_a^{\top} z w_a \, \text{ and } \, B_b z = c_b^{\top} z w_b
\end{align}
for some $c_a$, $c_b$, $w_a$, $w_b\in\C^2\backslash\lbrace (0,0)\rbrace$.
Since the vectors $w_a$ and $w_b$ are linearly independent (recall that $\rang(B_a|B_b)=2$) one infers that
\begin{align}
 B_a\begin{pmatrix} \BCa^1(f) \\ \BCa^2(f) \end{pmatrix} = B_b \begin{pmatrix} \BCb^1(f) \\ \BCb^2(f) \end{pmatrix} \, \text{ iff } \,  B_a\begin{pmatrix} \BCa^1(f) \\ \BCa^2(f) \end{pmatrix} = B_b \begin{pmatrix} \BCb^1(f) \\ \BCb^2(f) \end{pmatrix}=0.
\end{align}
In particular,
\begin{align}
 B_a J B_a^\ast = B_b J B_b^\ast  \, \text{ iff } \, B_a JB_a^\ast = B_b J B_b^\ast = 0.
\end{align}
Now let $v\in\dom{\Tmax}$ with $\BCa^2(\ol{v})=c_1$ and $\BCa^1(\ol{v})=-c_2$. 
A simple calculation yields 
\begin{align}
\begin{split} 
 0=B_aJB_a^* & = W(w_1,w_2)(a)(\BCa^1(v)\BCa^2(\ol{v})-\BCa^2(v)\BCa^1(\ol{v}))w_a \ol{w_a}^{\top} \\ 
 & = W(w_1,w_2)(a)W(v,\ol{v})(a)w_a \ol{w_a}^{\top}. 
 \end{split} 
\end{align} 
Hence, $W(v,\ol{v})(a)=0$ and since $(\BCa^1(v),\BCa^2(v))=(c_2,c_1)\not=0$, $v\not\in\dom{\Tmin}$. 
Furthermore, for each $f\in\dom{\Tmax}$,
\begin{align}
B_a\begin{pmatrix} \BCa^1(f) \\ \BCa^2(f)\end{pmatrix} & =(\BCa^1(f)\BCa^2(\ol{v}) - \BCa^2(f)\BCa^1(\ol{v}))w_a = W(f,\ol{v})(a)w_a.
\end{align}
Similarly one obtains a function $f\in\dom{\Tmax}\backslash\dom{\Tmin}$ with $W(w,\ol{w})(b)=0$ and
\begin{align}
 B_b\begin{pmatrix} \BCb^1(f) \\ \BCb^2(f)\end{pmatrix}= W(f,\ol{w})(b) w_b, \quad f\in\dom{\Tmax}.
\end{align}
However, this shows that $S$ is a self-adjoint restriction with separated boundary conditions.  Hence, both matrices, $B_a$ and $B_b$, have rank two. If we set $B=B_b^{-1}B_a$, 
then $B=J(B^{-1})^\ast J^\ast$ and therefore, $|\det (B)|=1$; hence, $\det (B) = \E^{2i\phi}$ for 
some $\phi\in[0,\pi)$.
If we set $R=\E^{-i\phi}B$, one infers from the identities
\begin{align}
\begin{split}
 B & = \begin{pmatrix} b_{11} & b_{12} \\ b_{21} & b_{22} \end{pmatrix} 
     = J(B^{-1})^\ast J^\ast = \E^{2\I\phi} \begin{pmatrix} 0 & -1 \\ 1 & 0 \end{pmatrix}
       \begin{pmatrix} \ol{b_{22}} & - \ol{b_{21}} \\ -\ol{b_{12}} & \ol{b_{11}} \end{pmatrix}
       \begin{pmatrix} 0 & 1 \\ -1 & 0 \end{pmatrix} \\
   & = \E^{2i\phi} \begin{pmatrix} \ol{b_{11}} & \ol{b_{12}} \\ \ol{b_{21}} & \ol{b_{22}} \end{pmatrix},
\end{split}
\end{align}
that $R\in\R^{2\times2}$ with $\det (R)=1$.
Now because for each $f\in\dom{\Tmax}$ 
\begin{align}
 B_a\begin{pmatrix} \BCa^1(f) \\ \BCa^2(f) \end{pmatrix} =
   B_b\begin{pmatrix} \BCb^1(f) \\ \BCb^2(f) \end{pmatrix} \, \text{ iff } \,  
   \begin{pmatrix} \BCb^1(f) \\ \BCb^2(f) \end{pmatrix} = \E^{\I\phi} R 
    \begin{pmatrix} \BCa^1(f) \\ \BCa^2(f) \end{pmatrix},
\end{align}
$S$ has the claimed representation.

Conversely, if $S$ is of the form \eqref{eqn:SRLCLCcoup}, then Theorem \ref{thm:LCLCBC} shows that it is a 
  self-adjoint restriction of $\Tmax$. Now if $S$ were a self-adjoint restriction with separated boundary
 conditions, there would exist an $f\in \dom{S}\backslash\dom{\Tmin}$, vanishing in some vicinity of $a$. By the boundary condition 
 we would also have $\BCb^1(f)=\BCb^2(f)=0$, that is, $f\in\dom{\Tmin}$. Hence, $S$ cannot be a self-adjoint restriction with separated boundary conditions.
\end{proof}

We note that the separated self-adjoint extensions described in \eqref{eqn:SRLCLCsep} are always real (that is, commute with the antiunitary operator of complex conjugation, resp., the natural conjugation in $\Lr$). The coupled boundary conditions in \eqref{eqn:SRLCLCcoup} are real if and only if $\phi = 0$ (see also 
\cite[Sect.\ 4.2]{Ze05}).

\section{The Spectrum and the Resolvent} \lb{s7}

In this section we will compute the resolvent $R_z = (S - z I_r)^{-1}$ of a self-adjoint restriction $S$ of $\Tmax$.
First we deal with the case when both endpoints are in the l.c.\ case.

\begin{theorem}\label{thmSResolLCLC}
 Suppose $\tau$ is in the l.c.\ case at both endpoints and $S$ is a self-adjoint restriction of $\Tmax$.
 Then for each $z\in\rho(S)$, the resolvent $R_z$ is an integral operator
 \begin{align}
  R_z g (x) = \int_a^b G_z(x,y) g(y) \, r(y) dy, \quad x\in(a,b),~g\in\Lr,
 \end{align}
 with a  square integrable kernel $G_z$, that is, $R_z$ is a Hilbert-Schmidt operator, 
 $R_z \in \cB_2\big(L^2((a,b); r(x)dx)\big)$. 
 For any two given linearly independent solutions $u_1$, $u_2$ of $(\tau-z)u=0$,
 there are coefficients $m^\pm_{ij}(z)\in\C$, $i$, $j\in\lbrace 1,2\rbrace$, such that the kernel is given by
 \begin{equation}
  G_z(x,y) = \begin{cases}
               \sum_{i,j=1}^2 m^+_{ij}(z) u_i(x) u_j(y), & y\in(a,x], \\
               \sum_{i,j=1}^2 m^-_{ij}(z) u_i(x) u_j(y), & y\in[x,b).
             \end{cases}.     \lb{7.0} 
 \end{equation}
\end{theorem}
\begin{proof}
Let $u_1$, $u_2$ be two linearly independent solutions of $(\tau-z)u=0$ with $W(u_1,u_2)=1$. 
 If $g\in L^2_c((a,b);r(x)dx)$, then $R_z g$ is a solution of $(\tau -z)f=g$ which lies in $\dom{S}$. Hence, from Lemma \ref{prop:repsol} we get for suitable constants $c_1$, $c_2\in\C$
 \begin{align}\label{eqn::resoleqn}
  R_z g(x)=u_1(x)\left(c_1+\int_a^x{u_2(t) g(t) \, r(t) dt}\right)+u_2(x)\left(c_2-\int_a^x{u_1(t) g(t) \, r(t) dt}\right),
 \end{align} 
 for $x\in(a,b)$. Furthermore, since $R_z g$ satisfies the boundary conditions, we obtain
 \begin{align}
  B_a\begin{pmatrix} \BCa^1(R_z g) \\ \BCa^2(R_z g) \end{pmatrix} = B_b\begin{pmatrix} \BCb^1(R_z g) \\ \BCb^2(R_z g) \end{pmatrix},
 \end{align} for some suitable matrices $B_a$, $B_b\in\C^{2\times 2}$ as in Theorem \ref{thm:LCLCBC}.
 Now since $g$ has compact support, we infer that
 \begin{align}
  \begin{pmatrix} \BCa^1(R_z g) \\ \BCa^2(R_z g) \end{pmatrix} & = \begin{pmatrix} c_1\BCa^1(u_1)+c_2\BCa^1(u_2) \\ 
     c_1\BCa^2(u_1)+c_2\BCa^2(u_2) \end{pmatrix} = \begin{pmatrix} \BCa^1(u_1) & \BCa^1(u_2) \\ 
  \BCa^2(u_1) & \BCa^2(u_2) \end{pmatrix}\begin{pmatrix} c_1 \\ c_2 \end{pmatrix}   \no \\
 & = M_\alpha \begin{pmatrix} c_1 \\ c_2 \end{pmatrix},
\end{align} 
 as well as 
 \begin{align} 
 \begin{pmatrix} \BCb^1(R_z g) \\ \BCb^2(R_z g) \end{pmatrix} & = 
   \begin{pmatrix} \left(c_1+\int_a^b{u_2(t) g(t) \, r(t) dt}\right)\BCb^1(u_1) \\ 
     \left(c_1+\int_a^b{u_2(t) g(t) \, r(t) dt}\right)\BCb^2(u_1) \end{pmatrix} \nonumber\\
     &\quad+ 
  \begin{pmatrix} \left(c_2-\int_a^b{u_1(t) g(t) \, r(t) dt}\right)\BCb^1(u_2) \\ 
        \left(c_2-\int_a^b{u_1(t) g(t) \, r(t) dt}\right)\BCb^2(u_2) \end{pmatrix}  \no \\
 & = \begin{pmatrix} \BCb^1(u_1) & \BCb^1(u_2) \\ \BCb^2(u_1) & \BCb^2(u_2) \end{pmatrix}
   \begin{pmatrix} c_1+\int_a^b{u_2(t) g(t) \, r(t) dt} \\ c_2-\int_a^b{u_1(t) g(t) \, r(t) dt} \end{pmatrix} \no \\
 & = M_\beta \begin{pmatrix} c_1 \\ c_2 \end{pmatrix} + M_\beta \begin{pmatrix} \int_a^b{u_2(t) g(t) \, r(t) dt} 
 \\ 
  -\int_a^b{u_1(t) g(t) \, r(t) dt}\end{pmatrix}.
\end{align}
 Consequently, 
 \begin{align}
   \left(B_aM_\alpha -B_b M_\beta\right)\begin{pmatrix} c_1 \\ c_2 \end{pmatrix} = B_b M_\beta \begin{pmatrix} \int_a^b{u_2(t) g(t) \, r(t) dt} \\ -\int_a^b{u_1(t) g(t) \, r(t) dt} \end{pmatrix}.
 \end{align}
 Now if $B_aM_\alpha-B_bM_\beta$ were not invertible, we would have
 \begin{align}
  \begin{pmatrix} d_1 \\ d_2 \end{pmatrix}\in\C^2\backslash\lbrace(0,0)\rbrace \, \text{ with } \, B_aM_\alpha\begin{pmatrix} d_1 \\ d_2 \end{pmatrix}=B_bM_\beta\begin{pmatrix} d_1 \\ d_2 \end{pmatrix},
 \end{align} 
and the function $d_1u_1+d_2u_2$ would be a solution of $(\tau-z)u=0$ satisfying the boundary conditions
  of $S$, and consequently would be an eigenvector with eigenvalue $z$. However, this would contradict $z\in\rho(S)$, and it follows that $B_aM_\alpha-B_bM_\beta$ must be invertible. 
 Since
  \begin{align}
   \begin{pmatrix} c_1 \\ c_2 \end{pmatrix} = \left(B_aM_\alpha -B_b M_\beta\right)^{-1} B_b M_\beta \begin{pmatrix} \int_a^b{u_2(t) g(t) \, r(t) dt} \\ -\int_a^b{u_1(t) g(t) \, r(t) dt} \end{pmatrix},
  \end{align}
  the constants $c_1$ and $c_2$ may be written as linear combinations of 
 \begin{equation} 
 \int_a^b{u_2(t) g(t) \, r(t) dt}\, \text{ and } \, \int_a^b{u_1(t) g(t) \, r(t) dt},
 \end{equation}
 where the coefficients are independent of $g$. 
 Using equation \eqref{eqn::resoleqn} one verifies that $R_z g$ has an integral-representation
 with a function $G_z$ as claimed. The function $G_z$ is square-integrable, since the solutions
 $u_1$ and $u_2$ lie in $\Lr$ by assumption. 
 Finally, since the operator $K_z$ defined
 \begin{equation}
 K_z g(x) = \int_a^b G_z(x,y)g(y) \, r(y)dy, \quad x\in(a,b),~g\in\Lr,
 \end{equation}
 on $\Lr$, and the resolvent $R_z$ are bounded, the claim follows since they coincide on a dense subspace.
\end{proof}

Since the resolvent $R_z$ is compact, in fact, Hilbert--Schmidt, this implies discreteness of the spectrum.

\begin{corollary}\label{corSpecRDis}
 Suppose $\tau$ is in the l.c.\ case at both endpoints and $S$ is a self-adjoint restriction of $\Tmax$. 
 Then $S$ has purely discrete spectrum, that is, $\sigma(S)=\sigdis(S)$. Moreover,
 \begin{align}
  \sum_{\lambda\in\sigma(S)} \frac{1}{1+\lambda^2} < \infty 
  \, \text{ and } \,  
    \dim(\ker(S-\lambda)) \leq 2, \quad \lambda\in\sigma(S).
 \end{align}
\end{corollary}

If $S$ is a self-adjoint restriction of $\Tmax$ with separated boundary conditions or if (at least) one endpoint is in the l.c.\ case, then the resolvent has a simpler form.

\begin{theorem}\label{thm:ressep}
 Suppose $S$ is a self-adjoint restriction of $\Tmax$ $($with separated boundary conditions if $\tau$ is in the l.c.\ at both endpoints$)$ and $z\in\rho(S)$.
 Furthermore, let $u_a$ and $u_b$ be nontrivial solutions of $(\tau-z)u=0$, such that
 \begin{align}
  u_a\, \begin{cases}
         \text{satisfies the boundary condition at }a\text{ if }\tau\text{ is in the l.c.\ case at }a, \\
         \text{lies in }\Lr\text{ near }a\text{ if }\tau\text{ is in the l.p.\ case at }a,
      \end{cases}
 \end{align}
 and
 \begin{align}
  u_b\, \begin{cases}
         \text{satisfies the boundary condition at }b\text{ if }\tau\text{ is in the l.c.\ case at }b, \\
         \text{lies in }\Lr\text{ near }b\text{ if }\tau\text{ is in the l.p.\ case at }b.
      \end{cases}
 \end{align}
 Then the resolvent $R_z$ is given by
 \begin{align}\label{eq:ressbc}
  R_z g(x) 
           & = \int_{a}^b G_z(x,y) g(y) \, r(y) dy, \quad x\in(a,b),~g\in\Lr,
 \end{align}
 where
 \begin{equation}
  G_z(x,y) = \frac{1}{W(u_b,u_a)}\begin{cases}
               u_a(y) u_b(x), & y\in(a,x], \\
               u_a(x) u_b(y), & y\in[x,b). \\
             \end{cases}       \lb{7.4} 
 \end{equation}
\end{theorem}
\begin{proof}
The functions $u_a$, $u_b$ are linearly independent; otherwise, they would be eigenvectors of $S$ 
 with eigenvalue $z$. Hence, they form a fundamental system of $(\tau-z)u=0$.
 Now for each $f\in\Lr$ we define a function $f_g$ by 
 \begin{align}
&  f_g(x)=W(u_b,u_a)^{-1}\left(u_b(x)\int_a^xu_a(t) g(t) \, r(t) dt+u_a(x)\int_x^bu_b(t) g(t) \, r(t) dt\right), \nonumber\\
 & \hspace*{9.5cm} x\in(a,b). 
 \end{align}
 If $f\in L^2_c((a,b);r(x)dx)$, then $f_g$ is a solution of $(\tau-z)f=g$ by Lemma \ref{prop:repsol}.
 Moreover, $f_g$ is a scalar multiple of $u_a$ near $a$ and a scalar multiple of $u_b$ near $b$. 
 Hence, the function $f_g$ satisfies the boundary conditions of $S$
  and therefore, $R_z g = f_g$.
 Now if $g\in\Lr$ is arbitrary and $g_n\in L^2_c((a,b);r(x)dx)$ is a sequence with $g_n\rightarrow g$ as $n\rightarrow\infty$,
  we obtain $R_z g_n\rightarrow R_z g$ since the resolvent is bounded. 
  Furthermore, $f_{g_n}$ converges pointwise to $f_g$, hence $R_z g=f_g$. 
\end{proof}

If $\tau$ is in the l.p.\ case at some endpoint, then Corollary \ref{cor:regtypeuniqsol} shows that
 there is always a, unique up to scalar multiples, nontrivial solution of $(\tau-z)u=0$, 
 lying in $\Lr$ near this endpoint. 
 Also if $\tau$ is in the l.c.\ case at some endpoint, there exists a, unique up to scalar multiples,
 nontrivial solution of $(\tau-z)u=0$, satisfying the boundary condition at this endpoint. 
 Hence, functions $u_a$ and $u_b$, as in Theorem \ref{thm:ressep} always exist.

\begin{corollary}\label{corSpecRSimple}
 If $S$ is a self-adjoint restriction of $\Tmax$ $($with separated boundary conditions if $\tau$ is in the l.c.\ at both endpoints\,$)$, then all eigenvalues of $S$ are simple.
\end{corollary}
\begin{proof}
 Suppose $\lambda\in\R$ is an eigenvalue and $u_i\in \dom{S}$ with $\tau u_i=\lambda u_i$ for $i=1,2$, that is, they are solutions of $(\tau-\lambda)u=0$.
  If $\tau$ is in the l.p.\ case at some endpoint, then clearly the Wronskian $W(u_1,u_2)$ vanishes.
  Otherwise, since both functions satisfy the same boundary conditions this follows using the Pl\"{u}cker identity.
\end{proof}

Since the deficiency index of $\Tmin$ is finite, the essential spectrum of self-adjoint realizations is 
independent of the boundary conditions, that is, all self-adjoint restrictions of $\Tmax$ have the 
same essential spectrum (cf., e.g., \cite[Theorem\ 8.18]{We80})
We conclude this section by proving that the essential spectrum of the self-adjoint restrictions 
of $\Tmax$ is determined by the behavior of the coefficients in some arbitrarily small neighborhood 
of the endpoints. In order to state this result we need some notation. Fix some $c\in(a,b)$ and 
denote by $\tau|_{(a,c)}$ (resp., by $\tau|_{(c,b)}$) the differential expression on $(a,c)$ (resp., on $(c,b)$) corresponding to our coefficients restricted to $(a,c)$ (resp., to $(c,b)$). Furthermore, let 
 $S_{(a,c)}$ (resp., $S_{(c,b)}$) be some self-adjoint extension of $\tau|_{(a,c)}$ (resp., of $\tau|_{(c,b)}$).

\begin{theorem}
 For each $c\in(a,b)$ we have
 \begin{align}
  \sigess\left(S\right) = \sigess\left(S_{(a,c)}\right) \cup \sigess\left(S_{(c,b)}\right). 
 \end{align}
\end{theorem}
\begin{proof}
 If one identifies $\Lr$ with the orthogonal sum 
 \begin{align}
  \Lr = L^2((a,c);r(x)dx) \oplus L^2((c,b);r(x)dx),
 \end{align}
 then the operator
 \begin{align}
  S_c = S_{(a,c)} \oplus S_{(c,b)},
 \end{align}
 is self-adjoint in $\Lr$.
 Now the claim follows, since $S$ and $S_c$ are both finite dimensional extensions of the symmetric operator given by
 \begin{equation}
T_c f = \tau f, \quad f\in \dom{T_c} = \big\{g\in\dom{\Tmin} \big|\, g(c) = g^\qd(c) = 0 \big\}. 
 \end{equation}
\end{proof}

An immediate corollary is that the essential spectrum only depends on the behavior of the coefficients in some neighborhood of the endpoints, recovering Weyl's splitting method.

\begin{corollary}
 For each $\alpha, \beta\in(a,b)$ with $\alpha<\beta$ we have
 \begin{align}
  \sigess\left(S\right) = \sigess\left(S_{(a,\alpha)}\right) \cup \sigess\left(S_{(\beta,b)}\right). 
 \end{align}
\end{corollary}

\section{The Weyl--Titchmarsh--Kodaira {\em m}-Function}  \lb{s8}

In this section let $S$ be a self-adjoint restriction of $\Tmax$ (with separated boundary conditions if $\tau$ is in the l.c.\ case at both endpoints). 
Our aim is to define a singular Weyl--Titchmarsh--Kodaira function as introduced
 recently in \cite{ET12}, \cite{GZ06}, and \cite{KST12}. To this end we need a real entire fundamental system $\theta_z$, $\phi_z$
 of $(\tau-z)u=0$ with $W(\theta_z,\phi_z)=1$, such that $\phi_z$ lies in $\dom{S}$ near $a$, that is, $\phi_z$ lies in $\Lr$ near $a$ and satisfies the boundary condition at $a$ if $\tau$ is in the l.c.\ case at $a$.

\begin{hypothesis}\label{hypREFS}
 There is a real entire fundamental system $\theta_z$, $\phi_z$ of $(\tau-z)u=0$ with $W(\theta_z,\phi_z)=1$, 
  such that $\phi_z$ lies in $\dom{S}$ near $a$.
\end{hypothesis}

Under the assumption of Hypothesis \ref{hypREFS} we
may define a function \,$m:\rho(S)\rightarrow\C$\, by requiring that the solutions 
\begin{equation}
\psi_z = \theta_z + m(z)\phi_z, \quad z\in\rho(S), 
\end{equation}
 lie in $\dom{S}$ near $b$, that is, they lie in $\Lr$ near $b$ and satisfy the boundary condition at $b$, if $\tau$ is 
 in the l.c.\ case at $b$.
 This function $m$ is well-defined (use Corollary \ref{cor:regtypeuniqsol} if $\tau$ is in the l.p.\ case at $b$)
  and called the singular Weyl--Titchmarsh--Kodaira function of $S$. 
The solutions $\psi_z$, $z\in\rho(S)$, are called the Weyl solutions of $S$. 

\begin{theorem}\label{thmweyltitchManal}
 The singular Weyl--Titchmarsh--Kodaira function $m$ is analytic on $\rho(S)$ and satisfies 
 \begin{align}\label{eqprop::mpsi} 
  m(z)= \ol{m(\ol{z})}, \quad z\in\rho(S).
 \end{align}
\end{theorem}
\begin{proof}
 Let $c$, $d\in(a,b)$ with $c<d$. From Theorem \ref{thm:ressep} and the equation
  \begin{equation}
 W(\psi_z,\phi_z)=W(\theta_z,\phi_z)+m(z)W(\phi_z,\phi_z)=1,\quad z\in\rho(S), 
  \end{equation}
 we obtain for each 
 $z\in\rho(S)$ and $x\in[c,d)$,
 \begin{align} 
  R_z\chi_{[c,d)}(x) & =\psi_z(x)\int_{c}^x \phi_z(y) \, r(y) dy + \phi_z(x)\int_x^d \psi_z(y) \, r(y) dy  \no \\
 & = (\theta_z(x)+m(z)\phi_z(x))\int_c^x \phi_z(y) \, r(y) dy    \no \\
& \quad + \phi_z(x)\int_x^d \left[\theta_z(y) +m(z)\phi_z(y)\right] r(y)dy   \no \\
 & = m(z)\phi_z(x)\int_c^d \phi_z(y) \, r(y) dy + \int_c^d \widetilde{G}_z(x,y) \, r(y) dy, 
 \end{align} 
 where 
  \begin{equation}
 \widetilde{G}_z(x,y)=\begin{cases} \phi_z(y)\theta_z(x), & y \leq x, \\
         \phi_z(x)\theta_z(y), & y \geq x,\end{cases}
  \end{equation}        
 and hence
 \begin{align} 
 \spr{R_z\chi_{[c,d)}}{\chi_{[c,d)}}_{r} & = m(z)\left(\int_c^d \phi_z(y) r(y) dy\right)^2 + \int_c^d\int_c^d \widetilde{G}_z(x,y) r(y) dy\, r(x)dx. 
 \end{align}
 The left-hand side of this equation is analytic in $\rho(S)$ since the resolvent is.
 Furthermore, the integrals are analytic in $\rho(S)$ as well, since the integrands are analytic and locally 
  bounded by Theorem \ref{thmMSLEanaly}.
 Hence, $m$ is analytic if for each $z_0\in\rho(S)$, there exist $c$, $d\in(a,b)$ such that 
  \begin{equation}
 \int_c^d \phi_{z_0}(y) \, r(y) dy\not=0. 
  \end{equation}
However, this holds; otherwise, $\phi_{z_0}$ would vanish almost everywhere. 
 Moreover, equation \eqref{eqprop::mpsi} is valid since the function 
  \begin{equation}
 \theta_{\ol{z}}+ \ol{m(z)} \phi_{\ol{z}}= \ol{\left[\theta_z+m(z)\phi_z\right]}, 
  \end{equation}
 lies in $\dom{S}$ near $b$ 
  by Lemma \ref{lem:WronskPropSR}.
\end{proof}

As an immediate consequence of Theorem \ref{thmweyltitchManal} one infers that $\psi_z(x)$ 
 and $\psi_z^\qd(x)$ are analytic functions in $z\in\rho(S)$ for each $x\in(a,b)$.

\begin{remark}\label{remWTtilde}
 Suppose $\tilde{\theta}_z$, $\tilde{\phi}_z$ is some other real entire fundamental system of $(\tau-z)u=0$ with
  $W(\tilde{\theta}_z,\tilde{\phi}_z)=1$, such that $\tilde{\phi}_z$ lies in $S$ near $a$. Then 
 \begin{align}
  \tilde{\theta}_z = \E^{-g(z)} \theta_z - f(z)\phi_z, \, \text{ and } \, \tilde{\phi}_z = \E^{g(z)} \phi_z, \quad z\in\C,
 \end{align}
 for some entire functions $f$, $g$ with $f(z)$ real and $g(z)$ real modulo $\I\pi$. The corresponding singular 
 Weyl--Titchmarsh--Kodaira functions are related via
 \begin{align}
  \widetilde{m}(z) = \E^{-2g(z)} m(z) + \E^{-g(z)}f(z), \quad z\in\rho(S).
 \end{align}
 In particular, the maximal domain of holomorphy or the structure of poles and singularities do not change.
\end{remark}
We continue with the construction of a real entire fundamental system in the case when $\tau$ is in the l.c.\ case at $a$.

\begin{theorem}\label{thmweyltitchfundsys}
 Suppose $\tau$ is in the l.c.\ case at $a$. Then there exists a real entire fundamental system 
 $\theta_z$, $\phi_z$ of $(\tau-z)u=0$ with $W(\theta_z,\phi_z)=1$,
 such that $\phi_z$ lies in $\dom{S}$ near $a$, 
 \begin{align}
  W(\theta_{z_1},\phi_{z_2})(a) = 1 \, \text{ and } \, 
  W(\theta_{z_1},\theta_{z_2})(a) = W(\phi_{z_1},\phi_{z_2})(a)=0, \quad z_1,\,z_2\in\C.
 \end{align}
\end{theorem}
\begin{proof}
 Let $\theta$, $\phi$ be a real fundamental system of $\tau u=0$ with $W(\theta,\phi)=1$ such that $\phi$ lies in $\dom{S}$ near $a$.
 Now fix some $c\in(a,b)$ and for each $z\in\C$ let $u_{z,1}$, $u_{z,2}$ be the fundamental system of
 \begin{align}
  (\tau-z)u=0 \, \text{ with } \, u_{z,1}(c)=u_{z,2}^\qd(c) = 1 \, \text{ and } \, u_{z,1}^\qd(c)=u_{z,2}(c)=0.
 \end{align}
 Then by the existence and uniqueness theorem we have $u_{\ol{z},i}= \ol{u_{z,i}}$, $i=1,2$. If we introduce
 \begin{align}
  \theta_z(x) & = W(u_{z,1},\theta)(a) u_{z,2}(x) - W(u_{z,2},\theta)(a) u_{z,1}(x), & x\in(a,b), \\
  \phi_z(x)   & = W(u_{z,1},\phi)(a) u_{z,2}(x)   - W(u_{z,2},\phi)(a) u_{z,1}(x),   & x\in(a,b),
 \end{align}
 then the functions $\phi_z$ lie in $\dom{S}$ near $a$ since
 \begin{align}
  W(\phi_z,\phi)(a) & = W(u_{z,1},\phi)(a)W(u_{z,2},\phi)(a) - W(u_{z,2},\phi)(a) W(u_{z_1},\phi)(a) = 0.
 \end{align}
 Furthermore, a direct calculation shows that $\theta_{\ol{z}}=\ol{\theta_{z}}$ and 
 $\phi_{\ol{z}}=\ol{\phi_{z}}$.
 The remaining equalities follow upon repeatedly using the Pl\"{u}cker identity. 
 It remains to prove that the functions $W(u_{z,1},\theta)(a)$, $W(u_{z,2},\theta)(a)$,
  $W(u_{z,1},\phi)(a)$ and $W(u_{z,2},\phi)(a)$ are entire in $z$.
 Indeed, by the Lagrange identity
 \begin{align}
  W(u_{z,1},\theta)(a) = W(u_{z,1},\theta)(c) - z \lim_{x\downarrow a} \int_x^c \theta(t) u_{z,1}(t) \, r(t) dt.
 \end{align}
 Now the integral on the right-hand side is analytic by Theorem \ref{thmMSLEanaly} and in order to prove that 
  the limit is also analytic we need to show that the integral is bounded as $x\downarrow a$, locally uniformly
  in $z$. But the proof of Lemma \ref{lemweylaltLC} shows that, for each $z_0\in\C$,
 \begin{align}
  \left| \int_x^c \theta(t) u_{z,1}(t) r(t) \, dt \right|^2 & \leq K \int_a^c \left|\theta(t)\right|^2 \, r(t) dt 
                   \int_a^c \left[\left|u_{z_0,1}(t)\right| +\left|u_{z_0,2}(t)\right|\right]^2 \, r(t) dt,
 \end{align}
 for some constant $K\in\R$ and all $z$ in some neighborhood of $z_0$. Analyticity of the other functions
  is proved similarly.
\end{proof}

If $\tau$ is regular at $a$, then one may even take $\theta_z$, $\phi_z$ to be the solutions of $(\tau-z)u=0$ 
 with the initial values 
  \begin{equation}
\theta_z(a)=\phi_z^\qd(a)=\cos(\varphi_a) \, \text{ and } 
\, -\theta_z^\qd(a)=\phi_z(a)=\sin(\varphi_a), 
 \end{equation}
for some suitable $\varphi_a\in[0,\pi)$.

\begin{corollary}\label{spek::manal}
 Suppose $\tau$ is in the l.c.\ case at $a$ and $\theta_z$, $\phi_z$ is a real entire fundamental system of $(\tau-z)u=0$ as in Theorem \ref{thmweyltitchfundsys}.
 Then the corresponding singular Weyl--Titchmarsh--Kodaira function $m$ is a Nevanlinna--Herglotz function.
\end{corollary} 
\begin{proof}
In order to prove the Nevanlinna--Herglotz property, we show that
 \begin{align}\label{eqnWTHerglotz}
  0<\|\psi_z\|_{2,r}^2=\frac{\im(m(z))}{\im(z)}, \quad z\in\C\backslash\R.
 \end{align}
Indeed, if $z_1$, $z_2\in\rho(S)$, then 
\begin{align} 
 W(\psi_{z_1},\psi_{z_2})(a) & = W(\theta_{z_1},\theta_{z_2})(a)+m(z_2)W(\theta_{z_1},\phi_{z_2})(a) \no \\ 
 & \quad + m(z_1)W(\phi_{z_1},\theta_{z_2})(a)+m(z_1)m(z_2)W(\phi_{z_1},\phi_{z_2})(a)  \no \\ 
 & = m(z_2)-m(z_1).
\end{align}
If $\tau$ is in the l.p.\ case at $b$, then furthermore we have $W(\psi_{z_1},\psi_{z_2})(b)=0$, since clearly $\psi_{z_1}$, $\psi_{z_2}\in\dom{\Tmax}$. 
This also holds if $\tau$ is in the l.c.\ case at $b$, since then $\psi_{z_1}$ and $\psi_{z_2}$ satisfy the same boundary condition at $b$.
Now the Lagrange identity yields 
\begin{align}\begin{split}
(z_1-z_2)\int_a^b \psi_{z_1}(t)\psi_{z_2}(t) \, r(t)dt & =W(\psi_{z_1},\psi_{z_2})(b)-W(\psi_{z_1},\psi_{z_2})(a) \\ & =m(z_1)-m(z_2).
\end{split}\end{align} 
In particular, for $z\in\C\backslash\R$, using $m(\ol{z})= \ol{m(z)}$ as well as $\psi_{\ol{z}}=\theta_{\ol{z}}+m(\ol{z})\phi_{\ol{z}}=\ol{\psi_{z}}$, we obtain
\begin{align} 
||\psi_z||_r^2 & = \int_a^b \psi_z(t) \psi_{\ol{z}}(t) \, r(t)dt=\frac{m(z)-m(\ol{z})}{z-\ol{z}}=\frac{\im(m(z))}{\im(z)}.\end{align} 
Since $\psi_z$ is a nontrivial solution, we furthermore have $0<||\psi_z||_r^2$.
\end{proof}

We conclude this section with a necessary and sufficient condition for Hypothesis \ref{hypREFS} to hold.
 To this end, for each $c\in(a,b)$, let $S^D_{(a,c)}$ be the self-adjoint operator associated to $\tau|_{(a,c)}$ with a Dirichlet boundary condition at $c$ and the same boundary condition as $S$ at $a$. 

\begin{theorem}
 The following items $(i)$--$(iii)$ are equivalent: \\
$(i)$ \hspace*{1mm} Hypothesis \ref{hypREFS}. \\
$(ii)$ \hspace*{.01mm} There is a real entire solution $\phi_z$ of $(\tau-z)u=0$ which lies in $\dom{S}$ near $a$. \\
$(iii)$ The spectrum of $S^D_{(a,c)}$ is purely discrete for some $c\in(a,b)$. 
\end{theorem}
\begin{proof}
 The proof follows the one for Schr\"{o}dinger operators given in 
 \cite[Lemma 2.2 and Lemma 2.4]{KST12} step by step.
\end{proof}

\section{The Spectral Transformation} \lb{s9}

In this section let $S$ be a self-adjoint restriction of $\Tmax$ (with separated boundary conditions if $\tau$ is in the l.c.\ case at both endpoints) as in the preceding section. Furthermore, we assume that there is a real entire fundamental system $\theta_z$, $\phi_z$ of $(\tau-z)u=0$ with $W(\theta_z,\phi_z)=1$ such that $\phi_z$ lies in $\dom{S}$ near $a$. By $m$ we denote the 
 corresponding singular Weyl--Titchmarsh--Kodaira function and by $\psi_z$ the Weyl solutions of $S$.

Recall that by the spectral theorem, for all functions $f$, $g\in\Lr$ there is a unique complex measure $E_{f,g}$ such that
 \begin{align}
  \spr{R_z f}{g}_{r} = \int_\R \frac{1}{\lambda-z} \, dE_{f,g}(\lambda), \quad z\in\rho(S).
 \end{align}
In order to obtain a spectral transformation we define for each $f\in L^2_c((a,b);r(x)dx)$ the transform of $f$
\begin{align}\label{eqnSThatf}
 \hat{f}(z) = \int_a^b \phi_z(x) f(x) \, r(x) dx, \quad z\in\C.
\end{align}

Next, we will use this to associate a measure with $m(z)$ by virtue of
the Stieltjes--Liv\v{s}i\'{c} inversion formula following literally the proof of \cite[Lemma\ 3.3]{KST12} 
(see also \cite[Theorem\ 2.6]{GZ06}):

\begin{lemma}\label{lemspectransEfgmu}
There is a unique Borel measure $\mu$ defined via
\be\label{defrho}
 \mu((\lambda_1,\lambda_2]) = \lim_{\delta\downarrow 0}\,\lim_{\varepsilon\downarrow 0} \frac{1}{\pi} 
                      \int_{\lambda_1+\delta}^{\lambda_2+\delta} \im(m(\lambda+\I\varepsilon))\, d\lambda,
\ee
 for each $\lambda_1$, $\lambda_2\in\R$ with $\lambda_1<\lambda_2$, such that
\be\label{rhomuf}
dE_{f,g} = \hat{f}\, \ol{\hat{g}} \, d\mu, \quad f,\,g\in L^2_c((a,b);r(x)dx).
\ee
In particular,
\begin{align}
 \spr{R_z f}{g}_{r} = \int_\R \frac{\hat{f}(\lambda) \ol{\hat{g}(\lambda)}}{\lambda-z} \, d\mu(\lambda), 
 \quad z\in\rho(S).
\end{align}
\end{lemma}

In particular, the preceding lemma shows that the mapping $f\mapsto \hat{f}$ is an isometry from $L^2_c((a,b);r(x)dx)$ into $\Lrmu$.
Indeed, for each $f\in L^2_c((a,b);r(x)dx)$ one infers that
\begin{align}
 \| \hat{f}\|_{\mu}^2 = \int_\R \hat{f}(\lambda) \ol{\hat{f}(\lambda)} \, d\mu(\lambda) 
 = \int_\R dE_{f,f} =  \| f\|_{2,r}^2.
\end{align}
Hence, we may extend this mapping uniquely to an isometric linear operator $\mathcal{F}$ from $\Lr$ into $\Lrmu$ by
\begin{align}
 \mathcal{F} f(\lambda) = \lim_{\alpha\downarrow a}\, \lim_{\beta\uparrow b} \int_\alpha^\beta \phi_\lambda(x) f(x) \, r(x) dx, \quad \lambda\in\R,~f\in\Lr,
\end{align}
where the limit on the right-hand side is a limit in the Hilbert space $\Lrmu$.
Using this linear operator $\mathcal{F}$, it is quite easy to extend the result of Lemma \ref{lemspectransEfgmu} to functions $f$, $g\in\Lr$.
In fact, one gets that $dE_{f,g} =  \mathcal{F}f\, \ol{\mathcal{F} g} \, d\mu$, that is, 
\begin{align}
 \spr{R_z f}{g}_{r} = \int_\R \frac{\mathcal{F}f(\lambda) \ol{\mathcal{F}g(\lambda)}}{\lambda-z} \, d\mu(\lambda), \quad z\in\rho(S).
\end{align}
We will see below that $\mathcal{F}$ is not only isometric, but also onto, that is, $\ran(\mathcal{F})=\Lrmu$.
In order to compute the inverse and the adjoint of $\mathcal{F}$, we introduce for each function $g\in L^2_c(\R;d\mu)$ the transform
\begin{align}
 \check{g}(x) = \int_\R \phi_\lambda(x) g(\lambda) \, d\mu(\lambda), \quad x\in(a,b).
\end{align}
For arbitrary $\alpha, \beta\in(a,b)$ with $\alpha<\beta$ we estimate
\begin{align} 
 \int_\alpha^\beta \left|\check{g}(x)\right|^2 \, r(x) dx & = \int_\alpha^\beta \check{g}(x) \int_\R \phi_\lambda(x) \ol{g(\lambda)} \, d\mu(\lambda) \, r(x) dx  \no \\
          & = \int_\R \ol{g(\lambda)}  \int_\alpha^\beta \phi_\lambda(x) \check{g}(x) \,  r(x) dx \, d\mu(\lambda)  
          \no \\
          & \leq \left\| g\right\|_\mu \left\| \mathcal{F}\left(\chi_{[\alpha,\beta)}\check{g}\right)\right\|_\mu  \no \\
          & \leq \left\| g\right\|_\mu \sqrt{\int_\alpha^\beta \left|\check{g}(x)\right|^2 r(x) dx}.
 \end{align}
Hence, $\check{g}$ lies in $\Lr$ with $\|\check{g}\|_{2,r} \leq \|g\|_{2,\mu}$ and we may extend this mapping uniquely to a bounded linear operator $\mathcal{G}$ 
on $\Lrmu$ into $\Lr$.

If $F$ is a Borel measurable function on $\R$, then we denote by $\M_F$ the maximally defined operator of 
 multiplication with $F$ in $\Lrmu$.

\begin{lemma}\label{lemspectransfourtrans}
 The operator $\mathcal{F}$ is unitary with inverse $\mathcal{G}$.
\end{lemma}
\begin{proof}
 First we prove $\mathcal{G}\mathcal{F}f=f$ for each $f\in\Lr$.
 Indeed, if $f$, $g\in L^2_c((a,b);r(x)dx)$, then
 \begin{align} 
  \spr{f}{g}_{r} & = \int_\R dE_{f,g} = \int_\R \hat{f}(\lambda) \ol{\hat{g}(\lambda)} \, d\mu(\lambda) \no \\
        & = \lim_{n\rightarrow\infty} \int_{(-n,n]} \hat{f}(\lambda) \int_a^b \phi_\lambda(x) \ol{g(x)} \, r(x) dx\, d\mu(\lambda) \no \\
        & = \lim_{n\rightarrow\infty} \int_a^b \ol{g(x)} \int_{(-n,n]} \hat{f}(\lambda) \phi_\lambda(x) \, 
        d\mu(\lambda) \, r(x) dx \no \\
        & = \lim_{n\rightarrow\infty} \spr{\mathcal{G}\M_{\chi_{(-n,n]}}\mathcal{F}f}{g}_{r} = \spr{\mathcal{GF}f}{g}_{r}.
\end{align}
 Now since $L^2_c((a,b);r(x)dx)$ is
  dense in $\Lr$ we infer that $\mathcal{GF}f=f$ for all $f\in\Lr$.
In order to prove that $\mathcal{G}$ is the inverse of $\mathcal{F}$, it remains to show that $\mathcal{F}$ 
 is surjective, that is, $\ran(\mathcal{F})=\Lrmu$.
 Therefore, let $f$, $g\in\Lr$ and $F$, $G$ be bounded measurable functions on $\R$.
 Since $E_{f,g}$ is the spectral measure of $S$ we get
 \begin{align}
  \spr{\M_G \mathcal{F} F(S)f}{\mathcal{F}g}_\mu = \spr{G(S)F(S)f}{g}_{r}
     = \spr{\M_G \M_F \mathcal{F}f}{\mathcal{F}g}_\mu.
 \end{align}
 Now if we set $h=F(S)f$, then we obtain from this last equation
 \begin{align}
  \int_\R G(\lambda) \ol{\mathcal{F}g(\lambda)} \big[ \mathcal{F}h(\lambda) - F(\lambda)\mathcal{F}f(\lambda)\big]d\mu(\lambda) = 0. 
 \end{align}
 Since this holds for each bounded measurable function $G$, we infer 
 \begin{align}
 \ol{\mathcal{F}g(\lambda)} \left( \mathcal{F}h(\lambda) - F(\lambda)\mathcal{F}f(\lambda)\right)=0,
 \end{align}
  for almost all $\lambda\in\R$ with respect to $\mu$. Furthermore, for each $\lambda_0\in\R$ we can find
  a $g\in L^2_c((a,b);r(x)dx)$ such that $\hat{g}\not=0$ in a vicinity of $\lambda_0$. 
 Hence, we even have $\mathcal{F}h = F\mathcal{F}f$ 
  almost everywhere with respect to $\mu$. But this shows that $\ran(\mathcal{F})$ contains all characteristic 
  functions of intervals. Indeed, let $\lambda_0\in\R$ and choose $f\in L^2_c((a,b);r(x)dx)$ such that 
  $\hat{f}\not=0$ in a vicinity of $\lambda_0$. Then for each interval $J$, the closure of which is contained in 
  this vicinity, one may choose 
  \begin{align}
   F(\lambda) = \begin{cases} \hat{f}(\lambda)^{-1}, & \text{if }\lambda\in J, \\
                              0, & \text{if }\lambda\in\R\backslash J,
                \end{cases}
  \end{align}
which yields $\chi_J = \mathcal{F} h\in\ran(\mathcal{F})$.
 Thus, $\ran(\mathcal{F})=\Lrmu$ follows.
\end{proof}

\begin{theorem}\label{thmspectransS}
  The self-adjoint operator $S$ is given by $S = \mathcal{F}^\ast \M_{\mathrm{id}} \mathcal{F}$.
\end{theorem}
\begin{proof}
  First note that for each $f\in\Lr$,
   \begin{align}\begin{split}
& f\in\dom{S}  \, \text{ iff } \,  \int_\R |\lambda|^2 dE_{f,f}(\lambda)<\infty \, \text{ iff } \, \int_\R |\lambda|^2 |\mathcal{F}f(\lambda)|^2 d\mu(\lambda) < \infty \\
& \quad \, \text{ iff } \, \mathcal{F}f\in\dom{\M_{\mathrm{id}}} \, \text{ iff } \, f\in\dom{\mathcal{F}^\ast\M_{\mathrm{id}}\mathcal{F}}.
  \end{split}\end{align}
 In this case, Lemma \ref{lemspectransEfgmu} implies
 \begin{align} 
  \spr{S f}{g}_{r} &= \int_\R\lambda dE_{f,g}(\lambda) = \int_\R \lambda \mathcal{F}f(\lambda) 
  \ol{\mathcal{F}g(\lambda)} \, d\mu(\lambda) 
               = \int_\R \M_{\mathrm{id}}\mathcal{F} f(\lambda) \ol{\mathcal{F}g(\lambda)} \, d\mu(\lambda) \no \\
              &= \spr{\mathcal{F}^\ast\M_{\mathrm{id}}\mathcal{F}f}{g}_{r}, \quad g\in\Lr.
\end{align}
  Consequently, $\mathcal{F}^\ast\M_{\mathrm{id}}\mathcal{F}f=S f$.
\end{proof}

Now the spectrum can be read off from the boundary behavior of the 
singular Weyl--Titchmarsh--Kodaira function $m$ in the usual way (see, e.g.,  \cite{Gi89} in the 
classical context and the recent \cite[Corollary 3.5]{KST12}, as well as the references therein).

\begin{corollary}\label{cor:splimm}
 The spectrum of $S$ is given by
\begin{align}
 \sigma(S) & = \supp(\mu) = \overline{\{ \lambda\in\R \,|\, 0 < \limsup_{\varepsilon\downarrow 0} \im(m(\lambda+\I\varepsilon))\}}.
\end{align}
Moreover,
\begin{align}
\sigma_p(S) & = \{ \lambda\in\R \,|\, 0 < \lim_{\varepsilon\downarrow0} \varepsilon \im(m(\lambda+\I\varepsilon)) \}, \\
\sigma_{ac}(S) & = \overline{\{ \lambda\in\R \,|\, 0 < \limsup_{\varepsilon\downarrow0} \im(m(\lambda+\I\varepsilon)) < \infty \}}^{ess},
\end{align}
where $\overline{\Omega}^{ess} = \{ \lambda\in\R \,|\, |(\lambda-\varepsilon,\lambda+\varepsilon)\cap\Omega|>0 \text{ for all }\varepsilon>0 \}$,
is the essential closure of a Borel set $\Omega\subseteq\R$, and
\begin{align}
\Sigma_s = \{ \lambda\in\R \,|\, \limsup_{\varepsilon\downarrow0} \im(m(\lambda+\I\varepsilon)) = \infty \}
\end{align}
is a minimal support for the singular spectrum $($singular continuous plus pure point spectrum\,$)$ of $S$.
\end{corollary}

\begin{lemma}\label{propspectransEVmu}
 If $\lambda\in\sigma(S)$ is an eigenvalue, then
 \begin{align}
  \mu(\lbrace\lambda\rbrace) = \left\| \phi_\lambda\right\|_{2,r}^{-2}.
 \end{align}
\end{lemma}
\begin{proof}
 Under this assumptions $\phi_\lambda$ is an eigenvector of $S$ and 
 $\hat{f}(\lambda) = \spr{f}{\phi_\lambda}_{r}$, $f\in\Lr$. Consequently,
  \begin{align}
   \left\| \phi_\lambda\right\|_{2,r}^2 & = E_{\phi_\lambda,\phi_\lambda}(\lbrace\lambda\rbrace) = \mathcal{F}\phi_\lambda(\lambda) \ol{\mathcal{F}\phi_\lambda(\lambda)} \mu(\lbrace\lambda\rbrace) = \left\|\phi_\lambda\right\|_{2,r}^4 \mu(\lbrace\lambda\rbrace),
  \end{align}
  since $E(\lbrace\lambda\rbrace)$ is the orthogonal projection onto $\phi_\lambda$.
\end{proof}

\begin{lemma}\label{lemspectransGreentransform}
 For every $z\in\rho(S)$ and all $x\in(a,b)$ the transform of the Green's function $G_z(x,\cdot\,)$ and its quasi-derivative $\partial_x^\qd G_z(x,\cdot\,)$ are given by
 \begin{align}
  \mathcal{F} G_z(x,\cdot\,) (\lambda) = \frac{\phi_\lambda(x)}{\lambda-z} \, \text{ and } \,  
      \mathcal{F} \partial_x^\qd G_z(x,\cdot\,) (\lambda) = \frac{\phi_\lambda^\qd(x)}{\lambda-z}, \quad \lambda\in\R.
 \end{align}
\end{lemma}
\begin{proof}
 First note that $G_z(x,\cdot\,)$ and $\partial_x^\qd G_z(x,\cdot\,)$ both lie in $\Lr$.
 Then using Lemma \ref{lemspectransEfgmu}, we get for each $f\in L^2_c((a,b);r(x)dx)$ and $g\in L^2_c(\R;d\mu)$
 \begin{align}
  \spr{R_z \check{g}}{f}_{r} & = \int_\R \frac{g(\lambda) \ol{\hat{f}(\lambda)}}{\lambda-z} \, d\mu(\lambda) 
            = \int_a^b \int_\R \frac{\phi_\lambda(x)}{\lambda-z} g(\lambda) \, d\mu(\lambda)\, \ol{f(x)} \, r(x) dx.
 \end{align}
 Hence,
 \begin{equation}
  R_z \check{g}(x) = \int_\R \frac{\phi_\lambda(x)}{\lambda-z} g(\lambda) \, d\mu(\lambda)
 \end{equation}
 for almost all $x\in(a,b)$. Using Theorem \ref{thm:ressep}, one verifies
 \begin{equation}
  \spr{\mathcal{F} G_z(x,\cdot\,)}{\ol{g}}_\mu = \spr{G_z(x,\cdot\,)}{\ol{\check{g}}}_{r} = \int_\R \frac{\phi_\lambda(x)}{\lambda-z} g(\lambda) \, d\mu(\lambda),
 \end{equation}
 for almost all $x\in(a,b)$. Since all three terms are absolutely continuous, this equality holds for 
 all $x\in(a,b)$, which proves the first part of the claim. The equality for the transform of the quasi-derivative follows from
 \begin{align}
  \spr{\mathcal{F} \partial_x^\qd G_z(x,\cdot\,)}{\ol{g}}_\mu = \spr{\partial_x^\qd G_z(x,\cdot\,)}{\ol{\check{g}}}_{r} = R_z \check{g}^\qd(x) = \int_\R \frac{\phi_\lambda^\qd(x)}{\lambda-z} g(\lambda) \, d\mu(\lambda).
 \end{align}
\end{proof}

\begin{lemma}\label{lemspectranweyltrans}
 Suppose $\tau$ is in the l.c.\ case at $a$ and $\theta_z$, $\phi_z$ is a real entire fundamental system as in Theorem \ref{thmweyltitchfundsys}.
  Then for each $z\in\rho(S)$ the transform of the Weyl solution $\psi_z$ is given by
 \begin{align}
  \mathcal{F} \psi_z(\lambda) = \frac{1}{\lambda-z},\quad \lambda\in\R.
 \end{align}
\end{lemma}
\begin{proof}
 From Lemma \ref{lemspectransGreentransform} we obtain for each $x\in(a,b)$
 \begin{align}
  \mathcal{F} \widetilde{\psi}_z(x,\cdot\,)(\lambda) = \frac{W(\theta_z,\phi_\lambda)(x)}{\lambda-z}, \quad \lambda\in\R,
 \end{align}
 where
 \begin{align}
  \widetilde{\psi}_z(x,y) = \begin{cases}
                         \psi_z(y), & y\geq x, \\
                         m(z) \phi_z(y), & y<x.
                        \end{cases}
 \end{align}
 Now the claim follows by letting $x\downarrow a$, using Theorem \ref{thmweyltitchfundsys}.
\end{proof}

 Under the assumptions of Lemma \ref{lemspectranweyltrans}, $m$ is a Nevanlinna--Herglotz function. Hence,
 \begin{align}\label{eqnSThergMrep}
  m(z) = c_1 + c_2 z + \int_\R 
 \bigg( \frac{1}{\lambda-z} - \frac{\lambda}{1+\lambda^2} \bigg) d\mu(\lambda),\quad z\in\C\backslash\R,
 \end{align}
 where the constants $c_1$, $c_2$ are given by 
\begin{align}
 c_1 = \re(m(\I)) \, \text{ and } \, c_2=\lim_{\eta\uparrow\infty} \frac{m(\I\eta)}{\I\eta}\geq 0.
\end{align}

\begin{corollary}\label{corlinearterm}
 If $\tau$ is in the l.c.\ case at $a$ and $\theta_z$, $\phi_z$ is a real entire fundamental system as in Theorem \ref{thmweyltitchfundsys}, then $c_2 = 0$ in \eqref{eqnSThergMrep}.
\end{corollary}
\begin{proof}
 Taking imaginary parts in \eqref{eqnSThergMrep} yields for each $z\in\C\backslash\R$, 
 \begin{align}
  \im(m(z)) = c_2\im(z) + \int_\R\im\left(\frac{1}{\lambda-z}\right)d\mu(\lambda) = c_2\im(z) + \int_\R \frac{\im(z)}{|\lambda-z|^2} \, d\mu(\lambda).
 \end{align}
 Using the last identity in conjunction with Lemma \ref{lemspectranweyltrans} and \eqref{eqnWTHerglotz}, we obtain
 \begin{align}
  c_2 + \int_\R \frac{1}{|\lambda-z|^2} \, d\mu(\lambda) & = \frac{\im(m(z))}{\im(z)} 
  = \|\psi_z\|_{2,r}^2 = \int_\R \frac{1}{|\lambda-z|^2} \, d\mu(\lambda). 
 \end{align}
\end{proof}

\begin{remark}
Given another singular Weyl--Titchmarsh--Kodaira function $\widetilde{m}$ as in Remark \ref{remWTtilde}, the corresponding spectral measures are related by
\begin{align}
 d\tilde{\mu} = \E^{-2g} d\mu,
\end{align}
where $g$ is the real entire function appearing in Remark \ref{remWTtilde}.
In particular, the measures are mutually absolutely continuous and the associated spectral 
transformations only differ by a simple rescaling with the positive function $\E^{-2g}$. 
\end{remark}

\section{The Spectral Multiplicity} \lb{s10}

In the present section we consider the general case where none of the endpoints are supposed to satisfy the requirements of the previous section. Therefore, let $S$ be a self-adjoint restriction of $\Tmax$ (with separated boundary conditions if $\tau$ is in the l.c.\ case at both endpoints). 
In this situation, the spectral multiplicity of $S$ is potentially two and hence we will work with a matrix-valued spectral transformation. The results in this section extend classical spectral multiplicity results for  second-order Schr\"odinger operators originally due to Kac \cite{Ka62}, \cite{Ka63} (see also Gilbert \cite{Gi98} and Simon \cite{Si05}) to the general situation discussed in this paper. 

We fix some interior point $x_0\in(a,b)$ and consider the real entire fundamental system $\theta_z$, $\phi_z$ of solutions of $(\tau-z)u=0$ with the initial conditions
\begin{align}
  \theta_z(x_0) = \phi^\qd_z(x_0) = \cos(\vphi_a) \,\text{ and }\, -\theta_z^\qd(x_0) = \phi_z(x_0) = \sin(\vphi_a), 
\end{align}
for some fixed $\vphi_a\in[0,\pi)$.
The Weyl solutions are defined by
\begin{align}
 \psi_{z,\pm}(x) = \theta_z(x) \pm m_\pm(z)\phi_z(x), \quad x\in(a,b),~z\in\C\backslash\R,
\end{align} 
such that for all $c \in (a,b)$, 
\begin{equation} 
\psi_{z,-} \in L^2((a,c);r(x)dx) \, \text{ and } \,  \psi_{z,+} \in L^2((c,b);r(x)dx).   
\end{equation} 
Hereby, $m_\pm$ are the regular Weyl--Titchmarsh--Kodaira functions of the operators $S_\pm$ obtained by restricting $S$
to $(a,x_0)$ and $(x_0,b)$ with a boundary condition 
\begin{align}
 f(x_0) \cos(\vphi_a) - f^\qd(x_0) \sin(\vphi_a)=0,
\end{align}
 respectively. 
 One notes that according to Corollary~\ref{spek::manal}, $m_\pm$ are 
 Nevanlinna--Herglotz functions.
One introduces the $2\times 2$ Weyl--Titchmarsh--Kodaira matrix
\begin{align}\label{eq:wmmat}
M(z) = \begin{pmatrix}
         -\frac{1}{m_+(z) + m_-(z)} & \frac{1}{2} \frac{m_-(z) - m_+(z)}{m_+(z) + m_-(z)} \\
         \frac{1}{2} \frac{m_-(z) - m_+(z)}{m_+(z) + m_-(z)} & \frac{m_-(z)m_+(z)}{m_+(z) + m_-(z)}
       \end{pmatrix}, \quad z\in\C\backslash\R, 
\end{align}
and observes that $\det(M(z)) = - 1/4$. Moreover, a brief computation shows that the function $M$ is a matrix-valued Nevanlinna--Herglotz function
and thus has a representation
\begin{align}\label{NHrepWeylM}
 M(z) = C_1 + C_2 z + \int_\R \left(\frac{1}{\lambda-z}-\frac{\lambda}{1+\lambda^2}\right) d\Omega(\lambda), \quad z\in\C\backslash\R,
\end{align}
where $C_1$ is a self-adjoint matrix, $C_2$ a nonnegative matrix, and $\Omega$ is a self-adjoint, matrix-valued measure which is given by the Stieltjes inversion formula
\begin{align}
 \Omega((\lambda_1,\lambda_2]) = \lim_{\delta\downarrow 0} \lim_{\eps\downarrow 0} \frac{1}{\pi} \int_{\lambda_1+\delta}^{\lambda_2+\delta} \im(M(\lambda+\I\eps)) d\lambda, \quad\lambda_1,\lambda_2\in\R, ~\lambda_1<\lambda_2.
\end{align}
It will be shown in Corollary \ref{c10.4} that one actually has $C_2 = 0$ in \eqref{NHrepWeylM}.  
Furthermore, the trace $\Omega^{\mathrm{tr}} = \Omega_{1,1}+\Omega_{2,2}$ of $\Omega$ defines a nonnegative measure and the components of $\Omega$ are absolutely continuous with respect to $\Omega^{\mathrm{tr}}$. The respective densities are denoted by $R_{i,j}$, $i, j\in\lbrace 1,2\rbrace$, and are given by
\begin{align}\label{eq:Rlim}
R_{i,j}(\lambda) = \lim_{\eps\downarrow 0} \frac{\im(M_{i,j}(\lambda+\I\eps))}{\im(M_{1,1}(\lambda+\I\eps) + M_{2,2}(\lambda+\I\eps))},
\end{align}
where the limit exists almost everywhere with respect to $\Omega^{\mathrm{tr}}$. One notes that $R$ is nonnegative and has trace equal to one. In particular, all entries
of $R$ are bounded, 
\begin{equation} 
0 \le R_{1,1},R_{2,2} \le 1, \quad | R_{1,2} | = |R_{2,1}| \le 1/2. 
\end{equation} 

Furthermore, the corresponding Hilbert space $L^2(\R; d\Omega)$ is associated with the inner product
\begin{align}
 \spr{\hat{f}}{\hat{g}}_\Omega = \int_\R \hat{f}(\lambda) \overline{\hat{g}(\lambda)} 
 \, d\Omega(\lambda) = 
 \int_\R \sum_{i,j=1}^2 \hat{f}_i(\lambda) R_{i,j}(\lambda) \overline{\hat{g}_j(\lambda)} \, d\Omega^{\mathrm{tr}}(\lambda),
\end{align}
where for each $f\in L^2_c((a,b);r(x)dx)$, one defines the transform, $\hat f$ of $f$, as 
\begin{align}\label{eqnSTIItrans}
 \hat{f}(z) = \begin{pmatrix} \hat f_1(z) \\ \hat f_2(z) \end{pmatrix} = 
 \begin{pmatrix} \int_a^b \theta_z(x)f(x)\, r(x) dx \\ \int_a^b \phi_z(x) f(x)\, r(x) dx \end{pmatrix}, 
 \quad z\in\C.
\end{align}
In the following lemma, we will relate the $2\times 2$ matrix-valued measure $\Omega$ to the operator-valued spectral measure $E$ of $S$. If $F$ is a measurable function on $\R$, we denote with $\M_F$ the maximally defined operator of multiplication with $F$ in the Hilbert space $L^2(\R; d\Omega)$.

\begin{lemma}\label{lemSTIIisom}
Assume that  $f$, $g\in L^2_c((a,b);r(x)dx)$. Then, 
  \begin{align}
    \spr{E((\lambda_1,\lambda_2])f}{g}_r = \spr{\M_{\indik_{(\lambda_1,\lambda_2]}} \hat{f}}{\hat{g}}_\Omega
  \end{align}
  for all $\lambda_1$, $\lambda_2\in\R$ with $\lambda_1<\lambda_2$.
\end{lemma}
\begin{proof}
This follows by evaluating Stone's formula
\begin{equation}  
\spr{E((\lambda_1,\lambda_2])f}{g}_r =
\lim_{\delta\downarrow 0} \lim_{\eps\downarrow 0} \frac{1}{\pi} \int_{\lambda_1+\delta}^{\lambda_2+\delta} \im\left( \spr{R_{\lam+\I\eps} f}{g}_r\right) d\lam,
\end{equation} 
using formula \eqref{eq:ressbc} for the resolvent together with the Stieltjes inversion formula, literally following the proof of \cite[Theorem~2.12]{GZ06}.
\end{proof}

Lemma \ref{lemSTIIisom} shows that the transformation defined in~\eqref{eqnSTIItrans} uniquely extends to an
 isometry $\mathcal{F}$ from $\Lr$ into $L^2(\R; d\Omega)$.

\begin{theorem}
 The operator $\mathcal{F}$ is unitary with inverse given by
 \begin{align}\label{eq:Finv2}
 \mathcal{F}^{-1} g(x) =
 \lim_{N\rightarrow\infty} \int_{[-N,N)} g(\lam) \begin{pmatrix} \theta_\lam(x)\\ 
 \phi_\lam(x)\end{pmatrix} d\Omega(\lam), \quad g\in L^2(\R; d\Omega),
 \end{align}
 where the limit exists in $\Lr$.
 Moreover, one has $S = \mathcal{F}^\ast \M_{\mathrm{id}} \mathcal{F}$.
\end{theorem}
\begin{proof}
Because of Lemma \ref{lemSTIIisom}, it remains to show that $\mathcal{F}$ is onto. Since it is straightforward to verify
that the integral operator on the right-hand side of \eqref{eq:Finv2} is the adjoint of $\mathcal{F}$, we can equivalently show that $\ker(\mathcal{F}^*)=\{0\}$. 
To this end, let $g\in L^2(\R; d\Omega)$, $N\in\N$, and $z\in\rho(S)$. Then
\begin{equation} 
(S-z) \int_{[-N,N)} \frac{1}{\lam-z} \, g(\lam) 
\begin{pmatrix} \theta_\lam(x)\\ \phi_\lam(x)\end{pmatrix} d\Omega(\lam)
= \int_{[-N,N)} g(\lam) \begin{pmatrix} \theta_\lam(x)\\ \phi_\lam(x)\end{pmatrix} d\Omega(\lam),
\end{equation}
since interchanging integration with differentiation can be justified using Fubini's theorem. Taking the limit
$N\to\infty$, one concludes that 
\begin{equation}\label{eqnTransRes}
\mathcal{F}^* \frac{1}{\cdot -z} \, g = R_z \mathcal{F}^* g, \quad g\in L^2(\R; d\Omega).
\end{equation}
By Stone--Weierstra{\ss}, one concludes in addition that 
$\mathcal{F}^* \M_F g = F(S) \mathcal{F}^* g$ for any
continuous function $F$ vanishing at infinity, and by a consequence of the spectral theorem 
(see, e.g., the last part of
\cite[Theorem~3.1]{Te09}), one can further extend this to characteristic functions of intervals $I$. 
Hence, for $g \in \ker(\mathcal{F}^*)$ one infers that 
\begin{equation} 
\int_I g(\lam) \begin{pmatrix} \theta_\lam(x)\\ \phi_\lam(x)\end{pmatrix} d\Omega(\lam) =0
\end{equation}
for any compact interval $I$. Moreover, after taking derivatives, one also obtains 
\begin{equation} 
\int_I g(\lam) \begin{pmatrix} \theta^\qd_\lam(x)\\ \phi^\qd_\lam(x)\end{pmatrix} d\Omega(\lam) =0.
\end{equation}
Choosing $x=x_0$ implies 
\begin{equation} 
\int_I g(\lam) \begin{pmatrix} \cos(\vphi_a)\\ \sin(\vphi_a)\end{pmatrix} d\Omega(\lam) =
\int_I g(\lam) \begin{pmatrix} -\sin(\vphi_a)\\ \cos(\vphi_a)\end{pmatrix} d\Omega(\lam) =0
\end{equation}
for any compact interval $I$, and thus $g=0$, as required.
\end{proof}

As in Lemma \ref{lemspectransGreentransform}, one can determine the transform of the Green's function upon employing Theorem \ref{thm:ressep} and equation \eqref{eqnTransRes}. 

\begin{lemma}\label{lemspectrans2Greentransform}
 For every $z\in\rho(S)$ and all $x\in(a,b)$ the transform of the Green's function $G_z(x,\cdot\,)$ and its quasi-derivative $\partial_x^\qd G_z(x,\cdot\,)$ are given by
\begin{align}
& \mathcal{F} G_z(x,\cdot\,) (\lambda) = \frac{1}{\lambda-z} \begin{pmatrix} \theta_\lam(x)\\ \phi_\lam(x)\end{pmatrix} \, \text{ and } \,  
      \mathcal{F} \partial_x^\qd G_z(x,\cdot\,) (\lambda) = \frac{1}{\lambda-z} \begin{pmatrix} \theta^\qd_\lam(x)\\ \phi^\qd_\lam(x)\end{pmatrix},    \no \\
&\hspace*{9.5cm}  \lambda\in\R.      
 \end{align}
\end{lemma}

As a consequence, one obtains the following refinement of \eqref{NHrepWeylM}:

\begin{corollary} \lb{c10.4} 
The matrix $C_2$ in \eqref{NHrepWeylM} is zero.
\end{corollary}
\begin{proof}
Following the proof of Corollary \ref{corlinearterm}, it suffices to show that 
\begin{equation}\label{eqnHerglpropM}
 \frac{\im(M(z))}{\im(z)} = \int_\R \frac{1}{|\lambda-z|^2} \, 
 d\Omega(\lambda), \quad z\in\C\backslash\R. 
\end{equation}
Therefore, one first concludes from Lemma \ref{lemspectrans2Greentransform} that for 
every $z\in\C\backslash\R$,   
\begin{equation} 
\int_a^b |G_z(x_0,y)|^2 \, r(y) dy 
= \begin{pmatrix} \cos(\vphi_a)\\ \sin(\vphi_a)\end{pmatrix} \int_\R \frac{1}{|z-\lam|^2} \, 
d\Omega(\lambda) \begin{pmatrix} \cos(\vphi_a)\\ \sin(\vphi_a)\end{pmatrix}.  \lb{10.15a} 
\end{equation} 
Using \eqref{7.4} and \eqref{eqnWTHerglotz} to evaluate the left-hand side of \eqref{10.15a}, 
one obtains 
\begin{align}
\int_a^b |G_z(x_0,y)|^2 \, r(y) dy &= \frac{1}{|W(\psi_{z,+},\psi_{z,-}))|^2} \Big(
|\psi_{z,+}(x_0)|^2 \int_a^{x_0} |\psi_{z,-}(y)|^2 \, r(y) dy    \no \\
& \quad + |\psi_{z,-}(x_0)|^2 \int_{x_0}^b |\psi_{z,+}(y)|^2 \, r(y) dy \Big)    \no \\
& =   \begin{pmatrix} \cos(\vphi_a)\\ \sin(\vphi_a)\end{pmatrix} \frac{\im(M(z))}{\im(z)} \begin{pmatrix} \cos(\vphi_a)\\ \sin(\vphi_a)\end{pmatrix}.
\end{align} 
In a similar manner, one proves corresponding formulas for
\begin{align}
\begin{pmatrix} -\sin(\vphi_a)\\ \cos(\vphi_a)\end{pmatrix} \frac{\im(M(z))}{\im(z)} \begin{pmatrix} \cos(\vphi_a)\\ \sin(\vphi_a)\end{pmatrix} \quad \text{and}\quad
\begin{pmatrix} -\sin(\vphi_a)\\ \cos(\vphi_a)\end{pmatrix} \frac{\im(M(z))}{\im(z)} \begin{pmatrix} -\sin(\vphi_a)\\ \cos(\vphi_a)\end{pmatrix}, 
\end{align}
establishing the identity \eqref{eqnHerglpropM}.
\end{proof}

We note that the vanishing of the linear term $C_2 z$ in \eqref{NHrepWeylM} is typical in this context and refer to 
\cite[Ch.\ 7]{ABT11} and \cite{Ma92} for detailed discussions. 

Finally we turn to spectral multiplicities.
Therefore, one introduces the measurable unitary matrix $U(\lam)$ which diagonalizes
$R(\lam)$, that is,
\begin{equation} \label{defrulam}
R(\lam) = U(\lam)^* \begin{pmatrix} \varrho_1(\lam) & 0 \\ 0 & \varrho_2(\lam) \end{pmatrix} U(\lam),
\end{equation}
where $0\le \varrho_1(\lam) \le \varrho_2(\lam)\le 1$ are the eigenvalues of $R(\lam)$.
In addition, one observes that $\varrho_1(\lam)+ \varrho_2(\lam)=1$ since $\tr(R(\lam))=1$. The matrix $U(\lam)$
gives rise to a unitary operator $L^2(\R; d\Omega) \to L^2(\R; \varrho_1 d\Omega^{\mathrm{tr}}) \oplus L^2(\R; \varrho_2 d\Omega^{\mathrm{tr}})$
which leaves $\M_{\mathrm{id}}$ invariant. From this observation one immediately obtains
the analog of Corollary~\ref{cor:splimm}.

\begin{corollary}
Introduce the Nevanlinna--Herglotz function
\be
M^\mathrm{tr}(z) = \tr(M(z))= \frac{m_-(z)m_+(z) -1}{m_+(z) + m_-(z)}, \quad z\in\C\backslash\R,
\ee
associated with the trace measure $d \Omega^\mathrm{tr}$.
Then the spectrum of $S$ is given by
\begin{align}
 \sigma(S) & =  \supp(d\Omega^\mathrm{tr}) = \overline{\lbrace \lambda\in\R \,|\, 0 < \limsup_{\varepsilon\downarrow 0} \im(M^\mathrm{tr}(\lambda+\I\varepsilon))\rbrace}.
\end{align}
Moreover,
\begin{align}
\sigma_p(S) & = \{ \lambda\in\R \,|\, 0 < \lim_{\varepsilon\downarrow0} \varepsilon \im(M^\mathrm{tr}(\lambda+\I\varepsilon)) \}, \\
\sigma_{ac}(S) & = \overline{\{ \lambda\in\R \,|\, 0 < \limsup_{\varepsilon\downarrow0} \im(M^\mathrm{tr}(\lambda+\I\varepsilon)) < \infty \}}^{ess},
\end{align}
and
\begin{align}\label{defSigs}
\Sigma_s = \{ \lambda\in\R \,|\, \limsup_{\varepsilon\downarrow0} \im(M^\mathrm{tr}(\lambda+\I\varepsilon)) = \infty \}
\end{align}
is a minimal support for the singular spectrum $($singular continuous plus pure point spectrum$)$ 
of $S$.
\end{corollary}

Furthermore, this allows us to investigate the spectral multiplicity of $S$.

\begin{lemma} \label{lem:spmul}
If we define
\begin{align}
\Sigma_1 &= \{ \lam\in \supp(d\Omega^{tr}) \,| \det(R(\lam))= \varrho_1(\lam) \varrho_2(\lam)=0 \},\\
\Sigma_2 &= \{ \lam\in \supp(d\Omega^{tr}) \,| \det(R(\lam))= \varrho_1(\lam) \varrho_2(\lam)>0 \},
\end{align}
then $\M_{\mathrm{id}}=  \M_{\mathrm{id}\cdot \indik_{\Sigma_1}} \oplus  \M_{\mathrm{id}\cdot \indik_{\Sigma_2}}$ and
the spectral multiplicity of $\M_{\mathrm{id}\cdot \indik_{\Sigma_1}}$ is one and
the spectral multiplicity of $\M_{\mathrm{id}\cdot \indik_{\Sigma_2}}$ is two.
\end{lemma}
\begin{proof}
For fixed $\lam\in\Sigma_1$ we have either $\varrho_1(\lam)=1$, $\varrho_2(\lam)=0$ or $\varrho_1(\lam)=0$, $\varrho_2(\lam)=1$. In the latter
case we can modify $U(\lam)$ to also switch components and hence we can assume $\varrho_1(\lam)=1$, $\varrho_2(\lam)=0$
for all $\lam\in\Sigma_1$. Hence $\M_{\mathrm{id}\cdot \indik_{\Sigma_1}}$ is unitarily equivalent to
multiplication with $\lam$ in $L^2(\R;\indik_{\Sigma_1} d\Omega^{\mathrm{tr}})$.
Moreover, since $\varrho_j \indik_{\Sigma_2} d\Omega^{\mathrm{tr}}$ and 
$\indik_{\Sigma_2} d\Omega^{\mathrm{tr}}$
are mutually absolutely continuous, $\M_{\mathrm{id}\cdot \indik_{\Sigma_2}}$ is unitary equivalent
to $\M_{\mathrm{id}}$ in the Hilbert space $L^2(\R; \indik_{\Sigma_1} d\Omega^{\mathrm{tr}} I_2)$.
\end{proof}

Combining \eqref{eq:wmmat} with \eqref{eq:Rlim}, one concludes that 
\begin{align}\label{detr}
\det(R(\lam)) = \lim_{\eps\downarrow 0} \frac{\im(m_+(\lam+\I\eps))
\im(m_-(\lam+\I\eps))}{|m_+(\lam+\I\eps) + m_-(\lam+\I\eps)|^2}
\frac{1}{\im(M^\mathrm{tr}(\lam+\I\eps))^2},
\end{align}
where the first factor is bounded by $1/4$. At this point Lemma \ref{lem:spmul} yields the 
following result. 

\begin{theorem}
The singular spectrum of $S$ has spectral multiplicity one. The
absolutely continuous spectrum of $S$ has multiplicity two on the subset 
$\sig_{ac}(S_+) \cap \sig_{ac}(S_-)$ and multiplicity one
on $\sig_{ac}(S) \backslash (\sig_{ac}(S_+) \cap \sig_{ac}(S_-))$.
Here $S_\pm$ are the restrictions of $S$ to $(a,x_0)$ and $(x_0,b)$, respectively.
\end{theorem}
\begin{proof}
Using the fact that $\Sigma_s$ is a minimal support for the singular part of $S$ one
obtains $S_s= S_{pp} \oplus S_{sc} = E(\Sigma_s) S$ and $S_{ac}= (1-E(\Sigma_s)) S$.
Thus, evaluating \eqref{detr} using \eqref{defSigs}, one infers that the singular part has multiplicity
one by Lemma \ref{lem:spmul}.

For the absolutely continuous part, one uses that the corresponding sets 
\begin{equation} 
\Sigma_{ac,\pm} = \{ \lam\in\R \,|\, 0< \lim_{\eps\downarrow 0}
\im(m_\pm(\lam+\I\eps)) <\infty \}
\end{equation}
are minimal supports for the absolutely continuous spectra of $S_\pm$.
Again, the remaining result follows from Lemma \ref{lem:spmul} upon evaluating \eqref{detr}.
\end{proof}

\section{(Non-)Principal Solutions, Boundedness from Below, and the Friedrichs Extension} \lb{s11}

In this section we develop various new applications to oscillation theory, establish the connection between 
non-oscillatory solutions and boundedness from below of $T_0$, extend a limit-point criterion for $T_0$ 
to our present general assumptions, and characterize the Friedrichs extension $S_F$ of $T_0$.  

Assuming Hypothesis \ref{h2.1}, we start by investigating some (non-)oscillatory-type properties of real-valued solutions $u\in \Deftau$ of the distributional Sturm--Liouville equation
\begin{equation}\lb{10.1}
 - \big(u^\qd\big)' + \foco u^\qd + qu =\lambda u r \; \text{ for fixed }\, \lambda \in \bbR.
\end{equation}
Throughout this section, solutions of \eqref{10.1} are always taken to be real-valued, in accordance with Theorem \ref{thm:exisuniq}. In addition, we occasionally refer to 
$p$ as being sign-definite on an interval $I \subseteq \bbR$, by which we mean that 
$p>0$ or $p<0$ a.e.\ on $I$. 

We begin with a Sturm-type separation theorem for the zeros of pairs of linearly independent real-valued solutions of \eqref{10.1}.

\begin{theorem}\lb{t10.1}
Assume Hypothesis \ref{h2.1} and suppose that $u_j$, $j=1,2$, are two linearly independent real-valued solutions of \eqref{10.1} for a fixed $\lambda \in \bbR$.  If $x_j\in (a,b)$, $j=1,2$, are two zeros of $u_1$ with $x_1<x_2$ and $p$ is sign-definite on $(x_1,x_2)$, then $u_2$ has at least one zero in $[x_1,x_2]$.  If, in addition, $\tau$ is regular at the endpoint $a$ and $x_1=a$, then $u_2$ has a zero in $[a,x_2]$.  An analogous result holds if $\tau$ is regular at the endpoint $b$.
\end{theorem}
\begin{proof}
Since the Wronskian of two real-valued solutions of \eqref{10.1} is a constant (cf.\ the discussion after Lemma \ref{propLagrange}), 
\begin{equation}\lb{10.2}
W(u_1,u_2)(x)=u_1(x)u_2^{[1]}(x)-u_1^{[1]}(x)u_2(x)=c, \quad x\in [x_1,x_2],
\end{equation}
for some $c\in \bbR$.  If $u_2$ has no zero in $[x_1,x_2]$ then the quotient $u_1/u_2$ is absolutely continuous on $[x_1,x_2]$ and \eqref{10.2} implies
\begin{equation}\lb{10.3}
\bigg(\frac{u_1}{u_2} \bigg)'(x)=-\frac{c}{p(x)u_2(x)^2} \; \text{ for a.e.\ $x\in (x_1,x_2)$.}
\end{equation}
Subsequently, integrating the equation in \eqref{10.3} from $x_1$ to $x_2$ and using $u_1(x_j)=0$, $j=1,2$, one obtains
\begin{equation}\lb{10.4}
c\int_{x_1}^{x_2}\frac{dx}{p(x)u_2(x)^2}=0.
\end{equation}
The sign definiteness assumption on $p$ implies the integral appearing in \eqref{10.4} is nonzero, and, consequently, one concludes $c=0$.  Therefore,  $u_1$ and $u_2$ must be linearly dependent real-valued solutions of \eqref{10.1}.  The result now follows by contraposition.

To prove the remaining statement, one may simply repeat the above argument, noting that 
regularity of $\tau$ at the endpoint $a$ guarantees that the function appearing in the right hand 
side of \eqref{10.3} is integrable on $(a,x_2)$.
\end{proof}

Note also that all zeros are simple in the sense that (nontrivial) solutions must change sign at a zero.

\begin{lemma}\lb{l10.1a}
Assume Hypothesis \ref{h2.1} and suppose that $u$ is a nontrivial real-valued solution of \eqref{10.1} for a fixed $\lambda \in \bbR$. If $x_0\in (a,b)$ is a zero and $p$ is sign-definite in a neighborhood of $x_0$, then $u$ must change sign at $x_0$.
\end{lemma}
\begin{proof}
Regarding $u'(x)=p(x)^{-1} u^\qd(x) - s(x) u(x)$ as a differential equation for $u$ we obtain
\begin{equation}
u(x) = \E^{-S(x)} \int_{x_0}^x \E^{S(y)} p(y)^{-1} u^\qd(y) \, dy, \qquad S(x) = \int_{x_0}^x s(y) \, dy.
\end{equation}
Since $u^\qd(x_0) \neq 0$ (otherwise, $u \equiv 0$) and $u^\qd \in AC_{\loc}((a,b))$, the claim follows.
\end{proof}

\begin{definition} \lb{d10.2} 
Suppose Hypothesis \ref{h2.1} holds and let $\lambda \in \bbR$.  The differential expression $\tau-\lambda$ is called {\it oscillatory at $a$} (resp., $b$) if some solution of \eqref{10.1} has infinitely many zeros accumulating at $a$ (resp., $b$); otherwise, $\tau-\lambda$ is called {\it non-oscillatory at $a$} (resp., $b$).
\end{definition}

Under the assumption that $\tau-\lambda$ is non-oscillatory at the endpoint $b$, and that $p$ is 
sign-definite a.e.\ on $(c,b)$, the next result establishes the existence of a distinguished solution 
which is, in a heuristic sense, ``smaller'' than any other solution near $b$.  An analogous result 
holds if \eqref{10.1} is non-oscillatory at $a$.

\begin{theorem} \lb{t10.3}
Assume Hypothesis \ref{h2.1} and let $\lambda \in \bbR$ be fixed.  In addition, suppose that there exists $c\in (a,b)$ such that $p$ is sign-definite a.e.\ on $(c,b)$.  If $\tau-\lambda$ is non-oscillatory at $b$, there exists a real-valued solution $u_0$ of \eqref{10.1} satisfying the following properties $(i)$--$(iii)$ in which $u_1$ denotes an arbitrary real-valued solution of \eqref{10.1} linearly independent of $u_0$.\\
$(i)$ $u_0$ and $u_1$ satisfy the limiting relation
\begin{equation}\lb{10.5}
\lim_{x\uparrow b}\frac{u_0(x)}{u_1(x)}=0.
\end{equation}
$(ii)$  $u_0$ and $u_1$ satisfy
\begin{equation}\lb{10.6}
\int ^b\frac{dx}{|p(x)|u_1(x)^2}<\infty \, \text{ and } \,  \int ^b\frac{dx}{|p(x)|u_0(x)^2}=\infty.   
\end{equation}
$(iii)$ Suppose $x_0\in(c,b)$ strictly exceeds the largest zero, if any, of $u_0$, and 
$u_1(x_0)\neq 0$.  If $u_1(x_0)/u_0(x_0)>0$, then $u_1$ has no 
$($resp., exactly one$)$ zero in $(x_0,b)$ if $W(u_0,u_1) \gtrless 0$ $($resp., 
$W(u_0,u_1) \lessgtr 0$$)$, in the case $p\gtrless 0$ a.e.\ on $(c,b)$.  On the other hand, 
if $u_1(x_0)/u_0(x_0)<0$, then $u_1$ has no $($resp., exactly one$)$ 
zero in $(x_0,b)$ if $W(u_0,u_1) \lessgtr 0$ 
$($resp., $W(u_0,u_1) \gtrless 0$$)$ in the case $p\gtrless 0$ a.e.\ on $(c,b)$.\\
\end{theorem}
\begin{proof}
Let $u$ and $v$ denote a pair of linearly independent real-valued solutions of \eqref{10.1}.  Then their Wronskian is a nonzero constant, say $c\in \bbR\backslash\{0\}$.  If $x_0\in (c,b)$ strictly exceeds the largest zero, if any, of $v$, then $u/v\in AC_{\loc}((x_0,b))$, and one verifies (as in \eqref{10.3}) that 
\begin{equation}\lb{10.7}
\bigg(\frac{u}{v} \bigg)'(x)=-\frac{c}{p(x)v(x)^2} \; \text{ for a.e.\ $x\in (x_0,b)$.}
\end{equation}  
In particular, since $p$ is sign definite a.e.\ on $(x_0,b)$, the right-hand side of equation \eqref{10.7} is sign definite a.e.\ on the same interval; therefore, the function $u/v$ is monotone on $(x_0,b)$.  Consequently, 
\begin{equation}\lb{10.8}
C=\lim_{x\uparrow b}\frac{u(x)}{v(x)}
\end{equation}
exists, where $C=\pm \infty$ is permitted.  By renaming $u$ and $v$, if necessary, one may take $C=0$.  Indeed, in the case $C=\pm\infty$ in \eqref{10.8}, one simply interchanges the roles of the functions $u$ and $v$.  If $0<|C|<\infty$, then one replaces the solution $u$ by the linear combination $u-Cv$.     Choosing $u_0=u$, a real-valued solution $u_1$ of \eqref{10.1} is linearly independent of $u_0$ if and only if it is of the form $u_1=c_0u_0+c_1v$ with $c_1\neq0$.  In this case, $C=0$ implies
\begin{equation}\lb{10.9}
u_1(x)\underset{x\uparrow b}{=}[c_1+\oh(1)]v(x),
\end{equation}
and, consequently, \eqref{10.5}.  This proves item $(i)$.

In order to prove item $(ii)$, we first note a useful consequence of \eqref{10.7}.  To this end, suppose $u$ and $v$ are real-valued solutions of \eqref{10.1} and that $x_0'$ strictly exceeds the largest zero of $v$, so that \eqref{10.7} holds as before.  Integrating \eqref{10.7} from $x_0'$ to $x\in (x_0',b)$ and using sign-definiteness of $p$ yields 
\begin{equation}\lb{10.10}
\int_{x_0'}^{x}\frac{dt}{|p(t)|v(t)^2}=\frac{1}{|c|}\bigg|\frac{u(x)}{v(x)}-\frac{u(x_0')}{v(x_0')}\bigg|, \quad x\in (x_0',b).
\end{equation}
To prove item $(ii)$, let $u_1$ denote a real-valued solution linearly independent of $u_0$ (with $u_0$ the solution constructed in item $(i)$) and choose $x_0\in (c,b)$ strictly exceeding the largest zero of $u_0$ and the largest zero of $u_1$.  Choosing $u=u_0$ and $v=u_1$ (resp., $u=u_1$ and $v=u_0$) in \eqref{10.10}, taking the limit $x\uparrow b$, and applying \eqref{10.5} establishes convergence (resp., divergence) of the first (resp., second) integral appearing in \eqref{10.6}.  This completes the proof of item $(ii)$.

To prove item $(iii)$, we assume the case $p>0$ a.e.\ on $(c,b)$ for simplicity; the case $p<0$ a.e.\ on $(c,b)$ is handled similarly.  One infers from \eqref{10.5} and \eqref{10.7} (with $u=u_1$ and $v=u_0$) that $u_1/u_0$ is monotonic on $(x_0,b)$ and that
\begin{equation}\lb{10.11}
\lim_{x\uparrow b}\frac{u_1(x)}{u_0(x)}=\pm\infty, \, \text{ depending on whether } \, 
W(u_0,u_1) \gtrless 0
\end{equation}
As a result, if $u_1(x_0)/u_0(x_0)> 0$ then $u_1/u_0$ has no (resp., exactly one) zero in 
$(x_0,b)$ in the case $W(u_0,u_1) > 0$ (resp., $W(u_0,u_1) < 0$).  On the other hand, if 
$u_1(x_0)/u_0(x_0)<0$, then $u_1/u_0$ has no (resp., exactly one) zero in $(x_0,b)$ in the 
case $W(u_0,u_1) < 0$ (resp., $W(u_0,u_1) > 0$). (All Wronskians are of course constant, hence 
we evaluate them at $x_0$.) Item $(iii)$ now follows since the zeros of $u_1$ 
in $(x_0,b)$ are precisely the zeros of $u_1/u_0$.
\end{proof}

Evidently, a result analogous to Theorem \ref{t10.3} holds if $\tau-\lambda$ is non-oscillatory at $a$.  More specifically, one can establish the existence of a distinguished real-valued solution $v_0\neq 0$ of \eqref{10.1} which satisfies the following analogue to \eqref{10.5}:  If $v_1$ is any real-valued solution of \eqref{10.1} linearly independent of $v_0$, then
\begin{equation}\lb{10.12}
\lim_{x\downarrow a}\frac{v_0(x)}{v_1(x)}=0.
\end{equation}
Analogues of item $(ii)$ and $(iii)$ of Theorem \ref{t10.5} subsequently hold for $v_0$ and any real-valued solution $v_1$ linearly independent of $v_0$.

\begin{definition} \lb{d10.4} 
Assume Hypothesis \ref{h2.1} and suppose that $\lambda \in \bbR$.  If $\tau-\lambda$ is non-oscillatory at $c\in \{a,b\}$, then a nontrivial real-valued solution $u_0$ of \eqref{10.1} which satisfies 
\begin{equation}
\lim_{\substack{x\rightarrow c \\ x\in (a,b)}} \frac{u_0(x)}{u_1(x)}=0
\end{equation}
 for any other linearly independent real-valued solution $u_1$ of \eqref{10.1} is called a {\it principal solution} of \eqref{10.1} at $c$.  A real-valued solution of \eqref{10.1} linearly independent of a principal solution at $c$ is called a {\it non-principal solution} of \eqref{10.1} at $c$.
\end{definition}

If $\tau-\lambda$ is non-oscillatory at $c\in \{a,b\}$, one verifies that a principal solution at $c$ is unique up to constant multiples.  The main ideas for the proof of Theorem \ref{t10.3} presented above are taken from 
\cite[Theorem\ 11.6.4]{Ha02}; the notion of (non-)principal solutions dates back at least to 
Hartman \cite{Ha48} and was subsequently also used by Rellich \cite{Re51}.

If the differential expression $\tau-\lambda$ is non-oscillatory at $c\in\{a,b\}$, one can use any nonzero real-valued solution to construct a non-principal solution in a neighborhood of $c$.  The procedure for doing so is the content of our next result.  For simplicity, we consider only the case when $\tau-\lambda$ is non-oscillatory at $b$.  An analogous technique allows one to construct (non-)principal solutions near $a$ when $\tau-\lambda$ is non-oscillatory at $a$.

\begin{theorem}\lb{t10.5}
Assume Hypothesis \ref{h2.1} and suppose that $\tau-\lambda$ is non-oscillatory at $b$.  In addition, suppose that there exists $c\in (a,b)$ such that $p$ is sign-definite a.e.\ on $(c,b)$.  Let $u\neq 0$ be a real-valued solution of \eqref{10.1} and let $x_0\in (c,b)$ strictly exceed its last zero.  Then 
\begin{equation}\lb{10.13}
u_1(x)=u(x)\int_{x_0}^x \frac{dx'}{p(x')u(x')^2}, \quad x\in (x_0,b),
\end{equation}
is a non-principal solution of \eqref{10.1} on $(x_0,b)$.  If, on the other hand, $u$ is a non-principal solution of \eqref{10.1}, then
\begin{equation}\lb{10.14}
u_0(x)=u(x)\int_{x}^b \frac{dx'}{p(x')u(x')^2}, \quad x\in (x_0,b),
\end{equation}
is a principal solution of \eqref{10.1} on $(x_0,b)$.  Analogous results hold at $a$.
\end{theorem}
\begin{proof}
Suppose that $u\neq 0$ is a real-valued solution of \eqref{10.1} and define $u_1$ by \eqref{10.13}.  Evidently, $u_1$ is real-valued and $u_1\in AC_{\loc}((x_0,b))$.  In addition, $u_1\in \Deftau$ since
\begin{equation}\lb{10.15}
u_1^{[1]}(x)=\frac{1}{u(x)}+u^{[1]}(x)\int_{x_0}^x \frac{dx'}{p(x')u(x')^2}\in AC_{\loc}((x_0,b)),
\end{equation}
and one verifies $\tau u_1=\lambda u_1$ on $(x_0,b)$.  Moreover, $u_1$ is linearly independent of $u$ since $W(u,u_1)=1$, and $u_1$ is not a principal solution on $(x_0,b)$ because
\begin{equation}\lb{10.17}
\lim_{x\uparrow b}\frac{u_1(x)}{u(x)}=\lim_{x\uparrow b}\int_{x_0}^x \frac{dx'}{p(x')u(x')^2}\neq 0.
\end{equation}
It follows that $u_1$ is a non-principal solution on $(x_0,b)$.

Under the additional assumption that $u$ is a non-principal solution, one again readily verifies that $u_0$ defined by \eqref{10.14} is a solution on $(x_0,b)$, and that $u_0$ is linearly independent of $u$. Next, 
we write $u_0= c_0 \widetilde u_0 + c_1 u$ on $(x_0,b)$, where $\widetilde u_0$ is a principal solution 
on $(x_0,b)$ and $c_0, c_1 \in \bbR$.  Then after dividing through by $u$, one computes
\begin{equation}\lb{10.18}
0=\lim_{x\uparrow b}\int_{x}^b \frac{dx'}{p(x')u(x')^2}=c_0\lim_{x\uparrow b}\frac{\widetilde u_0(x)}{u(x)}+c_1=c_1,
\end{equation}
and it follows that $u_0=c_0\widetilde u_0$ is a principal solution on $(x_0,b)$.
\end{proof}

The following result establishes an intimate connection between non-oscillatory behavior and the l.p.\ case for $\tau$ at an endpoint.  More specifically, we derive a criterion for concluding that $\tau$ is in the l.p.\ case at an endpoint in the situation where $\tau-\lambda$ is non-oscillatory at the endpoint and $p$ has fixed sign in a neighborhood of the endpoint.  The proof of this result relies on the existence of principal solutions, as established in Theorem \ref{t10.3}, as well as the technique for constructing non-principal solutions described in Theorem \ref{t10.5}.  This condition is well-known within the context of traditional three-term Sturm--Liouville differential expressions of the form $\tau_0 u = r^{-1}[-(pu')'+qu]$, where $p>0, r>0$ a.e.\ and  
$p^{-1}, r, q \in L^1_{\loc}((a,b))$, etc. It was first derived by Hartman \cite{Ha48} in the particular 
case $p=r=1$ in 1948.  Three years later, Rellich \cite{Re51} extended the result to the general three-term case under some additional smoothness assumptions on $p, r$, and $q$. These smoothness restrictions, however, are inessential (see also \cite[Lemma C.1]{Ge93}). The following result extends this l.p.\ criterion to the general case governed by Hypothesis \ref{h2.1}. 

\begin{theorem}\lb{t10.6}
Assume Hypothesis \ref{h2.1} and suppose that there exists $c\in (a,b)$ such that $p$ is sign-definite a.e.\ on $(c,b)$.  In addition, suppose that $\tau-\lambda$ is non-oscillatory at $b$ for some $\lambda \in \bbR$.  If $\int^b |r(x)/p(x)|^{1/2} dx=\infty$, then $\tau$ is in the l.p.\ case at $b$.  An analogous result holds at $a$.
\end{theorem}
\begin{proof}
Since $\tau-\lambda$ is non-oscillatory at $b$, there exists a principal solution, say $u_0$, of \eqref{10.1} by Theorem \ref{t10.3}.  If $x_0$ strictly exceeds the largest zero of $u_0$ in $(c,b)$, then by Theorem \ref{t10.5}, $u_1$ defined by
\begin{equation}\lb{10.19}
u_1(x)=u_0(x)\int_{x_0}^x \frac{dx'}{p(x')u_0(x')^2}, \quad x\in (x_0,b),
\end{equation}
is a non-principal solution on $(x_0,b)$, and as a result,
\begin{equation}\lb{10.20}
\int_{x_0}^b \frac{dx}{|p(x)|u_1(x)^2}<\infty.
\end{equation}
Assuming $\tau$ to be in the l.c.\ case at $b$, one concludes that 
\begin{equation}\lb{10.21}
\int_{x_0}^b u_1(x)^2 r(x)dx<\infty.
\end{equation}
Consequently, H\"older's inequality yields the contradiction, 
\begin{align}\lb{10.22}
\int_{x_0}^b |r(x)/p(x)|^{1/2} dx \leq \bigg|\int_{x_0}^b  u_1(x)^2 r(x)dx \bigg|^{1/2}
\bigg|\int_{x_0}^b \frac{dx}{|p(x)|u_1(x)^2}\bigg|^{1/2}<\infty.
\end{align}
\end{proof}

\begin{corollary} \lb{c10.7} 
Assume Hypothesis \ref{h2.1}.  Suppose $\tau-\lambda_a$ is non-oscillatory at $a$ for some 
$\lambda_a\in \bbR$ and that $\tau-\lambda_b$ is non-oscillatory at $b$ for some 
$\lambda_b \in \bbR$.  If $p$ is sign-definite in neighborhoods of $a$ and $b$ $($the sign 
of $p$ may be different in the two neighborhoods$)$, and 
\begin{equation}\lb{10.23}
\int_a |r(x)/p(x)|^{1/2} dx=\infty, \quad \int^b |r(x)/p(x)|^{1/2} dx=\infty,
\end{equation}
then $\Tmin=\Tmax$ is a self-adjoint operator.
\end{corollary}
\begin{proof}
By Theorem \ref{t10.6}, $\tau$ is in the l.p.\ case at $a$ and $b$.  The result now follows from 
Theorem \ref{t5.2}.
\end{proof}

\begin{theorem} \lb{t10.8} 
Assume Hypothesis \ref{h2.1} and that $p>0$ a.e.\ on $(a,b)$.  Suppose there exist $\lambda_a,\lambda_b\in \bbR$ such that $\tau-\lambda_a$ is non-oscillatory at $a$ and $\tau-\lambda_b$ is non-oscillatory at $b$.\ Then $T_0$ and hence any 
self-adjoint extension $S$ of the minimal operator $T_{\min}$ is bounded from below.  That is, there exists $\gamma_{{}_S} \in \bbR$, such that
\begin{equation}
\langle u,S u\rangle_r \geq \gamma_{{}_S} \langle u,u\rangle_r, \quad u\in \dom{S}. \lb{10.24}
\end{equation}
\end{theorem}
\begin{proof}
Since $\tau-\lambda_a$ is non-oscillatory at $a$ and  $\tau-\lambda_b$ is non-oscillatory at $b$, there exist real-valued solutions $f_a,f_b\in \Deftau\backslash \{0\}$ satisfying 
\begin{equation}
(\tau - \lambda_a)f_a=0, \; (\tau - \lambda_b)f_b=0 \, \text{ a.e.\ on $(a,b)$},
\end{equation}
such that $f_a$ does not vanish in a neighborhood, say $(a,c)$ of $a$, and $f_b$ does not vanish in a neighborhood, say $(d,b)$, of $b$.  We may assume that $c<d$.  Note that the solution $f_a$ can have at most finitely many (distinct) zeros in the interval $(c,d)$.  For if $f_a$ has infinitely many zeros in $(c,d)$, then zeros of $f_a$ must accumulate at some point in $[c,d]$.  Let $\{c_n\}_{n=1}^{\infty}\subset (c,d)$ denote such a sequence of zeros and $c_{\infty} \in [c,d]$ with $\lim_{n\rightarrow \infty}c_n=c_{\infty}$.  Since $f_a$ is continuous on $[c,d]$, the accumulation point $c_{\infty}$ is also a zero of $f_a$, that is, 
\begin{equation}\lb{10.26a}
f_a(c_{\infty})=0.
\end{equation}
Let $f$ denote a real-valued solution of $(\tau-\lambda_a)f=0$ linearly independent of $f_a$ so that the Wronskian of $f$ and $f_a$ is a nonzero constant
\begin{equation}\lb{10.27a}
W(f,f_a)(c_{\infty})\in \bbR\backslash\{0\}.
\end{equation}
By the Sturm separation Theorem \ref{t10.1}, the zeros of $f_a$ and $f$ intertwine.  In particular, $c_{\infty}$ must also be a limit point of zeros of $f$, and by continuity of $f$ on $[c,d]$,
\begin{equation}\lb{10.28a}
f(c_{\infty})=0.
\end{equation}
However, \eqref{10.26a} and \eqref{10.28a} are a contradiction to \eqref{10.27a}, and it follows that $f_a$ has only finitely many zeros in $(c,d)$.  

Let $\{c_n\}_{n=2}^{N-1}\subset (c,d)$, $N\in \bbN$ chosen appropriately, denote a listing of the finitely many (distinct) zeros of $f_a$ in $(c,d)$ with $c_{n}<c_{n+1}$, $2\leq n \leq N-2$, and set $c_1=c$ and $c_N=d$.  Define the operators $T_{0,(a,c)}$, $T_{0,(c_n,c_{n+1})}$, $1\leq n\leq N-1$, and $T_{0,(d,b)}$ in the following manner: 
\begin{align}
&T_{0,(a,c)}f_1=\tau f_1,\lb{10.26}\\
&f_1\in \dom {T_{0,(a,c)}}=\big\{g|_{(a,c)}\, \big|\, g\in \dom {T_{\max}},\, \text{$g$ has compact support in $(a,c)$}\big\},\no\\[1mm]
&T_{0,(d,b)}f_2=\tau f_2,\lb{10.28}\\
&f_2\in \dom {T_{0,(d,b)}}=\big\{g|_{(d,b)}\, \big|\, g\in \dom {T_{\max}},\, \text{$g$ has compact support in $(d,b)$}\big\},\no\\[1mm]
&T_{0,(c_n,c_{n+1})}f_3=\tau f_3,\lb{10.27}\\
&f_3\in \dom {T_{0,(c_n,c_{n+1})}}=\big\{g|_{(c_n,c_{n+1})}\, \big|\, g\in \dom {T_{\max}},\, \text{ $\supp(g)\subset(c_n,c_{n+1})$}\big\},\no\\
&\hspace*{9.65cm} 1\leq n\leq N-1.\no
\end{align}
Obviously, $T_0$ defined by \eqref{3.3} is an extension of the direct sum $T_{0,\oplus}$ defined by 
\begin{equation}
T_{0,\oplus} =T_{0,(a,c)}\oplus T_{0,(c_1,c_2)}\oplus \cdots \oplus T_{0,(c_{N-1},c_N)}\oplus T_{0,(d,b)}.
\end{equation}
Moreover, $T_{0,\oplus}\subset \ol {T_{0,\oplus}}\subset T_{\min}$, and any self-adjoint extension of $T_{\min}$ is a self-adjoint extension of $T_{0,\oplus}$.  
Since the deficiency indices of $T_{\min}$ are at most $2$, it suffices to show that 
\begin{equation}\lb{10.30}
\text{$T_{0,\oplus}$ is bounded from below}.
\end{equation}
Subsequently, by \cite[Corollary\ 2, p.\ 247]{We80}, \eqref{10.30} implies that any self-adjoint extension 
of $T_{0,\oplus}$ (hence, any self-adjoint extension of $T_{\min}$) is bounded from below since the deficiency indices of $\ol{T_{0,\oplus}}$ are finite (in fact, they are at most $2N+2$).  It suffices to show that the symmetric operators \eqref{10.26}--\eqref{10.27} are separately bounded from below; a lower bound for 
$T_{0,\oplus}$ is then taken to be the smallest of the lower bounds for \eqref{10.26}--\eqref{10.27}.

The proof that $T_{0,(a,c)}$ and $T_{0,(d,b)}$ are bounded from below relies on the non-oscillatory assumptions on $\tau-\lambda_a$ and $\tau-\lambda_b$.  Since $(\tau-\lambda_a)f_a=0$ a.e.\ on 
$(a,b)$ and $f_a$ does not vanish on $(a,c)$, one can recover $q$ pointwise a.e.\ on $(a,c)$ by 
\begin{equation}\lb{10.31}
q(x)=\lambda_a r(x)-\foco(x)\frac{f_a^{[1]}(x)}{f_a(x)}+\frac{\big(f_a^{[1]}\big)'(x)}{f_a(x)} 
\, \text{ for a.e.\ $x \in (a,c)$.}
\end{equation}
Let $u\in \dom{T_{0,(a,c)}}$ be fixed.  Using \eqref{10.31} in conjunction with the fact that functions in $\dom{T_{0,(a,c)}}$ vanish in neighborhoods of $a$ and $c$ (to freely perform integration by parts), one computes
\begin{align}
&\langle u, T_{0,(a,c)}u\rangle_{L^2((a,c);r(x)dx)}-\lambda_a \langle u,u\rangle_{L^2((a,c);r(x)dx)}   \no \\
&\quad =\int_{(a,c)}  \bigg\{u'(x)\ol{u^{[1]}(x)}+u(x)\foco(x)\ol{u^{[1]}(x)}-\foco(x)|u(x)|^2\frac{f_a^{[1]}(x)}{f_a(x)} \lb{10.32} \\
&\hspace*{2.1cm}-\frac{f_a^{[1]}(x)}{f_a}\big(u'(x)\ol{u(x)}+u(x)\ol{u'(x)} \big)+\frac{f_a^{[1]}(x)f_a'(x)|u(x)|^2}{f_a(x)^2} \bigg\}dx.\no
\end{align}
Denoting the integrand on the right-hand side of \eqref{10.32} by $F_u(x)$ a.e.\ in $(a,c)$, algebraic manipulations using the definition of the quasi-derivative yield
\begin{equation}\lb{10.36}
F_u(x)=p(x)\bigg|u'(x)-u(x)\frac{f_a'(x)}{f_a(x)} \bigg|^2\geq 0 \, \text{ for a.e.\ $x\in (a,c)$.}
\end{equation}
Therefore, the integral appearing in the right-hand side of \eqref{10.32} is nonnegative.  Since 
$u\in \dom{T_{0,(a,c)}}$ is arbitrary, one obtains the lower bound
\begin{equation}\lb{10.34}
\langle u, T_{0,(a,c)}u\rangle_{L^2((a,c);r(x)dx)}\geq \lambda_a \langle u,u\rangle_{L^2((a,c);r(x)dx)}, \quad u\in {\rm dom}(T_{0,(a,c)}).
\end{equation}
The analogous strategy, using the solution $f_b$, establishes the lower bound for $T_{0,(d,b)}$, 
\begin{equation}\lb{10.35}
\langle u, T_{0,(d,b)}u\rangle_{L^2((d,b);r(x)dx)}\geq \lambda_b \langle u,u\rangle_{L^2((d,b);r(x)dx)}, \quad u\in {\rm dom}(T_{0,(d,b)}).
\end{equation}
To show that each $T_{0,(c_n,c_{n+1})}$, $1\leq n \leq N-1$, is semi-bounded from below, one closely follows the strategy used above to prove semi-boundedness of $T_{0,(a,c)}$, noting that since $f_a$ is nonvanishing on $(c_n,c_{n+1})$, $q$ can be solved for a.e.\ on the interval $(c_n,c_{n+1})$ in the same manner as in \eqref{10.31}. Then if $u\in \dom{T_{0,(c_n,c_{n+1})}}$, one obtains an identity which formally reads like \eqref{10.32} with the interval $(a,c)$ everywhere replaced by $(c_n,c_{n+1})$.  Factoring the integrand according to the factorization appearing on the right-hand side of the equality in \eqref{10.36} (this time a.e.\ on $(c_n,c_{n+1})$), one infers that
\begin{equation}\lb{10.39}
\begin{split}
\langle u, T_{0,(c_n,c_{n+1})}u\rangle_{L^2((c_n,c_{n+1});r(x)dx)}\geq \lambda_a \langle u, u \rangle_{L^2((c_n,c_{n+1});r(x)dx)},&\\
u\in \dom{T_{0,(c_n,c_{n+1})}},\, 1\leq n \leq N-1.&
\end{split}
\end{equation}
Together, \eqref{10.34}, \eqref{10.35}, and \eqref{10.39}, yield \eqref{10.30}, and hence \eqref{10.24}.
\end{proof}

\begin{corollary} \lb{c10.9}
Assume Hypothesis \ref{h2.1} and suppose that $p>0$ a.e.\ on $(a,b)$.  If $\tau$ is regular 
on $(a,b)$, then $T_0$ and hence every self-adjoint extension of $T_{\min}$ 
is bounded from below.
\end{corollary}
\begin{proof}
We claim that the differential expression $\tau$ is non-oscillatory at $a$. Indeed, if $\tau$ were oscillatory 
at $a$, then $\tau u=0$ has a nontrivial, real-valued solution $u_a$ with zeros accumulating at $a$.  
Let $v$ denote a nontrivial, real-valued solution of $\tau u =0$ linearly independent of $u_a$. Then 
Theorem \ref{t10.1} implies that $v$ also has zeros accumulating at $a$. By Theorem \ref{thm:EEreg}, 
$u_a, v$, and their quasi-derivatives have limits at $a$; by continuity,
\begin{equation}
\lim_{x\downarrow a}u_a(x)=\lim_{x\downarrow a}v(x)=0.
\end{equation}
As a result, the Wronskian of $u_a$ and $v$ must satisfy
\begin{equation}
\lim_{x\downarrow a}W(u_a,v)(x)=0,
\end{equation}
which yields a contradiction since the Wronskian of $u_a$ and $v$ equals a fixed, nonzero constant everywhere in $(a,b)$.  Similarly, one shows that $\tau$ is non-oscillatory at $b$.  The result now follows 
by applying Theorem \ref{t10.8}, with, say, $\lambda_a=\lambda_b=0$.
\end{proof}

Corollary \ref{c10.9}, under our present general assumptions, has originally been proved by M\"oller and Zettl \cite{MZ95} using a different approach (and for the general even-order case considered in \cite{We87} with  a positive leading coefficient).  

\begin{corollary} \lb{c10.10}
Assume Hypothesis \ref{h2.1} and suppose $p$ is sign-definite a.e.\ in $(a,b)$.  If $\tau$ is regular on $(a,b)$ and 
$\lambda \in \bbR$, then any nontrivial, real-valued solution of $\tau u=\lambda u$ has only finitely many zeros in $(a,b)$.
\end{corollary}
\begin{proof}
By absorbing $\lambda$ into $\tau$, it suffices to consider the case $\lambda=0$.  A nontrivial, real-valued function $u$ satisfying $\tau u=0$ cannot have zeros accumulating at a point in $[a,b]$.
\end{proof}

\begin{definition} \lb{d10.11} 
Assume Hypothesis \ref{h2.1}.  The operator $T_0$ (defined by \eqref{3.3}) is said to be 
{\it bounded from below at $a$} if there exists a $c\in (a,b)$ and a $\lambda_a\in \bbR$ such that 
\begin{equation}
\langle u, T_0u\rangle_r\geq \lambda_a \langle u, u\rangle_r, \quad u\in \dom {T_0} \, \text{such that $u\equiv 0$ on $(c,b)$}.
\end{equation}
Similarly, $T_0$ is said to be {\it bounded from below at $b$} if there exists a $d\in (a,b)$ and a $\lambda_b\in \bbR$ such that 
\begin{equation}
\langle u, T_0u\rangle_r\geq \lambda_b \langle u, u\rangle_r, \quad u\in \dom {T_0} \, \text{such that $u\equiv 0$ on $(a,d)$}.
\end{equation}
\end{definition}

\begin{theorem} \lb{t10.12}
Assume Hypothesis \ref{h2.1}.  If $T_0$ is bounded from below at $a$ and $p$ is sign-definite a.e.\ 
near $a$, then there exists an $\alpha\in \bbR$ such that for all $\lambda < \alpha$, 
$\tau-\lambda$ is non-oscillatory at $a$.  A similar result holds if $T_0$ is bounded from below at $b$.
\end{theorem}
\begin{proof}
By assumption, there exists a $c\in(a,b)$ such that each self-adjoint extension $S_{(a,c)}$ of 
$\tau_{(a,c)}$ with separated boundary conditions in $L^2((a,c);r(x)dx)$ is bounded from below by some $\alpha\in\R$. More precisely, this follows from Definition \ref{d10.11} and 
\cite[Corollary\ 2 on p.\ 247]{We80}. Then for each $\lambda<\alpha$, the diagonal of the 
corresponding Green's function $G_{(a,c),\lambda}(x,x)$, $x\in(a,c)$ is nonnegative (cf.\ 
\cite[Lemma on p.\ 195]{Jo82}). In fact, since $G_{(a,c),\lambda}$ is continuous on $(a,c) \times (a,c)$ one has
\begin{align}
 G_{(a,c),\lambda}(x,x) = \lim_{\eps\rightarrow0} \spr{(S_{(a,c)}-\lambda)^{-1}f_{x,\eps}}{f_{x,\eps}}_{L^2((a,c);r(x)dx)} \geq 0 \lb{Ques1}
\end{align}
for each $x\in(a,c)$, where
\begin{align}
 f_{x,\eps}(y) = \left(\int_{x-\eps}^{x+\eps} r(t)dt\right)^{-1} \chi_{(x-\eps,x+\eps)}(y), \quad y\in(a,c), 
 \; \eps>0.\lb{Ques2}
\end{align}
Indeed, if $x\in (a,c)$, then by continuity along the diagonal, for any $\delta>0$, there exists an $\varepsilon(\delta)>0$ such that
\begin{equation}
\begin{split}
G_{(a,c),\lambda}(x,x) - \delta \leq G_{(a,c),\lambda}(s,t) \leq G_{(a,c),\lambda}(x,x) + \delta,&\\
(s,t)\in(x-\eps,x+\eps)\times(x-\eps,x+\eps), \quad \varepsilon<\varepsilon(\delta).&
\end{split}
\end{equation}
As a result,
\begin{equation}
\begin{split}
G_{(a,c),\lambda}(x,x) - \delta  \leq \spr{(S_{(a,c)}-\lambda)^{-1} f_{x,\eps}}{f_{x,\eps}} \leq G_{(a,c),\lambda}(x,x) + \delta,&\\
\varepsilon< \varepsilon(\delta), \; \delta>0.&
\end{split}
\end{equation}
Therefore, one obtains
\begin{equation}
G_{(a,c),\lambda}(x,x) - \delta  \leq \liminf_{\varepsilon\downarrow 0}
\spr{(S_{(a,c)}-\lambda)^{-1} f_{x,\eps}}{f_{x,\eps}} \leq G_{(a,c),\lambda}(x,x) + \delta, \quad 
\delta>0,
\end{equation}
and the analogous inequality with ``$\liminf$'' replaced by ``$\limsup$.''  Subsequently taking $\delta \downarrow 0$ yields \eqref{Ques1}.

Now let $u_a$ and $u_c$ be solutions of $(\tau-\lambda) u = 0$ lying in $L^2((a,c);r(x)dx)$ near $a$ and $c$ respectively and satisfying the boundary conditions there (if any). If $u_a$ had a zero $x$ in $(a,c)$, then $y\mapsto G_{(a,c),\lambda}(y,y)$ would change sign there (note that $u_c$ is nonzero in $x$ since otherwise $\lambda$ would be an eigenvalue of $S_{(a,c)}$). Hence $u_a$ cannot have a zero in $(a,c)$ which shows that $\tau-\lambda$ is non-oscillatory at $a$. 
\end{proof}

\begin{corollary}\lb{c10.13}
Assume Hypothesis \ref{h2.1} and suppose $p>0$ a.e.\ on $(a,b)$.  Then $T_0$ is bounded from below if and only if there exist $\mu\in \bbR$ and functions $g_a, g_b\in AC_{\loc}((a,b))$ such that $g_a^{[1]}, g_b^{[1]}\in AC_{\loc}((a,b))$, $g_a>0$ near $a$, $g_b>0$ near $b$, 
\begin{equation}\lb{10.55}
\begin{split}
q&\geq \mu r-\foco\frac{g_a^{[1]}}{g_a}+\big(g_a^{[1]}\big)'  \, \text{ a.e.\ near $a$},\\
q&\geq \mu r-\foco\frac{g_b^{[1]}}{g_b}+\big(g_b^{[1]}\big)'  \, \text{ a.e.\ near $b$}.
\end{split}
\end{equation}
\end{corollary}
\begin{proof}
We first assume in addition that 
\begin{equation}\lb{10.54}
\int_a\frac{dx}{p(x)g_a(x)^2}=\int^b\frac{dx}{p(x)g_b(x)^2}=\infty. 
\end{equation}
Then for the necessity part of the corollary, Theorem \ref{t10.12} permits one to choose $g_a$ and 
$g_b$ as principal solutions of $(\tau-\mu)u=0$ at $a$ and $b$, respectively, for $\mu$ less than a lower bound of $T_0$.  For the sufficiency part, one replaces $\lambda_a$ by $\mu$, ``$=$'' by ``$\geq$'', 
and $f_a$ by $g_a$ in \eqref{10.31} and \eqref{10.32}.  The endpoint $b$ is handled analogously. 

As originally pointed out in \cite[Sect.\ 3]{Ka78} in the context of traditional Sturm--Liouville 
operators (i.e., those 
without distributional potentials), one may replace condition \eqref{10.54} by the condition that one (resp., both) of the integrals appearing in \eqref{10.54} is (resp., are) convergent.  Indeed, the sufficiency 
proof of Corollary \ref{c10.13} is carried out independent of the condition in \eqref{10.54}.  For necessity, Theorem \ref{t10.12} permits one to choose $g_a$ or $g_b$ as a non-principal solution, yielding 
equality in \eqref{10.55}. 
\end{proof}

\begin{definition} \lb{d10.15} 
Assume Hypothesis \ref{h2.1} and let $\lambda \in \bbR$.  Two points $x_1,x_2\in (a,b)$, $x_1\neq x_2$, are called {\it conjugate points with respect to $\tau-\lambda$} if there is some nontrivial, real-valued solution $u$ of $(\tau-\lambda)u=0$ satisfying $u(x_1)=u(x_2)=0$.  If no pair of conjugate points with respect to $\tau-\lambda$ exists, then the differential expression $\tau-\lambda$ is called {\it disconjugate}.
\end{definition}

The disconjugacy property has been extensively studied for Sturm--Liouville expressions with 
standard $L^1_{\loc}$-coefficients, and in this connection we refer to the monograph by 
Coppel \cite{Co71}.  The proof of Theorem \ref{t10.12} immediately yields the following 
disconjugacy result for the distributional Sturm--Liouville expressions studied throughout this 
manuscript. 

\begin{corollary} \lb{c10.16} 
Assume Hypothesis \ref{h2.1}, and suppose $p>0$ a.e.\ on $(a,b)$.  If $T_0$ is bounded from below, then there is an $\alpha\in\R$ such that $(\tau -\lambda)$ is disconjugate for every $\lambda<\alpha$.  If $\tau$ is regular on $(a,b)$, then there exists a $\alpha_0\in \bbR$, such that for $\lambda<\alpha_0$, each 
nontrivial solution to $(\tau-\lambda)u=0$ has at most one zero in the closed interval $[a,b]$.
\end{corollary}
\begin{proof} 
Repeating the proof of Theorem \ref{t10.12} with $c=b$ shows that there is an $\alpha\in\R$ such that for each $\lambda<\alpha$ there is a solution of $(\tau-\lambda)u=0$ which has no zero in $(a,b)$. Now the claim follows immediately from Theorem \ref{t10.1}.  To prove the final statement, let $\alpha$ denote a real number (shown to exist in the first part of the corollary) such that for every $\lambda<\alpha$ there is a solution of $(\tau-\lambda)u=0$ which has no zeros in $(a,b)$.  Now, let $\alpha_0=\min\{ \alpha, \inf(\sigma(S_{0,0}))\}$, where $S_{0,0}$ denotes the Dirichlet extension of $\Tmin$ defined by \eqref{eqn:SRLCLCsep} with $\varphi_a=\varphi_b=0$ and the functionals $BC_a^1$ and $BC_b^1$ chosen such that (cf.\ Lemma \ref{prop:PointEvalBC})
\begin{equation}
BC_a^1(g)=g(a), \quad BC_b^1(g)=g(b), \quad g\in \text{dom}\big(\Tmax \big).
\end{equation}
If for some $\lambda<\lambda_{\min}$ a solution to $(\tau-\lambda)u=0$, call it $u_0$, has more than one zero, then necessarily $u_0(a)=u_0(b)=0$, as $u$ has no zeros in $(a,b)$ because $\lambda<\alpha$.  Consequently, $u_0$ is an eigenfunction of $S_{0,0}$ with eigenvalue $\lambda<\inf \sigma\big(S_{0,0}\big)$, an obvious contradiction. 
\end{proof}

We conclude this section with an explicit characterization of the Friedrichs extension \cite{Fr34} of $T_0$ 
(assuming the latter to be bounded from below). Before proceeding with this characterization, we recall the 
intrinsic description of the Friedrichs extension $S_F$ of a densely defined, symmetric operator $S_0$ in a complex, separable Hilbert space $\cH$ (with scalar product denoted by $(\cdot, \cdot)_{\cH}$), bounded from below, due to Freudenthal \cite{Fr36} in 1936. Assuming that $S_0 \geq \gamma_{{}_{S_0}} I_{\cH}$,  Freudenthal's characterization  describes $S_F$ by    
\begin{align}
& S_F u=S_0^*u,   \no \\
& u \in \dom{S_F} = \Big\{v\in\dom{S_0^*}\,\Big|\,  \mbox{there exists} \, 
\{v_j\}_{j\in\bbN}\subset \dom{S_0},    \label{Fr-2} \\
& \hspace*{4.5mm} \mbox{with} \, \lim_{j\to\infty}\|v_j-v\|_{\cH}=0  
\mbox{ and } ((v_j-v_k),S_0 (v_j-v_k))_\cH 
\underset{ j,k\to\infty}{\longrightarrow} 0\Big\}.     \no 
\end{align}
Then, as is well-known,  
\begin{align}
& S_F \geq \gamma_{{}_{S_0}} I_{\cH},  \label{Fr-4}  \\
& {\rm dom} \big((S_F-\gamma_{{}_{S_0}} I_{\cH})^{1/2}\big)=\Big\{v\in\cH\,\Big|\, \mbox{there exists} \, 
\{v_j\}_{j\in\bbN}\subset \dom{S_0},   \label{Fr-4J} \\
& \hspace*{7mm} \mbox{with} \lim_{j\to\infty}\|v_j-v\|_{\cH}=0  
\mbox{ and } ((v_j-v_k),S_0(v_j-v_k))_\cH 
\underset{ j,k\to\infty}{\longrightarrow} 0\Big\},  \no 
\end{align}
and
\begin{equation}\label{Fr-4H}
S_F=S_0^*|_{\dom{S_0^*}\cap {\rm dom} ((S_F - \gamma_{{}_{S_0}} I_{\cH})^{1/2})}.
\end{equation}

Equations \eqref{Fr-4J} and \eqref{Fr-4H} are intimately related to the definition of $S_F$ via (the closure of) the sesquilinear form generated by $S_0$ as follows: One introduces the sesquilinear form
\begin{equation}
\gq_{S_0}(f,g)=(f,S_0g)_{\cH}, \quad f, g \in \dom{\gq_{S_0}}=\dom{S_0}. 
\end{equation}
Since $S_0\geq \gamma_{{}_{S_0}} I_{\cH}$, the form $q_{S_0}$ is closable and we denote by 
$\ol{\gq_{S_0}}$ the closure of $\gq_{S_0}$. Then $\ol{\gq_{S_0}} \geq \gamma_{{}_{S_0}}$ is densely defined and closed. By the first and second representation theorem for forms (cf., e.g., \cite[Sect.\ 6.2]{Ka80}), $\ol{\gq_{S_0}}$ is uniquely associated with a self-adjoint operator in $\cH$. This operator is precisely the Friedrichs extension, $S_F \geq \gamma_{{}_{S_0}} I_{\cH}$, of $S_0$, and hence,
\begin{equation}
\ol{\gq_{S_0}} (f,g)=(f,S_F g)_{\cH}, \quad f \in \dom{\ol{\gq_{S_0}}} 
= {\rm dom} \big((S_F- \gamma_{{}_{S_0}} I_{\cH})^{1/2}\big), \, 
g \in \dom{S_F}.  \lb{Fr-Q}     
\end{equation}

The following result describes the Friedrichs extension of $T_0$ (assumed to be bounded from below) in terms of functions that mimic the behavior of principal solutions near an endpoint. The proof closely follows  
the treatment by Kalf \cite{Ka78} in the special case $\foco=0$ a.e.\ on $(a,b)$.\ (For more recent results on the Friedrichs extension of ordinary differential operators we also refer to \cite{MZ00}, \cite{MZ95}, \cite{MZ95a}, \cite{NZ90}, \cite{NZ92}, \cite{Ro85}, and \cite{Ze98}.)  

\begin{theorem}\lb{t10.17}
Assume Hypothesis \ref{h2.1} and suppose $p>0$ a.e.\ on $(a,b)$.  If $T_0$ is bounded from below by $\gamma_0 \in \bbR$, $T_0 \geq \gamma_0 I_r$, which by Corollary \ref{c10.13} is equivalent to the existence of $\mu\in \bbR$ and functions $g_a$ and $g_b$ satisfying $g_a,g_b,g_a^{[1]},g_b^{[1]}\in AC_{\loc}((a,b))$, $g_a>0$ a.e.\ near $a$, $g_b>0$ a.e.\ near $b$,
\begin{equation}\lb{10.56}
\int_a\frac{dx}{p(x)g_a(x)^2}=\int^b\frac{dx}{p(x)g_b(x)^2}=\infty,
\end{equation}
and
\begin{equation}\lb{10.57}
\begin{split}
q&\geq \mu r-\foco\frac{g_a^{[1]}}{g_a}+\frac{\big(g_a^{[1]}\big)'}{g_a} \, \text{ a.e.\ near $a$},\\
q&\geq \mu r-\foco\frac{g_b^{[1]}}{g_b}+\frac{\big(g_b^{[1]}\big)'}{g_b} \, \text{ a.e.\ near $b$},
\end{split}
\end{equation}
then the Friedrichs extension $S_F$ of $T_0$ is characterized by
\begin{align}
&S_Ff=\tau f,\no\\
&f\in \dom{S_F}=\bigg\{g\in \dom{\Tmax} \, \bigg|\, \int_a  pg_a^2\bigg|\bigg(\frac{g}{g_a}\bigg)' \bigg|^2 dx<\infty, \lb{10.58}\\
&\hspace*{6.6cm} \int^b pg_b^2\bigg|\bigg(\frac{g}{g_b}\bigg)' \bigg|^2 dx <\infty  \bigg\}.\no
\end{align}
In particular, 
\begin{equation}\lb{10.59a}
\begin{split}
\int_a \bigg|q-\frac{\big(g_a^{[1]}\big)'}{g_a}+\foco\frac{g_a^{[1]}}{g_a} \bigg||f|^2 dx <\infty, \, \int^b \bigg|q-\frac{\big(g_b^{[1]}\big)'}{g_b}+\foco\frac{g_b^{[1]}}{g_b} \bigg| |f|^2 dx <\infty,&\\
 f\in \dom {S_F}.&
\end{split}
\end{equation}
\end{theorem}
\begin{proof}
Let $S$ denote the operator defined by \eqref{10.58} and $S_F$ the Friedrichs extension of $T_0$.  We begin by showing $S$ is symmetric.  In order to do this, it suffices to prove $S$ is densely defined and 
\begin{equation}\lb{10.59}
\langle u,Su\rangle_r\in \bbR,\quad u\in \dom S.
\end{equation}  
Since functions in $\dom{T_0}$ are compactly supported one has $\dom{T_0}\subset \dom{S}$, which guarantees that $S$ is densely defined. Hence it remains to show \eqref{10.59}.  
 To this end, let $a<c_0<d_0<b$ such that $g_a>0$ on $(a,c_0]$, $g_b>0$ on $[d_0,b)$ and consider the self-adjoint operator $S_{(c_0,d_0)}$ on $L^2((c_0,d_0);r(x)dx)$ induced by $\tau$ with the boundary conditions
\begin{align}
 f(c_0)g_a^\qd(c_0) - f^\qd(c_0) g_a(c_0) = f(d_0) g_b^\qd(d_0) - f^\qd(d_0) g_b(d_0) = 0.
\end{align} 
 The proof of Theorem \ref{t10.12} shows that the solutions $u_\lambda$ of $(\tau-\lambda)u=0$, 
 $\lambda\in\R$, satisfying the initial conditions $u_\lambda(c_0) = g_a(c_0)$ and 
 $u_\lambda^\qd(c_0) = g_a^\qd(c_0)$, 
 are positive as long as $\lambda$ lies below the smallest eigenvalue $\lambda_0$ of 
 $S_{(c_0,d_0)}$ (which is bounded from below by assumption).
 In particular, this guarantees that the eigenfunction $u_{\lambda_0}$ is nonnegative on $[c_0,d_0]$ and hence even positive since it would change sign at a zero. 
 As a consequence, the function $h$ defined by
 \begin{align}
  h(x) = \begin{cases}
    g_a(x), & x\in(a,c_0), \\
    u_{\lambda_0}(x), & x\in[c_0,d_0], \\
    u_{\lambda_0}(d_0)g_b(d_0)^{-1} g_b(x), & x\in(d_0,b)
  \end{cases}
 \end{align} 
 is positive on $(a,b)$ and satisfies $h\in AC_{\loc}((a,b))$, $h^{[1]}\in AC_{\loc}((a,b))$. Note that in particular $h$ is a scalar multiple of $g_b$ near $b$ and hence \eqref{10.56} and \eqref{10.57} hold with $g_b$ replaced by $h$. 
 Now fix some $f\in \dom S$ and let $a<c<d<b$.  In light of the following analog of Jacobi's factorization identity,
\begin{equation}\lb{10.61}
-\big(f^{[1]} \big)'+\foco f^{[1]}+\frac{(h^{[1]})'}{h}f-\foco\frac{h^{[1]}}{h}f=-\frac{1}{h}\bigg[ph^2\bigg(\frac{f}{h} \bigg)'\bigg]'  \, \text{ a.e.\ in $(a,b)$},
\end{equation}
one computes
\begin{align}
&\int_c^d f(x) \, \ol{Sf(x)} r(x) dx \lb{10.62}\\
&\quad=-\bigg[ph^2\bigg(\ol{\frac{f}{h}\bigg)'}\frac{{f}}{h}\bigg]\bigg|_c^d+\int_c^d \bigg\{ph^2\bigg|\bigg(\frac{f}{h} \bigg)' \bigg|^2+|f|^2\bigg(q-\frac{\big(h^{[1]} \big)'}{h}+\foco \frac{h^{[1]}}{h} \bigg) \bigg\} dx,\no\\
&\hspace*{9.5cm}a<c<d<b,\no
\end{align}
so that 
\begin{equation}\lb{10.63}
\Im\bigg(\int_c^d f(x) \, \ol{Sf(x)}  r(x)dx \bigg)
=\Im\bigg(-\bigg[ph^2\ol{\bigg(\frac{f}{h}\bigg)'}\frac{{f}}{h}\bigg]\bigg|_c^d \bigg),\quad a<c<d<b.
\end{equation} 
Taking $P=ph^2$ and $v=f/h$ in the subsequent Lemma \ref{lKalf}, one infers that
\begin{equation}\lb{RC1}
\int_a\frac{|v(x)|^2}{P(x) H_{\gamma}(x)}dx<\infty,\quad \gamma \in (a,b),
\end{equation}
where the function $H_{\gamma}$ is defined as in \eqref{RC-1}. We note that $H_{\gamma}$ is 
well-defined for any $\gamma\in (a,b)$ in light of the fact that $1/p\in L^1_{\loc}((a,b);dx)$ and the 
function $h\in AC_{\loc}((a,b))$ is strictly positive on any compact subinterval of $(a,b)$. 
Subsequently, an application of H\"older's inequality yields
\begin{equation}\lb{RC3}
\int_a\frac{P(x)\big|\ol{v(x)}v'(x)\big|}{P(x) H_{\gamma}(x)}dx <\infty,\quad \gamma\in (a,b),
\end{equation}
noting that square integrability of $P^{1/2} v'$ near $x=a$ is guaranteed by the condition $f\in \dom S$.  Moreover, the integral
\begin{equation}\lb{RC4}
\int_a\frac{dx}{P(x) H_{\gamma}(x)},\quad \gamma\in (a,b),
\end{equation}
diverges logarithmically to infinity, so \eqref{RC3} implies
\begin{equation}\lb{RC5}
\liminf_{x\downarrow a}\big|P\ol{v}v'\big|(x)
=\liminf_{x\downarrow a}\bigg|ph^2\ol{\bigg(\frac{f}{h}\bigg)'}\frac{{f}}{h}\bigg|(x)=0.
\end{equation}
An analogous argument at $x=b$ can be used to show
\begin{equation}\lb{10.64}
\liminf_{x\uparrow b}\bigg|ph^2\ol{\bigg(\frac{f}{h}\bigg)'}\frac{{f}}{h}\bigg|(x)=0.
\end{equation}
Equations \eqref{10.63}, \eqref{RC5}, and \eqref{10.64} show that one can choose sequences 
$\{c_n\}_{n\in\bbN}$ and $\{d_n\}_{n\in\bbN}$ with $a<c_n<d_n<b$, $n\in \bbN$, with 
$c_n\downarrow a$, $d_n\uparrow b$, such that
\begin{equation}\lb{10.65}
\lim_{n\rightarrow \infty}\Im\bigg(\int_{c_n}^{d_n} f(x) \, \ol{Sf(x)}  r(x) dx \bigg)=0.
\end{equation}
On the other hand
\begin{equation}\lb{10.66}
\lim_{\substack{c\downarrow a \\ d\uparrow b}}\Im\bigg(\int_{c}^{d} f(x) \, \ol{Sf(x)}  r(x) dx \bigg)
\end{equation}
exists.  Consequently, \eqref{10.65} implies
\begin{equation}\lb{10.67}
\Im\bigg(\int_{a}^{b} f(x) \, \ol{Sf(x)}  r(x) dx \bigg)=\lim_{\substack{c\downarrow a \\ 
d\uparrow b}} \Im\bigg(\int_{c}^{d} f(x) \, \ol{Sf(x)}  r(x) dx \bigg)=0.
\end{equation}
Since $f\in \dom S$ was arbitrary, \eqref{10.59} follows.

We now show that $S$ coincides with $S_F$, the Friedrichs extension of $T_0$.  It suffices to show $S_F\subset S$; self-adjointness of $S_F$ and symmetry of $S$ then yield $S_F = S$.  In turn, since $S_F$ is a restriction of $T_{\max}$ (because the self-adjoint extensions of $T_0$ are precisely the self-adjoint extensions of $T_{\min}$, and the latter are self-adjoint restrictions of $T_{\max}$), it suffices to verify the two integral conditions appearing in \eqref{10.58} are satisfied for elements of $\dom {S_F}$.  Freudenthal's characterization of the domain of the Friedrichs extension for the present setting is
\begin{align}
&\dom {S_F} =\Big\{f\in \dom{\Tmax} \, \Big| \, \text{there exists $\{f_j\}_{j=1}^{\infty}\subset \dom{T_0}$} \text{ such} \lb{10.68}\\
&\hspace*{2.1cm} \text{that $\lim_{j\rightarrow \infty}\|f_j-f\|_{2,r}=0$ and 
$\lim_{j,k\rightarrow \infty}\langle f_j-f_k, T_0(f_j-f_k) \rangle_r=0$}\Big\}.\no 
\end{align}
Let $f\in \dom{S_F}$ and $\{f_j\}_{j=1}^{\infty}$ a sequence with the properties in \eqref{10.68}.  Define $f_{j,k}=f_j-f_k$, $j,k\in \bbN$, and choose numbers $c$ and $d$ in the interval $(a,b)$ such that $g_a$ and $g_b$ are positive on $(a,c]$ and $[d,b)$, respectively.  Then using the identities
\begin{align}
&\int_{\alpha}^{c} \big\{p^{-1} \big|u^{[1]}\big|^2+q|u|^2\big\} dx\lb{10.69}\\
&\quad =\frac{g_a^{[1]}}{g_a}|u|^2\bigg|_{\alpha}^c+\int_{\alpha}^c \bigg\{pg_a^2\bigg|\bigg(\frac{u}{g_a}\bigg)' \bigg|^2+|u|^2\bigg[q+\foco\frac{g_a^{[1]}}{g_a}-\frac{\big(g_a^{[1]}\big)'}{g_a}\bigg] \bigg\}dx,\quad \alpha \in [a,c],\no\\
&\int_d^{\beta} \big\{p^{-1} \big|u^{[1]}\big|^2+q|u|^2\big\}dx \lb{10.70}\\
&\quad =\frac{g_b^{[1]}}{g_b}|u|^2\bigg|_d^{\beta}+\int_d^{\beta} \bigg\{pg_b^2\bigg|\bigg(\frac{u}{g_b}\bigg)' \bigg|^2+|u|^2\bigg[q+\foco\frac{g_b^{[1]}}{g_b}-\frac{\big(g_b^{[1]}\big)'}{g_b}\bigg] \bigg\}dx,\quad \beta \in [d,b],\no\\
&\hspace*{10.2cm}u\in \dom{T_0},\no
\end{align}
one computes 
\begin{align}
\langle f_{j,k},T_0f_{j,k}\rangle_r&=\int_a^c \bigg\{pg_a^2\bigg|\bigg(\frac{f_{j,k}}{g_a}\bigg)' \bigg|^2+|f_{j,k}|^2\bigg[q+\foco\frac{g_a^{[1]}}{g_a}-\frac{\big(g_a^{[1]}\big)'}{g_a}\bigg] \bigg\}dx\no\\
&\quad+\int_d^{b} \bigg\{pg_b^2\bigg|\bigg(\frac{f_{j,k}}{g_b}\bigg)' \bigg|^2+|f_{j,k}|^2\bigg[q+\foco\frac{g_b^{[1]}}{g_b}-\frac{\big(g_b^{[1]}\big)'}{g_b}\bigg] \bigg\}dx\no\\
&\quad +\bigg(\frac{g_a^{[1]}}{g_a}|f_{j,k}|^2\bigg)(c)-\bigg(\frac{g_b^{[1]}}{g_b}|f_{j,k}|^2\bigg)(d)\no\\
&\quad +\int_c^d \big\{p^{-1}\big|f_{j,k}^{[1]}\big|^2+q|f_{j,k}|^2\big\}dx, \quad j,k\in \bbN.\lb{10.71}
\end{align}
On the other hand, choosing $\nu\in \bbR$ such that
\begin{align}
\begin{split} 
\nu \int_c^d  r\big|f_{j,k}\big|^2 dx &\leq\bigg(\frac{g_a^{[1]}}{g_a}|f_{j,k}|^2\bigg)(c)-\bigg(\frac{g_b^{[1]}}{g_b}|f_{j,k}|^2\bigg)(d)\lb{10.72}\\
&\quad +\int_c^d \big\{p^{-1} \big|f_{j,k}^{[1]}\big|^2+q|f_{j,k}|^2\big\}dx, \quad j,k\in \bbN, 
\end{split} 
\end{align}
\noindent
the existence of such a $\nu$ being guaranteed by Lemma \ref{lA.3} (cf., in particular, \eqref{A.15}),  
and taking $\kappa=|\mu|+|\nu|$, one obtains
\begin{equation}\lb{10.73}
\begin{split}
\langle f_{j,k},T_0f_{j,k}\rangle_r+\kappa \big\|f_{j,k} \big\|_{2,r}^2\geq \int_a^c pg_a^2\bigg|\bigg(\frac{f_{j,k}}{g_a}\bigg)' \bigg|^2 dx +\int_d^b pg_b^2\bigg|\bigg(\frac{f_{j,k}}{g_b}\bigg)' \bigg|^2 dx,&\\
j,k\in \bbN.\\
\end{split}
\end{equation}
Moreover, the left-hand side of \eqref{10.73} goes to zero as $j,k\rightarrow \infty$, and as a result, there exist functions $f_a$ and $f_b$ such that
\begin{align}
\lim_{j\rightarrow \infty}\int_a^c pg_a^2\bigg|\bigg(\frac{f_{j}}{g_a}\bigg)'-f_a \bigg|^2 dx=\lim_{j\rightarrow \infty }\int_d^b pg_b^2\bigg|\bigg(\frac{f_{j}}{g_b}\bigg)' -f_b\bigg|^2 dx =0,\lb{10.74}
\end{align}
implying, $f_a=(g_a^{-1}f)'$, $f_b=(g_b^{-1}f)'$ a.e.\ on $(a,c)$ and $(d,b)$, respectively.  Consequently, one infers that
\begin{equation}\lb{10.75}
\int_a pg_a^2\bigg|\bigg(\frac{f}{g_a}\bigg)' \bigg|^2 dx <\infty,\,\quad  \int^b  pg_b^2\bigg|\bigg(\frac{f}{g_b}\bigg)' \bigg|^2 dx <\infty, 
\end{equation}
and it follows that $f\in \dom S$.  This completes the proof that $S_F\subseteq S$ and hence, $S_F=S$.

To prove \eqref{10.59a}, note that in light of the inequalities in \eqref{10.57}, it suffices to prove that the positive part of $\big[q-\big(h^{[1]}\big)'/h+\foco h^{[1]}/h\big]$ times $|f|^2$ is integrable near $a$ and $b$ for each $f\in \dom{S_F}$.  This follows immediately from \eqref{10.62} and \eqref{10.64}.
\end{proof}

The proof of Theorem \ref{t10.17} relied on the following result:

\begin{lemma}\lb{lKalf} $($\cite[Lemma~1]{KW72}, \cite{Ka78}$)$
Let $P>0$, $1/P\in L^1_{\loc}((a,b);dx)$, and 
\begin{equation}\lb{RC-1}
H_{\gamma}(x)=\bigg|\int_{\gamma}^x\frac{dt}{P(t)} \bigg|,\quad x\in (a,b),\, \gamma\in [a,b].
\end{equation}
In addition, suppose that $v\in AC_{\loc}((a,b))$ satisfies $P^{1/2}v'\in L^2((a,b);dx)$.  If $H_a=\infty$, then
\begin{equation}\lb{RC0}
\int_a\frac{|v(x)|^2}{P(x)H_{\gamma}^2(x)}dx<\infty,\quad \gamma \in (a,b),
\end{equation}
the choice $\gamma=b$ being also possible if $H_b<\infty$.\\
\end{lemma}

The conditions on $g_a$ and $g_b$ in \eqref{10.56} are reminiscent of the integral conditions satisfied by principal solutions to the equation $(\tau-\lambda)u=0$, assuming the latter is non-oscillatory.  One can just as well characterize the Friedrichs extension of $T_0$ in terms of functions $g_a$ and $g_b$ satisfying the assumptions of Theorem \ref{t10.17} but for which one (or both) of the integrals in \eqref{10.56} is 
convergent (these conditions are equivalent to $T_0$ being bounded from below, see the proof of 
Corollary \ref{c10.13}).  In these cases, the characterization requires a certain boundary condition 
as our next result shows.  

\begin{theorem} \lb{t10.18} 
Assume Hypothesis \ref{h2.1} and suppose $p>0$ a.e.\ on $(a,b)$.  If $T_0$ is bounded from below by $\gamma_0 \in \bbR$, $T_0 \geq \gamma_0 I_r$, which by Corollary \ref{c10.13} is equivalent to the existence of $\mu\in \bbR$ and functions $g_a$ and $g_b$ satisfying $g_a,g_b,g_a^{[1]},g_b^{[1]}\in AC_{\loc}((a,b))$, $g_a>0$ a.e.\ near $a$, $g_b>0$ a.e.\ near $b$,
\begin{equation}\lb{10.77}
\int_a\frac{dx}{p(x)g_a(x)^2}<\infty, \quad \int^b\frac{dx}{p(x)g_b(x)^2}=\infty,
\end{equation}
and
\begin{equation}\lb{10.78}
\begin{split}
q&\geq \mu r-\foco\frac{g_a^{[1]}}{g_a}+\frac{\big(g_a^{[1]}\big)'}{g_a}  \, \text{ a.e.\ near $a$},\\
q&\geq \mu r-\foco\frac{g_b^{[1]}}{g_b}+\frac{\big(g_b^{[1]}\big)'}{g_b}  \, \text{ a.e.\ near $b$},
\end{split}
\end{equation}
then the Friedrichs extension $S_F$ of $T_0$ is characterized by
\begin{align}
&S_Ff=\tau f,\no\\
&f\in \dom{S_F}=\bigg\{g\in \dom{\Tmax} \, \bigg|\, \int^b  pg_b^2\bigg|\bigg(\frac{g}{g_b}\bigg)' \bigg|^2 dx <\infty,\lb{10.79}\\
&\hspace*{4.1cm} \int_a  pg_a^2\bigg|\bigg(\frac{g}{g_a}\bigg)' \bigg|^2 dx <\infty,\, \lim_{x\downarrow a}\frac{g(x)}{g_a(x)}=0\bigg\}.\no
\end{align}
In particular,
\begin{equation}\lb{10.79a}
\begin{split}
\int_a \bigg|q-\frac{\big(g_a^{[1]}\big)'}{g_a}+\foco\frac{g_a^{[1]}}{g_a} \bigg||f|^2 dx <\infty, \, \int^b \bigg|q-\frac{\big(g_b^{[1]}\big)'}{g_b}+\foco\frac{g_b^{[1]}}{g_b} \bigg| |f|^2 dx<\infty,&\\
 f\in \dom {S_F}.&
\end{split}
\end{equation}
We omit the obvious case where the roles of $a$ and $b$ are interchanged, but note that if \eqref{10.77} is replaced by 
\begin{equation}\lb{10.79b}
\int_a\frac{dx}{p(x)g_a(x)^2}<\infty, \quad \int^b\frac{dx}{p(x)g_b(x)^2}<\infty,
\end{equation}
one obtains 
\begin{align}
&S_Ff=\tau f,  \no \\
&f\in \dom{S_F}=\bigg\{g\in \dom{\Tmax} \, \bigg|\, \int_a  pg_a^2\bigg|\bigg(\frac{g}{g_a}\bigg)' \bigg|^2 dx <\infty,    \lb{10.79c} \\ 
& \hspace*{2.1cm}  \int^b pg_b^2\bigg|\bigg(\frac{g}{g_b}\bigg)' \bigg|^2 dx<\infty, \, \lim_{x\downarrow a}\frac{g(x)}{g_a(x)}=0, \, 
\lim_{x\uparrow b}\frac{g(x)}{g_b(x)}=0 \bigg\}.\no
\end{align}
\end{theorem}
\begin{proof}
Let $S$ denote the operator defined by \eqref{10.79} and $S_F$ the Friedrichs extension of $T_0$.  To show that $S$ is symmetric, one can follow line-by-line the argument for \eqref{10.59}--\eqref{10.63}, so that \eqref{10.63} remains valid.  One can then show that \eqref{10.64} continues to hold under the finiteness assumption in \eqref{10.77} (cf., the beginning of the proof of \cite[Remark 3]{Ka78}).  Repeating the argument \eqref{10.65}--\eqref{10.67} then shows that $S$ is symmetric.  In order to conclude $S=S_F$, it suffices to prove $S_F\subseteq S$.  In turn, it is enough to prove $\dom{S_F}\subseteq \dom S$.  To this end, let $f\in \dom{S_F}$.  Since \eqref{10.68}--\eqref{10.75} can be repeated without alteration, the problem reduces to proving
\begin{equation}\lb{10.80}
\lim_{x\downarrow a}\frac{|f(x)|}{g_a(x)}=0.
\end{equation}
One takes a sequence $\{f_n\}_{n=1}^{\infty}\subset \dom{T_0}$ with the properties 
\begin{equation}
\text{$\lim_{n\rightarrow \infty}\|f_n-f\|_{2,r}=0$ and 
$\lim_{n,m\rightarrow \infty}\langle f_n-f_m, T_0(f_n-f_m) \rangle_r=0$},  \lb{10.90}
\end{equation}
and let $\{f_{n_k}\}_{k=1}^{\infty}$ denote a subsequence converging to $f$ pointwise a.e.\ in $(a,b)$ as $k\to\infty$. Since $f_{n_k}, f$ are continuous on $(a,b)$, $f_{n_k}$ actually converge pointwise everywhere to $f$ on $(a,b)$ as $k\to\infty$. 

Then the proof of \eqref{10.79a} is exactly the same as the corresponding fact \eqref{10.59a} in Theorem \ref{t10.17}.

Next, one chooses $c\in (a,b)$ such that $g_a>0$ on $(a,c)$. Using H\"older's inequality and \eqref{10.73}, one obtains the estimate
\begin{align}
& \bigg|\frac{f_{n_k}(x)}{g_a(x)}\bigg|^2=\bigg|\int_a^x\frac{1}{p^{1/2}g_a}p^{1/2}g_a \bigg(\frac{f_{n_k}}{g_a}\bigg)' dx'\bigg|^2\leq \int_{a}^x\frac{dx'}{pg_a^2}\int_a^x pg_a^2\bigg|\bigg(\frac{f_{n_k}}{g_a^2} \bigg)' \bigg|^2 dx',  \no \\
& \quad \leq \int_{a}^x\frac{dx'}{pg_a^2} \, 
\Big[(f_{n_k},(T_0 - \gamma_0 I_r) f_{n_k})_{r}  + (|\gamma_0| + \kappa)\|f_{n_k}\big\|^2_{2,r}\Big], 
\quad x\in (a,c),\; k\in \bbN. 
\end{align}
Because of \eqref{10.90}, one obtains
\begin{equation}
\bigg|\frac{f_{n_k}(x)}{g_a(x)}\bigg|^2 \leq C \int_{a}^x\frac{dx'}{pg_a^2}, 
\quad x \in (a, c), \; k \in \bbN, 
\end{equation} 
with $C>0$ a $k$-independent constant. Writing
\begin{equation}
\bigg|\frac{f(x)}{g_a(x)}\bigg| \leq \bigg|\frac{f(x) - f_{n_k}(x)}{g_a(x)}\bigg| + 
\bigg|\frac{f_{n_k}(x)}{g_a(x)}\bigg|,
\end{equation}
and given $\varepsilon > 0$, one first chooses an $x(\varepsilon) \in (a,c)$ such that 
$|f_{n_k}(x)/g_a(x)| \leq \varepsilon/2$ for all $x \in (a,x(\varepsilon))$, and then 
for $x \in (a,x(\varepsilon))$ one chooses a $k(x,\varepsilon) \in \bbN$ such that
for all $k \geq k(x,\varepsilon)$, $|[f(x) - f_{n_k}(x)]/g_a(x)| \leq \varepsilon/2$, resulting in 
\begin{equation}
\bigg|\frac{f(x)}{g_a(x)}\bigg| \leq \varepsilon   \lb{10.94}
\end{equation}
whenever $x \in (a,x(\varepsilon))$ and $k \geq k(x,\varepsilon)$. Since the left-hand side of \eqref{10.94} is $k$-independent, \eqref{10.80} follows.
\end{proof}

\begin{corollary} \lb{c10.19} 
Assume Hypothesis \ref{h2.1} and suppose $p>0$ a.e.\ on $(a,b)$.  If $\tau$ is 
regular on $(a,b)$, then the Friedrichs extension $S_F$ of $T_0$ is of the form  
\begin{align}
\begin{split} 
&S_Ff=\tau f,   \\
&f\in \dom{S_F}=\left\{g\in \dom{\Tmax} \, \big|\, g(a) = g(b) =0 \right\}. \lb{10.101}
\end{split} 
\end{align}
\end{corollary}
\begin{proof}
Let $g_a$, $g_b$ be the solutions of $\tau u = 0$ with the initial conditions $g_a(a) = g_b(b) = 1$ and $g_a^\qd(a) = g_b^\qd(b) = 0$.
Since $\tau$ is regular on $(a,b)$ we have for each $g\in\dom{\Tmax}$ 
\begin{align}
\int_a pg_a^2\bigg|\bigg(\frac{g}{g_a}\bigg)' \bigg|^2 dx 
= \int_a pg_a^2\bigg|\frac{g_a g' - g g_a'}{g_a^2} \bigg|^2 dx  
= \int_a \frac{1}{p} \bigg|\frac{g^{[1]} g_a - g g_a^{[1]}}{g_a^2}\bigg| dx
<\infty,
\end{align}
and similarly for the endpoint $b$. 
Now the result follows from Theorem \ref{t10.18} and in particular \eqref{10.79c}. 
\end{proof}

\section{The Krein--von Neumann Extension in the Regular Case} \lb{s12}

In this section, we consider the Krein--von Neumann extension $S_K$ of $T_0 \geq \varepsilon I_r$, 
$\varepsilon > 0$. The operator $S_K$, like the Friedrichs extension $S_F$ of $T_0$, is a distinguished, in fact, extremal nonnegative extension of $T_0$. 

Temporarily returning to the abstract considerations \eqref{Fr-2}--\eqref{Fr-Q} in connection with the 
Friedrichs extension of $S_0$, an intrinsic description of the Krein--von Neumann extension $S_K$ of 
$S_0\geq 0$ has been given by Ando and Nishio \cite{AN70} in 1970, where $S_K$ has been 
characterized by
\begin{align} 
& S_Ku=S_0^*u,   \no \\
& u \in \dom{S_K}=\big\{v\in\dom{S_0^*}\,\big|\,\mbox{there exists} \, 
\{v_j\}_{j\in\bbN}\subset \dom{S_0},    \label{Fr-2X}  \\ 
& \quad \mbox{with} \, \lim_{j\to\infty} \|S_0 v_j-S_0^* v\|_{\cH}=0  
\mbox{ and } ((v_j-v_k),S_0 (v_j-v_k))_\cH\to 0 \mbox{ as } j,k\to\infty\big\}.  \no
\end{align}

We recall that $A \leq B$ for two self-adjoint operators in $\cH$ if 
\begin{align}
\begin{split}
& {\rm dom}\big(|A|^{1/2}\big) \supseteq {\rm dom}\big(|B|^{1/2}\big) \, \text{ and } \\ 
& \big(|A|^{1/2}u, U_A |A|^{1/2}u\big)_{\cH} \leq \big(|B|^{1/2}u, U_B |B|^{1/2}u\big)_{\cH}, \quad  
u \in {\rm dom}\big(|B|^{1/2}\big),      \lb{AleqB} 
\end{split}
\end{align}
where $U_C$ denotes the partial isometry in $\cH$ in the polar decomposition of 
a densely defined closed operator $C$ in $\cH$, $C=U_C |C|$, $|C|=(C^* C)^{1/2}$.

The following is a fundamental result to be found in M.\ Krein's celebrated 1947 paper
 \cite{Kr47} (cf.\ also Theorems\ 2 and 5--7 in the English summary on page 492): 
 
\begin{theorem}\label{T-kkrr}
Assume that $S_0$ is a densely defined, nonnegative operator in $\cH$. Then, among all 
nonnegative self-adjoint extensions of $S_0$, there exist two distinguished ones, $S_K$ and $S_F$, 
which are, respectively, the smallest and largest
$($in the sense of order between self-adjoint operators, cf.\ \eqref{AleqB}$)$ such extensions. 
Furthermore, a nonnegative self-adjoint operator $\widetilde{S}$ is a self-adjoint extension of 
$S_0$ if and only if $\widetilde{S}$ satisfies 
\begin{equation}\label{Fr-Sa}
S_K\leq\widetilde{S}\leq S_F.
\end{equation}
In particular, \eqref{Fr-Sa} determines $S_K$ and $S_F$ uniquely. \\

In addition,  if $S_0 \geq \varepsilon I_{\cH}$ for some $\varepsilon >0$, one has 
$S_F \geq \varepsilon I_{\cH}$, and 
\begin{align}
\dom{S_F} &= \dom{S_0} \dotplus (S_F)^{-1} \ker (S_0^*),     \lb{SF}  \\
\dom{S_K} & = \dom{S_0} \dotplus \ker (S_0^*),    \lb{SK}   \\
\dom{S^*} & = \dom{S_0} \dotplus (S_F)^{-1} \ker (S_0^*) \dotplus \ker (S_0^*)  \no \\
& = \dom{S_F} \dotplus \ker (S_0^*),    \lb{S*} 
\end{align}
in particular, 
\begin{equation} \label{Fr-4Tf}
\ker(S_K)= \ker\big((S_K)^{1/2}\big)= \ker(S_0^*) = \ran(S_0)^{\bot}.
\end{equation} 
\end{theorem}

Here the symbol $\dotplus$ represents the direct (though, not direct orthogonal) sum of subspaces, 
and the operator inequalities in \eqref{Fr-Sa} are understood in the sense of 
\eqref{AleqB} and hence they can  equivalently be written as
\begin{equation}
(S_F + a I_{\cH})^{-1} \le \big(\wti S + a I_{\cH}\big)^{-1} \le (S_K + a I_{\cH})^{-1} 
\, \text{ for some (and hence for all\,) $a > 0$.}    \lb{Res}
\end{equation}

In addition to Krein's fundamental paper \cite{Kr47}, we refer to the discussions in \cite{AS80}, 
\cite{AT03}, \cite{AT05}, \cite{Gr83}. It should be noted that the Krein--von Neumann extension was first 
considered by von Neumann \cite{Ne29} in 1929 in the case where $S_0$ is strictly positive, that is, if 
$S_0 \geq \varepsilon I_{\cH}$ for some $\varepsilon >0$. (His construction appears in the proof of Theorem 42 on pages 102--103.) However, von Neumann did not isolate the extremal property of this extension as described in \eqref{Fr-Sa} and \eqref{Res}. M.\ Krein \cite{Kr47}, \cite{Kr47a} was the first to systematically treat the general case $S_0 \geq 0$ and to study all nonnegative self-adjoint extensions of $S_0$, 
illustrating the special role of the {\it Friedrichs extension} $S_F$ and the Krein--von Neumann extension 
$S_K$ of $S_0$ as extremal cases when considering all nonnegative extensions of $S_0$. For a recent exhaustive treatment of self-adjoint extensions of semibounded operators we refer to 
\cite{AT02}--\cite{AGMT10}. For classical references on the subject of self-adjoint extensions of 
semibounded operators (not necessarily restricted to the Krein--von Neumann extension) we refer 
to Birman \cite{Bi56}, \cite{Bi08}, Freudenthal \cite{Fr36}, Friedrichs \cite{Fr34}, Grubb \cite{Gr68}, 
\cite{Gr06}, Krein \cite{Kr47a}, {\u S}traus \cite{St73}, and Vi{\u s}ik \cite{Vi63} (see also the monographs 
by Akhiezer and Glazman \cite[Sect.\ 109]{AG81a}, Faris \cite[Part III]{Fa75}, and 
Grubb \cite[Sect.\ 13.2]{Gr09}).  

\medskip

Throughout the remainder of this section, we assume that $\tau$ is regular on $(a,b)$ and that the 
coefficient $p$ is positive a.e.\ on $(a,b)$.  That is, we shall make the following assumptions:  

\begin{hypothesis}\lb{h11.1}
Assume Hypothesis \ref{h2.1} holds with $p>0$ a.e.\ on $(a,b)$ and that $\tau$ is regular on $(a,b)$. 
Equivalently, we suppose that 
$p$, $q$, $r$, $\foco$ are Lebesgue measurable on $(a,b)$ with $p^{-1}$, $q$, $r$, $\foco\in L^1((a,b);dx)$ and real-valued a.e.\ on $(a,b)$ with $p$, $r>0$ a.e.\ on $(a,b)$.
\end{hypothesis}

Assuming Hypothesis \ref{h11.1}, we now provide a characterization of the Krein--von Neumann extension, 
$S_K$ of $T_0$ (resp., $\Tmin$), in the situation where $T_0$ is strictly positive (in the operator sense).  An elucidation along these lines for the case $s=0$ a.e.\ on $(a,b)$ was set forth in \cite{CGNZ12}.

\begin{theorem}\lb{t11.2}
Assume Hypothesis \ref{h11.1} and suppose that the associated minimal operator $\Tmin$ is strictly positive in the sense that there exists $\varepsilon>0$ such that
\begin{equation}\lb{10.147}
\langle \Tmin f, f\rangle_r\geq \varepsilon \langle f, f \rangle_r,\quad f\in {\rm dom}\big(\Tmin\big).
\end{equation}
Then the Krein--von Neumann extension $S_K$ of $\Tmin$ is given by $($cf.\ \eqref{eqn:SRLCLCcoup}$)$
 \begin{align} \lb{SK1}
 \begin{split} 
 & S_K f = \tau f,     \\
& f \in \dom{S_K} = \bigg\{g \in\dom{\Tmax} \, \bigg| \, \begin{pmatrix} g(b) \\ g^{[1]}(b)\end{pmatrix} 
      = R_K \begin{pmatrix} g(a) \\ g^{[1]}(a)\end{pmatrix}\bigg\},    \\
\end{split} 
\end{align} 
where
\begin{equation}\lb{RK}
R_K = \frac{1}{u_1^{[1]}(a)}
\begin{pmatrix}
-u_2^{[1]}(a) & 1 \\
u_1^{[1]}(a)u_2^{[1]}(b)-u_1^{[1]}(b)u_2^{[1]}(a) & u_1^{[1]}(b)
\end{pmatrix} \in \SL_2(\bbR).
\end{equation}
Here $\big\{u_j(\cdot) \big\}_{j=1,2}$ are positive solutions of $\tau u = 0$ determined by the conditions
\begin{equation}\lb{10.150}
\begin{split}
&u_1(a)=0, \quad u_1(b)=1,\\
&u_2(a)=1, \quad u_2(b)=0.
\end{split}
\end{equation} 
\end{theorem} 
\begin{proof}
The assumption that $\Tmin$ is strictly positive implies that $0$ is a regular point of $\Tmin$ (cf.\ the paragraph preceding Lemma \ref{lemWeylRegType}), and since the deficiency indices of $\Tmin$ are 
equal to two (one notes that it is this fact that actually implies the existence of solutions $u_j$, $j=1,2$, satisfying the properties \eqref{10.150}), it follows that 
\begin{equation}\lb{10.148}
\text{dim}\big(\text{ker}\big(\Tmax \big) \big)=2
\end{equation}
and a basis for $\text{ker}\big(\Tmax \big)$ is given by $\big\{u_j(\cdot) \big\}_{j=1,2}$.
In this situation, the Krein--von Neumann extension $S_K$ of $\Tmin$ is given by (cf.\ \eqref{SK}),
\begin{equation}\lb{10.149}
\text{dom}\big(S_K \big)=\text{dom}\big( \Tmin\big) \dotplus \text{ker}\big(\Tmax \big).
\end{equation}
Alternatively, since $S_K$ is a self-adjoint extension of $\Tmin$, its domain can also be specified by boundary conditions at the endpoint of $(a,b)$ which we characterize next.
If $u\in \text{dom}\big( S_K \big)$, then in accordance with \eqref{10.149}, 
\begin{equation}\lb{10.151}
u(x)=f(x)+c_1u_1(x)+c_2u_2(x), \quad x\in [a,b], 
\end{equation}
for certain functions $f\in \text{dom}\big(\Tmin \big)$ and $c_1,c_2\in\C$. 
Since $f\in \text{dom}\big(\Tmin \big)$ satisfies
\begin{equation}\lb{10.152}
f(a)=f^{[1]}(a)=f(b)=f^{[1]}(b)=0,
\end{equation}
one infers that
\begin{equation}\lb{10.153}
u(a)=c_2\quad \text{and}\quad u(b)=c_1.
\end{equation}
Consequently
\begin{equation}\lb{10.157}
u^{[1]}(x)=f^{[1]}(x)+u(b)u_1^{[1]}(x)+u(a)u_2^{[1]}(x),\quad x\in [a,b].
\end{equation}
Evaluating separately at $x=a$ and $x=b$, yields the (non-separated) boundary conditions that $u$ must satisfy;
\begin{equation}\lb{10.158}
\begin{split}
u^{[1]}(a)&=u(b)u_1^{[1]}(a)+u(a)u_2^{[1]}(a),\\
u^{[1]}(b)&=u(b)u_1^{[1]}(b)+u(a)u_2^{[1]}(b).
\end{split}
\end{equation}
Since $u_1^{[1]}(a)\neq 0$ (otherwise, $u_1(\cdot)\equiv 0$ on $[a,b]$), the boundary condition 
in \eqref{10.158} may be recast as
\begin{equation}\lb{10.159}
\begin{pmatrix}
u(b) \\ u^{[1]}(b)
\end{pmatrix}=R_K 
\begin{pmatrix}
u(a) \\ u^{[1]}(a)
\end{pmatrix},
\end{equation}
with $R_K$ given by \eqref{RK}. Moreover, $R_K\in \SL_2(\bbR)$.  To see this, first note that the 
entries of $R_K$ are real-valued.  Additionally, the fact that 
\begin{equation}\lb{10.161}
-u_1^{[1]}(a)=W\big(u_1(\cdot),u_2(\cdot) \big)=u_2^{[1]}(b)
\end{equation}
implies $\text{det}\big(R_K\big)=1$. As a result, we have shown $S_K\subseteq S_{R=R_K,\phi=0}$, where $S_{R=R_K,\phi=0}$ is the self-adjoint restriction of $\Tmax$ corresponding to non-separated boundary conditions generated by the matrix $R_K$ and angle $\phi=0$ (cf.\ \eqref{eqn:SRLCLCcoup}).  On the other hand, since $S_K$ and $S_{R=R_K,\phi=0}$ are self-adjoint, one obtains the equality $S_K=S_{R=R_K,\phi=0}$.  That is to say, the Krein--von Neumann extension of $\Tmin$ is the self-adjoint extension corresponding to non-separated boundary conditions generated by $R=R_K$ and $\phi=0$.
\end{proof}

\begin{example} \lb{e10.3}
In the special case when $q=0$ a.e.\ on $(a,b)$, the above calculations become even more explicit.  
In this case, we denote the Krein--von Neumann restriction by $S_K^{(0)}$ (the superscript $(0)$ 
indicating that $q$ vanishes a.e.\ in $(a,b)$).  One may choose explicit basis vectors 
$\big\{u_j^{(0)}(\cdot)\big\}_{j=1,2}$ for $\text{ker}\big(\Tmin^* \big)$:
\begin{equation}\lb{10.162}
\begin{split}
u_1^{(0)}(x)&=C_0 e^{-\int_a^x s(t)dt} \int_a^x p(t)^{-1}e^{2\int_a^t s(t')dt'} dt,\\
u_2^{(0)}(x)&=e^{-\int_a^x s(t) dt}-e^{-\int_a^b s(t) dt}u_1^{(0)}(x),\quad x\in [a,b],
\end{split}
\end{equation}
where 
\begin{equation}\lb{10.163a}
C_0:=e^{\int_a^b s(t) dt}\bigg[ \int_a^b p(t)^{-1}e^{2\int_a^t s(t') dt'} dt \bigg]^{-1}>0.
\end{equation}
One computes
\begin{equation}\lb{10.163}
\begin{split}
\big(u_1^{(0)}(\cdot)\big)^{[1]}(x)&=C_0e^{\int_a^x s(t) dt},\\
\big(u_2^{(0)}(\cdot)\big)^{[1]}(x)&=-e^{-\int_a^b s(t) dt}\big(u_1^{(0)}(\cdot)\big)^{[1]}(x), \quad x\in [a,b],
\end{split}
\end{equation}
and 
\begin{equation}\lb{10.164}
\big[\tau^{(0)}u_j^{(0)}(\cdot) \big](x)=0\ \text{a.e.\ in $(a,b)$, $j=1,2$},
\end{equation}
where $\tau^{(0)}$ denotes the differential expression of \eqref{2.2} in the present special case $q=0$ 
a.e.\ in $(a,b)$.  It follows that $\big\{u_j^{(0)}(\cdot) \big\}_{j=1,2}\subset \text{dom}\big(\Tmin^* \big)$ 
forms a basis for $\text{ker}\big(\Tmin^* \big)=\text{ker}\big(\Tmax \big)$.  In addition, the equalities 
in \eqref{10.150} are satisfied.  With this pair of basis vectors, one infers that the matrix $R=R_K^{(0)}$ 
which parameterizes the (non-separated) boundary conditions for the Krein--von Neumann extension is 
\begin{equation}\lb{10.165}
R_K^{(0)}=
\begin{pmatrix}
e^{-\int_a^b s(t) dt} & e^{-\int_a^b s(t) dt}\int_a^b p(t)^{-1}e^{2\int_a^t s(t')dt'}dt \\
0 & e^{\int_a^b s(t) dt}
\end{pmatrix}.
\end{equation}
Explicitly, the boundary conditions corresponding to $S_K^{(0)}$ read:
\begin{align}
u^{[1]}(b)&=e^{\int_a^b s(t) dt}u^{[1]}(a)\no\\
&=e^{2\int_a^b s(t) dt}\bigg[ \int_a^b p(t)^{-1}e^{2\int_a^t s(t') dt'}dt\bigg]^{-1}\Big(u(b)-e^{-\int_a^b s(t) dt}u(a) \Big), \lb{10.166}\\
&\hspace*{7.6cm} u\in \text{dom}\big(S_K^{(0)} \big).\no
\end{align}
\end{example}

\section{Positivity Preserving and Improving Resolvents and Semigroups in~the Regular Case} 
\lb{s13}

In our final section, we prove a criterion for a self-adjoint extension of $\Tmin$ to generate a positivity improving resolvent or, equivalently, semigroup.  The notion of a positivity improving resolvent or semigroup proves critical in a study of the smallest eigenvalue of a self-adjoint restriction, as it guarantees that the lowest eigenvalue is non-degenerate and possesses a nonnegative eigenfunction.  
In fact, we will go a step further and prove that the notions of positivity preserving and positivity improving are equivalent in the regular case. 

The self-adjoint restrictions of $\Tmax$ are characterized in terms of the functionals $BC_a^j$ and 
$BC_b^j$, $j=1,2$, in Section \ref{s6} (cf.\ \eqref{eqn:ufuncBCa} 
and \eqref{eqn:ufuncBCb}), and assuming Hypothesis \ref{h11.1} throughout this section, the functionals $BC_a^j$ and $BC_b^j$, $j=1,2$ take the form of point evaluations of functions and their quasi-derivatives at the boundary points of $(a,b)$ as in Lemma \ref{prop:PointEvalBC}, that is, 
$BC_a^1(f)=f(a)$, $BC_a^2(f)=f^{[1]}(a)$, $BC_b^1(f)=f(b)$, $BC_b^2(f)=f^{[1]}(b)$, 
$f\in \text{dom}\big(T_{\max}\big)$. Since under the assumption of Hypothesis \ref{h11.1}, $\tau$ is in 
the l.c.\ case at both endpoints of the interval $(a,b)$, all real self-adjoint restrictions of $\Tmax$ are parametrized as described in 
Theorem \ref{thm:SRLCLCsepcoup} with $\phi = 0$. Hence, we adopt the following notational convention:
$S_{\varphi_a,\varphi_b}$ denote the (real) self-adjoint restrictions of $\Tmax$ corresponding to the 
separated boundary conditions \eqref{eqn:SRLCLCsep} in Theorem \ref{thm:SRLCLCsepcoup}, that is,
\begin{align}\lb{10.103c}
& S_{\varphi_a,\varphi_b} f = \tau f, \\
& f \in \dom{S_{\varphi_a,\varphi_b}} = \bigg\{g\in\dom{\Tmax} \, \bigg| 
\begin{array}{l} g(a)\cos(\varphi_a)-g^{[1]}(a)\sin(\varphi_a)=0, \\
g(b)\cos(\varphi_b)-g^{[1]}(b)\sin(\varphi_b)=0 \end{array}\bigg\},     \no 
\end{align}
and $S_{R}$ denote the real self-adjoint restrictions of $\Tmax$ corresponding to the coupled boundary conditions \eqref{eqn:SRLCLCcoup} with $\phi = 0$ in Theorem \ref{thm:SRLCLCsepcoup}, that is,
 \begin{align}\label{10.103d}
 \begin{split} 
 & S_R f = \tau f,     \\
& f \in \dom{S_R} 
= \bigg\{g \in\dom{\Tmax} \, \bigg| \, \begin{pmatrix} g(b) \\ g^{[1]}(b)\end{pmatrix}
      = R \begin{pmatrix} g(a) \\ g^{[1]}(a)\end{pmatrix}\bigg\}.    \\
\end{split} 
 \end{align} 

Following \cite{CGNZ12} and \cite{GZ12}, the sesquilinear forms associated to \eqref{10.103c} and 
\eqref{10.103d} are readily written down and read (cf.\ Appendix\ \ref{sA}) 
\begin{align}
& \mathfrak{Q}_{S_{\varphi_a,\varphi_b}}(f,g) = \int_a^b \big[p(x)^{-1} \ol{f^{[1]}(x)} g^{[1]}(x) + q(x) \ol{f(x)} g(x)\big]dx\no\\
& \hspace*{2.4cm} + \cot(\varphi_a) \ol{f(a)} g(a) -  \cot(\varphi_b) \ol{f(b)} g(b),    \no \\
& f, g \in \text{dom}(\mathfrak{Q}_{S_{\varphi_a,\varphi_b}}) 
= \big\{h\in L^2((a,b); r(x)dx) \, \big| \, h \in AC ([a,b]),       \lb{10.110} \\
& \hspace*{3.8cm} (rp)^{-1/2} h^{[1]} \in L^2((a,b); r(x)dx)\big\},  \quad 
\varphi_a,\varphi_b \in (0,\pi),  \no \\
& \mathfrak{Q}_{S_{0,\varphi_b}}(f,g) = \int_a^b \big[p(x)^{-1} \ol{f^{[1]}(x)} g^{[1]}(x) 
 + q(x) \ol{f(x)} g(x)\big]dx  -  \cot(\varphi_b) \ol{f(b)} g(b),     \no\\
& f, g \in \text{dom}(\mathfrak{Q}_{S_{0,\varphi_b}}) 
= \big\{h\in L^2((a,b); r(x)dx) \, \big| \, h \in AC ([a,b]), \, h(a) = 0, \,      \lb{10.110a} \\
& \hspace*{3.6cm} (rp)^{-1/2} h^{[1]} \in L^2((a,b); r(x)dx)\big\}, 
\quad \varphi_b \in (0,\pi),  \no \\
& \mathfrak{Q}_{S_{\varphi_a,0}}(f,g) 
= \int_a^b \big[p(x)^{-1} \ol{f^{[1]}(x)} g^{[1]}(x) + q(x) \ol{f(x)} g(x)\big]dx 
 + \cot(\varphi_a) \ol{f(a)} g(a),      \no\\ 
& f, g \in \text{dom}(\mathfrak{Q}_{S_{\varphi_a,0}}) 
= \big\{h\in L^2((a,b); r(x)dx) \, \big| \, h \in AC ([a,b]), \, h(b) = 0,       \lb{10.110b} \\
& \hspace*{3.6cm}  (rp)^{-1/2} h^{[1]} \in L^2((a,b); r(x)dx)\big\}, 
\quad \varphi_a \in (0,\pi),     \no \\ 
& \mathfrak{Q}_{S_{0,0}}(f,g) = \int_a^b \big[p(x)^{-1} \ol{f^{[1]}(x)} g^{[1]}(x) + q(x) \ol{f(x)} g(x)\big]dx, 
 \no\\
& f, g \in \text{dom}(\mathfrak{Q}_{S_{\varphi_a,0}}) 
= \big\{h\in L^2((a,b); r(x)dx) \, \big| \, h \in AC ([a,b]), \, h(a) = h(b) = 0,     \no \\
& \hspace*{3.6cm} (rp)^{-1/2} h^{[1]} \in L^2((a,b); r(x)dx)\big\},   \lb{10.110c} 
\end{align}
and 
\begin{align} 
& \mathfrak{Q}_{S_R}(f,g) = \int_a^b \big[p(x)^{-1} \ol{f^{[1]}(x)} g^{[1]}(x) + q(x) \ol{f(x)} g(x)\big]dx\no\\
& \hspace*{1.9cm} - \f{1}{R_{1,2}} \Big\{R_{1,1} \ol{f(a)} g(a) - \big[\ol{f(a)} g(b) + \ol{f(b)} g(a)\big]  
+ R_{2,2} \ol{f(b)} g(b)\Big\},    \no \\
& f, g \in \text{dom}(\mathfrak{Q}_{S_R}) 
= \big\{h\in L^2((a,b); r(x)dx) \, \big|\, h \in AC ([a,b]),        \lb{10.110d} \\
& \hspace*{3.3cm} (rp)^{-1/2} h^{[1]} \in L^2((a,b); r(x)dx)\big\}, \quad R_{1,2} \neq 0,     \no \\
& \mathfrak{Q}_{S_R}(f,g) = \int_a^b \big[p(x)^{-1} \ol{f^{[1]}(x)} g^{[1]}(x) + q(x) \ol{f(x)} g(x)\big]dx\no\\
& \hspace*{1.9cm} - R_{2,1} R_{1,1} \ol{f(a)} g(a),    \no \\
& f, g \in \text{dom}(\mathfrak{Q}_{S_R}) 
= \big\{h\in L^2((a,b); r(x)dx) \, \big|\, h \in AC ([a,b]), \, h(b) = R_{1,1} h(a),      \no \\
& \hspace*{3.3cm} (rp)^{-1/2} h^{[1]} \in L^2((a,b); r(x)dx)\big\}, \quad R_{1,2} = 0.     \lb{10.110e} 
\end{align}
To verify \eqref{10.110}--\eqref{10.110e}, it suffices to perform an appropriate integration by parts in each 
of these cases (noting that $R_{1,1} R_{2,2} =1$ if $R_{1,2} = 0$).

With the sesquilinear forms in hand, we are now prepared to characterize when self-adjoint restrictions 
of $\Tmax$ generate positivity preserving resolvents and semigroups. For background literature on positivity preserving semigroups and resolvents, we refer, for instance, to the monographs \cite[Ch.\ 7]{Da80}, 
\cite[Ch.\ 13]{Da07}, \cite[Sects.\ 8, 10]{Fa75}, \cite[Sect.\ 3.3]{GJ81}, \cite[Chs.\ 2, 3]{Ou05}, 
\cite[Sect.\ XIII.12]{RS78}, \cite[Sect.\ 10.5]{We80}, and to the extensive list of references in \cite{GMN12}. 

Let $(M,\cM,\mu)$ denote a $\sigma$-finite, separable measure space associated with a nontrivial 
measure $($i.e., $0 < \mu(M) \leq \infty$$)$ and $L^2(M;d\mu)$ the associated complex, separable Hilbert space (cf.\ \cite[Sect.\ 1.5]{BS88} and \cite[p.\ 262--263]{Jo82} for additional facts in this context).  Then the set of nonnegative elements $0 \leq f \in L^2(M; d\mu)$ (i.e., $f(x) \geq 0$ $\mu$-a.e.) is a cone 
in $L^2(M; d\mu)$, closed in the norm and weak topologies. 

\begin{definition} \lb{d12.1} 
A bounded operator $A$ defined on $L^2(M; d\mu)$ is called {\it positivity preserving}
$($resp., {\it positivity improving}$)$ if 
\begin{equation}\lb{10.120}
0 \neq f \in L^2(M; d\mu), \, f \geq 0 \text{ $\mu$-a.e.\ implies } \, A f \geq 0 \, 
\text{ $($resp., $Af > 0$$)$ $\mu$-a.e.}
\end{equation}
\end{definition}

In the special case where $A$ is a bounded integral operator in $L^2((a,b); r(x)dx)$ with integral kernel 
denoted by $A(\cdot, \cdot)$, it is well-known that 
\begin{align} 
\text{$A$ is positivity preserving if and only if $A(\cdot,\cdot) \geq 0$ $dx \otimes dx$-a.e.\ on $(a,b) \times (a,b)$}    \lb{12.11} 
\end{align} 
(we recall that $r>0$ a.e.\ by Hypothesis \ref{h11.1}). For an extension of this result to $\sigma$-finite, separable measure spaces we refer to \cite[Theorem\ 2.3]{GMN12}. Moreover,  
\begin{equation} 
\text{if $A(\cdot,\cdot)>0$ $\prodm$-a.e., then $A$ is positivity improving.}   \lb{12.13}
\end{equation} 
(The converse to \eqref{12.13}, however, is false, cf.\ \cite[Example\ 2.6]{GMN12}.)

The following result is fundamental to the theory of positivity preserving operators.

\begin{theorem}[\cite{RS78}, p.\ 204, 209]\lb{t12.2}
Suppose that $S$ is a semibounded self-adjoint operator in $L^2(M;d\mu)$ with 
$\lambda_0 = \inf(\sigma(S))$.  Then the following conditions, $(i)$--$(iii)$, are equivalent: \\[1mm]
$(i)$ \hspace*{1mm} $e^{-tS}$ is positivity preserving for all $t\geq 0$.\\[1mm] 
$(ii)$  $\big(S-\lambda I_{L^2(M;d\mu)}\big)^{-1}$ is positivity preserving for all $\lambda < \lambda_0$. 
\\[1mm] 
$(iii)$ The Beurling--Deny criterion: $f\in {\rm dom}\big(|S|^{1/2}\big)$ implies $|f| \in {\rm dom}\big(|S|^{1/2}\big)$ and \\[1mm]
\hspace*{7mm} $\big\| (S - \lambda_0 I_{L^2(M;d\mu)})^{1/2} |f|\big\|_{L^2(M;d\mu)} 
\leq \big\| (S - \lambda_0 I_{L^2(M;d\mu)})^{1/2} f \big\|_{L^2(M;d\mu)}$.
\end{theorem}

The next and principal result of this section provides a necessary and sufficient condition for a (necessarily real) self-adjoint restriction of $\Tmax$ (resp., extension of $\Tmin$) to generate a positivity preserving resolvent and semigroup. We recall that positivity preserving requires reality preserving and hence it suffices to consider real self-adjoint extensions of $\Tmin$. In fact, we will prove more and show that the notions of positivity preserving and positivity improving are, in fact, equivalent in the regular case.

\begin{theorem}\lb{t12.3}
Assume Hypothesis \ref{h11.1}. \\
$(i)$ In the case of separated boundary conditions, all self-adjoint extensions of $\Tmin$ lead to positivity improving semigroups and resolvents. More precisely, for all $\varphi_a, \varphi_b \in [0,\pi)$, 
$e^{-tS_{\varphi_a,\varphi_b}}$ is positivity improving for all $t\geq 0$, equivalently, 
$(S_{\varphi_a,\varphi_b} - \lambda I_r)^{-1}$ is positivity improving for all 
$\lambda<\inf (\sigma(S_{\varphi_a,\varphi_b}))$. In addition,
\begin{equation}\lb{13.46A}
(S_{\varphi_a,\varphi_b}-\lambda I_r)^{-1}-(S_{0,0}-\lambda I_r)^{-1},   \quad 
\lambda < \inf(\sigma(S_{\varphi_a,\varphi_b})),
\end{equation}  
is positivity improving, implying the inequality
\begin{equation}\lb{10.142A}
G_{\lambda,\varphi_a,\varphi_b}(x,x') \geq G_{\lambda,0,0}(x,x') \geq 0,\quad x,x'\in[a,b],\; 
\lambda < \inf(\sigma(S_{\varphi_a,\varphi_b})). 
\end{equation}
In particular,
\begin{equation}
G_{\lambda,0,0}(x,x') > 0,\quad x,x'\in(a,b),\;  \lambda < \inf(\sigma(S_{0,0})). 
\end{equation}
Here $G_{z,\varphi_a,\varphi_b}(\cdot, \cdot)$, $z\in \rho(S_{\varphi_a,\varphi_b})$ 
$($resp., $G_{z,0,0}(\cdot , \cdot)$, 
$z\in\rho(S_{0,0})$$)$, denotes the Green's function $($i.e., the integral kernel of the resolvent\,$)$ 
of $S_{\varphi_a,\varphi_b}$ $($resp., of $S_{0,0}$$)$.  \\[1mm]
$(ii)$ In the case of $($necessarily real\,$)$ coupled boundary conditions, $e^{-t S_R}$ is positivity preserving for all $t\geq 0$, equivalently, $(S_R - \lambda I_r)^{-1}$ is positivity preserving for all 
$\lambda<\inf (\sigma(S_R))$, if and only if  
\begin{equation}
\text{either $R_{1,2} < 0$, or $R_{1,2} =0$ and $R_{1,1} > 0$ $($equivalently, $R_{2,2} > 0$$)$.} 
\lb{12.12}
\end{equation} 
Moreover, $e^{-t S_R}$ is positivity improving for all $t\geq 0$ if and only if it is positivity preserving 
for all $t\geq 0$. Equivalently, $(S_R - \lambda I_r)^{-1}$ is positivity improving for all 
$\lambda<\inf (\sigma(S_R))$ if and only if it is positivity preserving for all 
$\lambda<\inf (\sigma(S_R))$. In addition,
\begin{equation}\lb{13.46a}
(S_R-\lambda I_r)^{-1}-(S_{0,0}-\lambda I_r)^{-1},   \quad 
\lambda < \inf (\sigma(S_R)),
\end{equation}  
is positivity improving, implying the inequality
\begin{equation}\lb{10.142a}
G_{\lambda,R}(x,x') \geq G_{\lambda,0,0}(x,x') \geq 0,  \quad x,x'\in[a,b],\; 
\lambda<\inf(\sigma(S_R)). 
\end{equation}
Here $G_{z,R}(\cdot, \cdot)$, $z\in \rho(S_R)$, denotes the Green's function of $S_R$.  
\end{theorem}
\begin{proof} 
Case $(i)$.\ {\it $($Real\,$)$ Separated Boundary Conditions:} Let 
$G_{z,\varphi_a,\varphi_b}(\cdot , \cdot)$, 
$z\in \bbC \backslash \sigma( S_{\varphi_a,\varphi_b})$, denote the Green's function for the 
resolvent of $S_{\varphi_a,\varphi_b}$. To demonstrate positivity improving, it suffices to show that 
\begin{equation}\lb{10.142}
G_{\lambda,\varphi_a,\varphi_b}(x,x') > 0 \, \text{ for all } \, (x,x')\in (a,b)\times (a,b), \;  
\lambda<\inf (\sigma(S_{\varphi_a,\varphi_b})), 
\end{equation}
employing the fact \eqref{12.13}. 
In this context, we note that $G_{z,\varphi_a,\varphi_b}(\cdot , \cdot)$ is continuous on 
$[a,b] \times [a,b]$. To this end, let 
$\lambda <\inf (\sigma(S_{\varphi_a,\varphi_b}))$ and let 
$f_{c,\theta_c}(\lambda, \cdot\,)$, 
$c\in \{a,b\}$, denote Weyl--Titchmarsh solutions of $(\tau-\lambda)u=0$ at $a$ and $b$, respectively, so that
\begin{equation}\lb{10.143a}
\begin{split}
&(\tau-\lambda)f_{c,\theta_c}(\lambda, \cdot \,)=0\ \text{a.e.\ in $(a,b)$},\\
&f_{c,\theta_c}(\lambda,c)\cos(\theta_c)-f_{c,\theta_c}^{[1]}(\lambda,c)\sin(\theta_c)=0, \quad c\in\{a,b\}.
\end{split}
\end{equation}
Then, by Theorem \ref{thm:ressep}, one obtains the representation
\begin{equation}\lb{10.143}
G_{\lambda,\varphi_a,\varphi_b}(x,x')=W_{\varphi_b,\varphi_a}^{-1}
\begin{cases}
f_{a,\varphi_a}(\lambda,x)f_{b,\varphi_b}(\lambda,x'), & a \leq x \leq x' \leq b, \\
f_{a,\varphi_a}(\lambda,x')f_{b,\varphi_b}(\lambda,x), & a \leq x' \leq x \leq b, 
\end{cases}
\end{equation}
where $W_{\varphi_b,\varphi_a}=W\big(f_{b,\varphi_b}(\lambda, \cdot \,), 
f_{a,\varphi_a}(\lambda, \cdot \,)\big)$ abbreviates the Wronskian of 
$f_{b,\varphi_b}(\lambda, \cdot \,)$ 
and $f_{a,\varphi_a}(\lambda, \cdot \,)$. We claim that both $f_{b,\varphi_b}(\lambda, \cdot \,)$ 
and $f_{a,\varphi_a}(\lambda, \cdot \,)$ are sign-definite on $(a,b)$.  In order to see this, 
one observes that the Green's function is nonnegative along the diagonal:
\begin{equation}\lb{10.144}
G_{\lambda,\varphi_a,\varphi_b}(x,x)\geq 0, \quad x\in (a,b),
\end{equation}
a fact that has already been used in the proof of Theorem \ref{t10.12}: Indeed, if \eqref{10.144} fails to hold, then there exists an $x_0\in (a,b)$ such that the inequality 
$G_{\lambda,\varphi_a,\varphi_b}(x_0,x_0)<0$ holds.  Since 
$G_{\lambda,\varphi_a,\varphi_b}(\cdot , \cdot)$ is continuous at the point $(x_0,x_0)$, there exists $\delta>0$ such that
\begin{equation}\lb{10.145}
G_{\lambda,\varphi_a,\varphi_b}(x,x')<0, \quad 
(x,x')\in (x_0-\delta,x_0+\delta)\times (x_0-\delta,x_0+\delta),
\end{equation}
and one obtains
\begin{equation}\lb{10.146}
\big\langle \big(S_{\varphi_a,\varphi_b}-\lambda I_r \big)^{-1}\chi_{(x_0-\delta,x_0+\delta)}, \chi_{(x_0-\delta,x_0+\delta)}  \big\rangle_r<0.
\end{equation}
However, \eqref{10.146} contradicts the fact that 
$\big(S_{\varphi_a,\varphi_b}-\lambda I_r \big)^{-1}\geq 0$.  Therefore, inequality \eqref{10.144} has been established.  

Since a nontrivial solution of $(\tau-\lambda)u=0$ must change signs at a zero in $(a,b)$ 
(cf.\ Lemma \ref{l10.1a}), and linearly independent solutions do not have common zeros, \eqref{10.144} implies that $f_{a,\varphi_a}(\lambda,\cdot \,)$ and $f_{b,\varphi_b}(\lambda, \cdot \,)$ are sign-definite (i.e., strictly negative or positive) on $(a,b)$. In particular, since $W_{\varphi_b,\varphi_a}$ is a constant, $G_{\lambda,\varphi_a,\varphi_b}(\cdot, \cdot)$ is sign-definite, and the inequality in \eqref{10.142} follows from the structure of the Green's function in \eqref{10.143}. 

To go beyond mere positivity improving and actually show \eqref{13.46A} and hence \eqref{10.142A} requires additional arguments: For each $z \in \rho(S_{0,0})$, 
let $u_j(z,\cdot \,)$, $j=1,2$, denote solutions to $\tau u = z u$ satisfying the conditions
\begin{equation}\lb{13.33}
\begin{split}
&u_1(z,a)=0, \quad u_1(z,b)=1,\\
&u_2(z,a)=1, \quad u_2(z,b)=0.
\end{split}
\end{equation} 
We note that for $\lambda < \inf(\sigma(S_{0,0}))$, 
$u_j(\lambda, \cdot \,)$, $j=1,2$, are nonnegative on $[a,b]$. 

Then, mimicking the proof 
of \cite[Theorem\ 3.1$(i)$]{CGNZ12} line by line, and assuming that 
$\varphi_a\neq 0$ and $\varphi_b\neq 0$, one infers that the matrix
\begin{equation}\lb{13.34A}
D_{\varphi_a,\varphi_b}(z)= \begin{pmatrix}
\cot(\varphi_b) - u_1^{[1]}(z,b) & - u_2^{[1]}(z,b)\\
u_1^{[1]}(z,a) & \cot(\varphi_a) + u_2^{[1]}(z,a)
\end{pmatrix},  \quad z\in \rho(S_{\varphi_a,\varphi_b})\cap\rho(S_{0,0}), 
\end{equation}
is invertible and one obtains the following Krein-type resolvent identity, 
\begin{align}\lb{13.35A}
& (S_{\varphi_a,\varphi_b} - z I_r)^{-1}-(S_{0,0} -z I_r)^{-1}
= - \sum_{j,k=1}^2D_{\varphi_a,\varphi_b}(\lambda)^{-1}_{j,k} \, 
\langle u_k(\ol z, \cdot \,), \cdot \, \rangle_r \, u_j(z, \cdot \,),     \no \\
& \hspace*{7.2cm} z \in \rho(S_{\varphi_a,\varphi_b}) \cap \rho(S_{0,0}).&
\end{align}
If $\varphi_a \neq 0$, $\varphi_b = 0$, one gets analogously to \cite[Theorem\ 3.1$(ii)$]{CGNZ12} that 
\begin{equation}\lb{3.19AA}
d_{\varphi_a,0}(z)= \cot(\varphi_a) + u_2^{[1]}(z,a), \quad
z\in \rho(S_{\varphi_a,0})\cap \rho(S_{0,0}),
\end{equation}
is nonzero and
\begin{align}\lb{3.20AA}
\begin{split}
(S_{\varphi_a,0} - z I_r)^{-1} = (S_{0,0}-z I_r)^{-1} - d_{\varphi_a,0}(z)^{-1} 
\langle u_2(\overline{z},\cdot \,),\cdot \, \rangle_r \, u_2(z,\cdot \,),&   \\
z\in \rho(S_{\varphi_a,0})\cap \rho(S_{0,0}).&
\end{split}
\end{align}
Similarly, if $\varphi_a = 0$, $\varphi_b \neq 0$, one obtains as in 
\cite[Theorem\ 3.1$(iii)$]{CGNZ12} that 
\begin{equation}\lb{3.22AA}
d_{0,\varphi_b}(z)= \cot(\varphi_b) - u_1^{[1]}(z,b), \quad
z\in \rho(S_{0,\varphi_b})\cap \rho(S_{0,0}),
\end{equation}
is nonzero and
\begin{align}\lb{3.23AA}
\begin{split}
(S_{0,\varphi_b} - z I_r)^{-1} = (S_{0,0} - zI_r)^{-1} - d_{0,\varphi_b}(z)^{-1}
\langle u_1(\overline{z},\cdot \,), \cdot \, \rangle_r \, u_1(z,\cdot \,),&  \\
z\in \rho(S_{0,\varphi_b})\cap \rho(S_{0,0}).&
\end{split}
\end{align}
Next, one observes that \eqref{13.33}, \eqref{13.35A}, \eqref{3.20AA}, and \eqref{3.23AA} imply
\begin{align}
G_{z,\varphi_a,\varphi_b}(a,a) &= - D_{\varphi_a,\varphi_b}(z)_{2,2}^{-1}, \quad 
G_{z,\varphi_a,\varphi_b}(b,b) = - D_{\varphi_a,\varphi_b}(z)_{1,1}^{-1},    \no \\
G_{z,\varphi_a,\varphi_b}(a,b) & = G_{z,\varphi_a,\varphi_b}(b,a)
= - D_{\varphi_a,\varphi_b}(z)_{2,1}^{-1} = - D_{\varphi_a,\varphi_b}(z)_{1,2}^{-1},   \lb{13.31B} \\
& \hspace*{5.55cm} \varphi_a \neq 0, \, \varphi_b \neq 0,   \no \\ 
G_{z,\varphi_a,0}(a,a) &= - d_{\varphi_a,0}(z)^{-1},   \no \\
G_{z,\varphi_a,0}(b,b) &= G_{z,\varphi_a,0}(a,b) = G_{z,\varphi_a,0}(b,a) = 0,  
\quad \varphi_a \neq 0, \; \varphi_b = 0, \\
G_{z,0,\varphi_b}(b,b) &= - d_{0,\varphi_b}(z)^{-1},   \no \\
G_{z,0,\varphi_b}(a,a) &= G_{z,0,\varphi_b}(a,b) = G_{z,0,\varphi_b}(b,a) = 0,  
\quad \varphi_a = 0, \; \varphi_b \neq 0.
\end{align}
Since $(S_{\varphi_a,\varphi_b} - \lambda I_r)^{-1}$, 
$\lambda<\inf(\sigma(S_{\varphi_a,\varphi_b}))$, is positivity preserving, its integral kernel is nonnegative a.e.~ in $[a,b]\times[a,b]$ by 
\eqref{12.11}. In fact, by continuity, it is nonnegative everywhere in 
$[a,b]\times[a,b]$. As a result, returning to the case $\varphi_a \neq 0$, $\varphi_b \neq 0$, 
\eqref{13.31B} implies
\begin{equation}\lb{13.37AA}
D_{\varphi_a,\varphi_b}(\lambda)_{j,k}^{-1}\leq 0,\quad j,k\in \{1,2\},  \; 
\lambda < \inf (\sigma(S_{\varphi_a,\varphi_b})),
\end{equation}
so that the matrix $-D_{\varphi_a,\varphi_b}(\lambda)^{-1}$ is actually positivity preserving as an operator on $\bbC^2$ for each $\lambda < \inf (\sigma(S_{\varphi_a,\varphi_b}))$.  Thus, 
\eqref{13.35A} and \eqref{13.37AA} immediately yield the following inequality for Green's functions:
\begin{align}\lb{13.38AA}
& G_{\lambda,\varphi_a,\varphi_b}(x,x')-G_{\lambda,0,0}(x,x') 
= - \sum_{j,k=1}^2D_{\varphi_a,\varphi_b}(\lambda)_{j,k}^{-1} \, 
u_j(\lambda,x)u_k(\lambda,x')\geq 0,     \no \\
& \hspace*{5.4cm} 
x,x'\in [a,b],\; \lambda < \inf (\sigma(S_{\varphi_a,\varphi_b})).
\end{align}
We note that the final inequality in \eqref{13.38AA} makes use of \eqref{13.37AA} as well as 
nonnegativity of the functions $u_j(\lambda,\cdot \,)$, $j=1,2$, on the interval $[a,b]$. Another 
application of \eqref{12.11} then implies that the resolvent difference
\begin{equation}\lb{13.39AA}
(S_{\varphi_a,\varphi_b} - \lambda I_r)^{-1}-(S_{0,0}-\lambda I_r)^{-1},\quad 
\lambda < \inf (\sigma(S_{\varphi_a,\varphi_b})),
\end{equation}
is positivity preserving.  Since $(S_{0,0}-\lambda I_r)^{-1}$ is positivity improving for all 
$\lambda < \inf(\sigma(S_{0,0}))$ by what was just shown at the beginning of this proof, and 
$S_{\varphi_a,\varphi_b} \neq S_{0,0}$, \cite[Corollary\ 9]{KR80} implies that the resolvent difference in \eqref{13.39AA} is actually positivity improving. In addition, it 
also implies that $(S_{\varphi_a,\varphi_b} - \lambda I_r)^{-1}$ is positivity improving for all 
$\lambda<\inf(\sigma(S_{\varphi_a,\varphi_b}))$. Inequality \eqref{10.142A} now directly follows 
from \eqref{13.38AA}. This completes the case where $\varphi_a\neq 0$ and $\varphi_b\neq 0$.

If $\varphi_a \neq 0$, $\varphi_b = 0$, the resolvent identity \eqref{3.20AA}, implies the 
following Green's function relation,
\begin{equation}\lb{13.43AA}
\begin{split}
G_{z,\varphi_a,0}(x,x') - G_{z,0,0}(x,x')= - d_{\varphi_a,0}(z)^{-1}
u_2(z,x) u_2(z,x'),&\\
z \in \rho(S_{\varphi_a,0}) \cap \rho(S_{0,0}),&
\end{split}
\end{equation}
and consequently,
\begin{equation}\lb{13.44AA}
0\leq G_{\lambda,\varphi_a,0}(b,b)=-q_{\varphi_a,0}(\lambda)^{-1}, \quad 
\lambda < \inf (\sigma(S_{\varphi_a,0})).
\end{equation}
Nonnegativity of $u_2(\lambda, \cdot \,)$ and \eqref{13.43AA}, \eqref{13.44AA} imply 
\begin{equation}\lb{13.45AA}
G_{\lambda,\varphi_a,0}(x,x')-G_{\lambda,0,0}(x,x')\geq 0,\quad x,x'\in[a,b], \; 
\lambda < \inf (\sigma(S_{\varphi_a,0})),
\end{equation}
which by \eqref{12.11} is equivalent to the fact that the resolvent difference,
\begin{equation}\lb{13.46AA}
(S_{\varphi_a,0} - \lambda I_r)^{-1}-(S_{0,0}-\lambda I_r)^{-1},\quad 
\lambda < \inf (\sigma(S_{\varphi_a,0})),
\end{equation}  
is positivity preserving.  Applying \cite[Corollary\ 9]{KR80} once again, one obtains the 
stronger result that the resolvent difference in \eqref{13.46AA} is positivity improving, and that 
$(S_{\varphi_a,0} - \lambda I_r)^{-1}$, $\lambda<\inf(\sigma(S_{\varphi_a,0}))$, is positivity 
improving as well. Inequality \eqref{10.142A} is just a restatement of \eqref{13.45AA}. 
This completes the case $\varphi_a \neq 0$, $\varphi_b = 0$.

The case $\varphi_a = 0$, $\varphi_b \neq 0$ is completely analogous and hence we skip it. 

\smallskip

Case $(ii)$.\ {\it $($Real\,$)$ Coupled Boundary Conditions:} 
First, we show the conditions in \eqref{12.12} are necessary and sufficient for positivity preserving of $e^{-t S_R}$ for all $t\geq 0$, or equivalently, positivity preserving of $(S_R - \lambda I_r)^{-1}$ for all
$\lambda<\inf (\sigma(S_R))$.
We begin with the proof of sufficiency.  To this end, suppose that either $R_{1,2} < 0$ or 
$R_{1,2} = 0$ and $R_{1,1} > 0$.  In order to show that 
$e^{-tS_R}$ is positivity preserving for all $t\geq 0$, we will verify the Beurling--Deny criterion 
Theorem \ref{t12.2}\,$(iii)$.  Therefore, we must show the following condition holds:
\begin{equation}\lb{10.126}
\begin{split}
&\text{$f\in \text{dom}(\mathfrak{Q}_{S_R})$ implies $|f|\in 
\text{dom}(\mathfrak{Q}_{S_R})$ and}\\
&\quad \text{ $\mathfrak{Q}_{S_R}(|f|,|f|)-\lambda_{S_R}\langle |f| , 
|f| \rangle_r\leq \mathfrak{Q}_{S_R}(f,f)-\lambda_{S_R}\langle f , f \rangle_r$,}
\end{split}
\end{equation}
where we have set $\lambda_{S_R}=\inf (\sigma(S_R))$.

First, we claim that 
\begin{equation}\lb{10.129}
\text{$f\in \text{dom}(\mathfrak{Q}_{S_R})$ implies 
$|f|\in \text{dom}(\mathfrak{Q}_{S_R})$ if $R_{1,2} \neq 0$.}
\end{equation}
Indeed, if $f\in \text{dom}(\mathfrak{Q}_{S_R})$ is fixed, then
\begin{equation}\lb{10.130}
\text{$f\in AC([a,b])$ and $(rp)^{-1/2}f^{[1]}\in L^2((a,b);r(x)dx)$},
\end{equation}
and it follows that $|f|\in AC([a,b])$.  Moreover, since $|f|'$ coincides a.e.\ in $(a,b)$ with the 
function (cf., e.g., \cite[Theorem 6.17]{LL01})
\begin{equation}\lb{10.131}
d_f(x)=\begin{cases}
|f(x)|^{-1}\big[\Re(f)(x) \Re(f)'(x)+\Im(f)(x) \Im(f)'(x) \big], & f(x)\neq 0, \\
0, & f(x)=0,
\end{cases} 
\end{equation}
one verifies that $|f|^{[1]}$ coincides a.e.\ in $(a,b)$ with the function
\begin{equation}\lb{10.132}
\widetilde{d}_f(x)=\begin{cases}
|f(x)|^{-1}\big[\Re(f)(x) \Re\big(f^{[1]}\big)(x)+\Im(f)(x) \Im\big(f^{[1]}\big)(x) \big], & f(x)\neq 0, \\
0, & f(x)=0,
\end{cases} 
\end{equation}
and, subsequently, the inequality
\begin{align}
&\Big||f(x)|^{-1}\big[\Re(f)(x) \Re\big(f^{[1]}\big)(x)+\Im(f)(x) \Im\big(f^{[1]}\big)(x) \big] \Big|^2\no\\
&\quad \leq \Re\big(f^{[1]}(x)\big)^2+\Im\big(f^{[1]}(x) \big)^2 \, 
\text{ for a.e.\ $x\in \{x'\in (a,b)\, |\, f(x')\neq 0\}$},\lb{10.133}
\end{align}
implies
\begin{equation}\lb{10.134}
\big||f|^{[1]} \big|\leq \big|f^{[1]} \big| \, \text{ a.e.\ in $(a,b)$},\, f\in AC([a,b]).
\end{equation}
The second containment in \eqref{10.130} then implies $(rp)^{-1/2}|f|^{[1]}\in L^2((a,b);r(x)dx)$, 
establishing \eqref{10.129} (cf.\ \eqref{10.110d}).  Thus, it remains to verify inequality \eqref{10.126}.  
Since the terms containing $\lambda_{S_R}$ in the inequality in \eqref{10.126} are equal, it suffices to 
establish the following inequality:
\begin{equation}\lb{10.135}
\mathfrak{Q}_{S_R}(|f|,|f|)\leq \mathfrak{Q}_{S_R}(f,f).
\end{equation}
On the other hand, \eqref{10.134} implies
\begin{equation}\lb{10.136}
\int_a^b p(x)^{-1}\big||f|^{[1]}(x) \big|^2dx\leq \int_a^b p(x)^{-1}\big|f^{[1]}(x) \big|^2dx,
\end{equation}
and hence by \eqref{10.110d} when $R_{1,2} < 0$, it suffices to verify the simpler inequality
\begin{equation}\lb{10.137}
\f{1}{R_{1,2}} \Big\{2 |f(a)| |f(b)| - \big[f(a)\ol{f(b)} + \ol{f(a)} f(b) \big]\Big\} \leq 0.
\end{equation}
One computes for the difference in \eqref{10.137}:
\begin{equation}\lb{10.138}
\f{2}{R_{1,2}} \big[|\ol{f(a)}f(b)|-\Re\big(\ol{f(a)}f(b)\big) \big] \leq 0,
\end{equation}
since $R_{1,2} < 0$, by assumption. If $R_{1,2} = 0$ and $R_{1,1} > 0$, then by \eqref{10.110e} 
it only remains to show that $f \in \text{dom}(\mathfrak{Q}_{S_R}) $ implies 
$|f| \in \text{dom}(\mathfrak{Q}_{S_R}) $, which is indeed guaranteed since $R_{j,j} > 0$, 
$j=1,2$, completing the proof of sufficiency.

In order to establish necessity of the conditions $R_{1,2} < 0$ or $R_{1,2} = 0$ and $R_{1,1} > 0$, 
suppose that $e^{-tS_R}$ is positivity preserving for all $t\geq 0$.  Then by the Beurling--Deny criterion, Theorem \ref{t12.2}\,$(iii)$, condition \eqref{10.126} holds.  In particular, for $R_{1,2} \neq 0$, 
equation \eqref{10.110d} and inequality \eqref{10.126} imply
\begin{equation}\lb{10.139}
\begin{split}
&\int_a^b p(x)^{-1}\Big[\big||f|^{[1]}(x) \big|^2 - \big|f^{[1]}(x) \big|^2 \Big]dx\\
&\quad + \f{2}{R_{1,2}} \big[|\ol{f(a)}f(b)|-\Re\big(\ol{f(a)}f(b)\big) \big]\leq 0, 
\quad f\in \text{dom}( \mathfrak{Q}_{S_R}).
\end{split}
\end{equation}
If $f\in \text{dom}( \mathfrak{Q}_{S_R})$ is real-valued, then one verifies that 
$|f|^{[1]}=\sgn(f)f^{[1]}$ a.e.\ in $(a,b)$, where $\sgn(f)$ equals $f/|f|$ if $f\neq 0$ and is zero 
otherwise, as a special case of \eqref{10.132}.  Consequently, in the case where $f$ is real-valued, 
the integral appearing in \eqref{10.139} vanishes, and the inequality reduces to
\begin{equation}\lb{10.140}
\f{2}{R_{1,2}} \big[|f(a)f(b)|-f(a)f(b) \big]\leq 0,
 \quad f\in \text{dom}( \mathfrak{Q}_{S_R})\ \text{and $f$ real-valued}.
\end{equation}
Choosing a real-valued function $f_0\in AC([a,b])$ such that $f_0^{[1]}\in AC([a,b])$ and 
$f_0(a) f_0(b)<0$, one infers that $f_0 \in \text{dom}( \mathfrak{Q}_{S_R})$. Taking $f_0$ 
as a test function in \eqref{10.140}, one concludes that $R_{1,2} < 0$. On the other hand, if 
$R_{1,2} = 0$, equation \eqref{10.110e} yields that the implication and the inequality \eqref{10.126} 
are satisfied provided  
the boundary condition $h(b) = R_{1,1} h(a)$ in $\text{dom}(\mathfrak{Q}_{S_R})$ holds. This 
necessitates the condition $R_{1,1} > 0$.   

The statement concerning positivity preserving of the resolvents follows from 
Theorem \ref{t12.2}\,$(iii)$.  This completes the proof that the conditions in \eqref{12.12} are necessary and sufficient for positivity preserving of $e^{-t S_R}$ for all $t\geq 0$, or equivalently, positivity preserving of $(S_R - \lambda I_r)^{-1}$ for all $\lambda<\inf (\sigma(S_R))$.  

It remains to prove the claim that positivity preserving is, in fact, equivalent to positivity improving 
in item $(ii)$. The sufficiency claim is clear since any bounded positivity improving operator is, 
of course, positivity preserving. Thus, it remains to prove the necessity claim. To this end, 
suppose that $R\in \SL_2(\bbR)$ is fixed and satisfies the conditions in \eqref{12.12}.  Then 
$(S_R-\lambda I_r)^{-1}$ is positivity preserving for all $\lambda < \inf(\sigma(S_R))$.  To 
establish the necessity claim, it is enough to show $(S_R-\lambda I_r)^{-1}$ is positivity 
improving for {\it some} $\lambda < \inf(\sigma(S_R))$, as positivity improving then extends 
to $(S_R-\lambda I_r)^{-1}$ for all $\lambda < \inf(\sigma(S_R))$ and to $e^{-tS_R}$ for all 
$t\geq 0$ by \cite[Theorem\ XIII.44]{RS78}. In order to do this, we consider separately the 
cases $R_{1,2}<0$ and $R_{1,2}=0$ (and therefore, $R_{2,2}>0$).  

First, we consider the case $R_{1,2}<0$.  
Then, mimicking the proof of \cite[Theorem\ 3.2\,$(i)$]{CGNZ12} line by line, one infers that the matrix
\begin{equation}\lb{13.34}
Q_{R}(z)=\begin{pmatrix}
\frac{R_{2,2}}{R_{1,2}} - u_1^{[1]}(z,b) & \frac{-1}{R_{1,2}} - u_2^{[1]}(z,b)\\
\frac{-1}{R_{1,2}} + u_1^{[1]}(z,a) & \frac{R_{1,1}}{R_{1,2}} + u_2^{[1]}(z,a)
\end{pmatrix}, \quad z \in \rho(S_R) \cap \rho(S_{0,0}),
\end{equation}
is invertible and one obtains the following Krein-type resolvent identity, 
\begin{align}\lb{13.35}
\begin{split}
(S_R - z I_r)^{-1}-(S_{0,0} -z I_r)^{-1}= - \sum_{j,k=1}^2Q_{R}(\lambda)^{-1}_{j,k} \, 
\langle u_k(\ol z, \cdot \,), \cdot \, \rangle_r \, u_j(z, \cdot \,),&  \\
z \in \rho(S_R) \cap \rho(S_{0,0}).&
\end{split}
\end{align}
Subsequently, \eqref{13.33} and \eqref{13.35} imply
\begin{align}
G_{z,R}(a,a) &= -Q_R(\lambda)_{2,2}^{-1},\quad G_{z,R}(b,b)=-Q_R(z)_{1,1}^{-1},  \no\\
G_{z,R}(a,b) &=G_{z,R}(b,a)=-Q_R(z)_{1,2}^{-1}=-Q_R(z)_{2,1}^{-1},   \lb{13.36}\\
& \hspace*{3.23cm} z \in \rho(S_R) \cap \rho(S_{0,0}).    \no
\end{align}
Since $(S_R-\lambda I_r)^{-1}$, $\lambda<\inf(\sigma(S_R))$, is positivity preserving, its integral kernel is nonnegative a.e.~ in $[a,b]\times[a,b]$ by \eqref{12.11}.  In fact, by continuity, it is nonnegative everywhere in $[a,b]\times[a,b]$.  As a result, \eqref{13.36} yields
\begin{equation}\lb{13.37}
Q_R(\lambda)_{j,k}^{-1}\leq 0,\quad j,k\in \{1,2\},  \; 
\lambda < \inf (\sigma(S_R)),
\end{equation}
so that the matrix $-Q_R(\lambda)^{-1}$ is actually positivity preserving as an operator on $\bbC^2$ for each $\lambda < \inf (\sigma(S_R))$.  Thus, \eqref{13.35} and \eqref{13.37} immediately yield the following inequality for Green's functions:
\begin{equation}\lb{13.38}
\begin{split}
G_{\lambda,R}(x,x')-G_{\lambda,0,0}(x,x')= 
- \sum_{j,k=1}^2Q_R(\lambda)_{j,k}^{-1} \, u_j(\lambda,x)u_k(\lambda,x')\geq 0,&\\
x,x'\in [a,b],\; \lambda < \inf (\sigma(S_R)).&
\end{split}
\end{equation}
We note that the final inequality in \eqref{13.38} makes use of \eqref{13.37} as well as 
nonnegativity of the functions $u_j(\lambda,\cdot \,)$, $j=1,2$, on the interval $[a,b]$.  Another application of \eqref{12.11} then implies that the resolvent difference
\begin{equation}\lb{13.39}
(S_R-\lambda I_r)^{-1}-(S_{0,0}-\lambda I_r)^{-1},\quad 
\lambda < \inf (\sigma(S_R)),
\end{equation}
is positivity preserving. Again, since $(S_{0,0}-\lambda I_r)^{-1}$ is positivity improving for all 
$\lambda < \inf(\sigma(S_{0,0}))$ by item $(i)$, and $S_R \neq S_{0,0}$, \cite[Corollary\ 9]{KR80} implies that the resolvent difference in \eqref{13.39} is actually positivity improving. In addition, it 
also implies that $(S_R-\lambda I_r)^{-1}$ is positivity improving for all 
$\lambda<\inf(\sigma(S_R))$. Inequality \eqref{10.142a} directly follows from \eqref{13.38}. 
This completes the case where $R_{1,2} < 0$.

The degenerate case where $R_{1,2}=0$ and $R_{2,2}>0$ is handled similarly. The primary difference is that in this case, the Krein-type resolvent identity reads, 
\begin{equation}\lb{13.40}
\begin{split}
(S_R - z I_r)^{-1} - (S_{0,0} - z I_r)^{-1} 
= -q_R(z)^{-1}\langle u_R(\ol z, \cdot \,), \cdot \, \rangle_r \, u_R(z, \cdot\,),&\\
z \in \rho(S_R) \cap \rho(S_{0,0}),&
\end{split}
\end{equation}
where
\begin{equation}\lb{13.41}
\begin{split}
q_R(z)&=R_{2,1}R_{2,2}+R_{2,2}^2u_2^{[1]}(z,a)+R_{2,2}u_1^{[1]}(z,a)\\
&\quad -R_{2,2}u_2^{[1]}(z,b)-u_1^{[1]}(z,b),   \quad 
z \in \rho(S_R) \cap \rho(S_{0,0}),
\end{split}
\end{equation}
is nonzero and
\begin{equation}\lb{13.42}
u_R(z,\cdot \,)=R_{2,2}u_2(z, \cdot \,)+u_1(z, \cdot \,),   \quad 
z \in \rho(S_R) \cap \rho(S_{0,0}).
\end{equation}
The proof of \eqref{13.40} follows the proof of \cite[Theorem\ 3.2\,$(ii)$]{CGNZ12} {\it mutatis mutandis}.  
As a result of the resolvent identity \eqref{13.40}, one obtains the following relation for Green's 
functions,
\begin{equation}\lb{13.43}
\begin{split}
G_{z,R}(x,x') - G_{z,0,0}(x,x')= - q_R(z)^{-1} u_R(z,x) u_R(z,x'),&\\
z \in \rho(S_R) \cap \rho(S_{0,0}),&
\end{split}
\end{equation}
and consequently,
\begin{equation}\lb{13.44}
0\leq G_{\lambda,R}(b,b)=-q_R(\lambda)^{-1}, \quad \lambda < \inf (\sigma(S_R)).
\end{equation}
Nonnegativity of the solutions $u_j(\lambda, \cdot \,)$, $j=1,2$, together with the condition $R_{2,2}>0$ guarantees that $u_R(\lambda, \cdot \,)$ is nonnegative on $[a,b]$.  Hence, \eqref{13.43} implies
\begin{equation}\lb{13.45}
G_{\lambda,R}(x,x')-G_{\lambda,0,0}(x,x')\geq 0,\quad x,x'\in[a,b], \; 
\lambda < \inf (\sigma(S_R)),
\end{equation}
which is equivalent to the fact that the resolvent difference,
\begin{equation}\lb{13.46}
(S_R-\lambda I_r)^{-1}-(S_{0,0}-\lambda I_r)^{-1},\quad 
\lambda < \inf (\sigma(S_R)),
\end{equation}  
is positivity preserving.  Applying \cite[Corollary\ 9]{KR80} once again, one obtains the 
stronger result that the resolvent difference in \eqref{13.46} is positivity improving, and that 
$(S_R-\lambda I_r)^{-1}$, $\lambda<\inf(\sigma(S_R))$, is positivity improving as well. 
Again, inequality \eqref{10.142a} is merely a restatement of \eqref{13.45}. This completes 
the case $R_{1,2} = 0$.
\end{proof}

We chose to rely on different strategies of proof of positivity preserving in the case of separated and coupled boundary 
conditions to illustrate the different possible approaches in this context. The principal observation in the 
proof of Theorem \ref{t12.3} in connection with separated boundary conditions is the statement in 
\eqref{10.144} that the corresponding Green's function is nonnegative along the diagonal, and follows 
from nonnegativity of the resolvent (in the operator sense) at points below the spectrum of 
$S_{\varphi_a,\varphi_b}$.  A much more general result regarding nonnegativity along the diagonal 
of the (continuous) integral kernel associated with a nonnegative integral operator may be found in 
\cite[Lemma on p.\ 195]{Jo82} in connection with Mercer's theorem \cite[Theorem 8.11]{Jo82}.  

In the particular case where $p=r=1$, $q=s=0$ a.e.\ on $(a,b)$ in Theorem \ref{t12.3}, the positivity preserving result has been derived by Feller \cite{Fe57} (see also \cite[p.\ 147]{FOT11}). In fact, he considered a more general situation involving a Radon--Nikodym derivative (i.e., he worked in the context of a measure-valued coefficient). We also mention that the sign of the Green's function associated with the periodic Hill equation has been studied in connection with the existence of so-called comparison principles in \cite{CC12} (and the references therein).

The fact that positivity preserving and positivity improving are equivalent notions in the regular case appears to be a new result. 

We conclude with some comments on the Krein--von Neumann extension of $\Tmin$: 

\begin{remark}\lb{r12.4}
Given Hypothesis \ref{h11.1} and assuming $\Tmin \geq \varepsilon I_r$ for some $\varepsilon > 0$, 
the fact \eqref{10.148}, that is, $\text{dim}\big(\text{ker}\big(\Tmin^* \big) \big)=2$, together with 
\eqref{Fr-4Tf}, yields a degenerate ground state $0 \in \sigma_{p} (S_K)$. Hence, $S_K$ cannot be positivity preserving (cf., e.g., \cite[Theorem\ XIII.44]{RS78}). This fact is known under more restrictive assumptions on the coefficients of $\tau$ (cf.\ \cite[p.\ 147]{FOT11}). In the particular case $q=0$ a.e.\ on $(a,b)$, this 
can directly be read off from Theorem \ref{t12.3} since 
\begin{equation}
R_{K,1,2}^{(0)} = e^{-\int_a^b s(t) dt}\int_a^b p(t)^{-1}e^{2\int_a^t s(t')dt'}dt > 0 
\end{equation}
violates condition \eqref{12.12}. (In the general case $q \neq 0$ a.e.\ on $(a,b)$ one also has 
$R_{K,1,2} > 0$ as $R_{K,1,2} \neq 0$ by \eqref{RK}, but now a direct proof of $u_1^{[1]}(0,a) > 0$
requires a lengthy disconjugacy argument).    
\end{remark}


\begin{appendix}

\section{Sesquilinear Forms in the Regular Case}    \lb{sA}

In this appendix we discuss the underlying sesquilinear forms associated with self-adjoint 
extensions of $T_{\min}$ in the regular case with separated boundary conditions, closely following the treatment in \cite[Appendix\ A]{GST96}.

The standing assumption throughout this appendix will be the following:

\begin{hypothesis}\lb{hA.1}
Assume Hypothesis \ref{h2.1} holds with $p>0$ a.e.\ on $(a,b)$ and that $\tau$ is regular on $(a,b)$. 
Equivalently, we suppose that 
$p$, $q$, $r$, $\foco$ are Lebesgue measurable on $(a,b)$ with $p^{-1}$, $q$, $r$, $\foco\in L^1((a,b);dx)$ and real-valued a.e.\ on $(a,b)$ with $p$, $r>0$ a.e.\ on $(a,b)$.
\end{hypothesis}

Our goal is to explore relative boundedness of certain sesquilinear forms in the Hilbert space $L^2((a,b);r(x)dx)$ defined in connection with $\tau$. 
Assuming Hypothesis \ref{hA.1}, one may use the function $q$ to define a sesquilinear form in $L^2((a,b);r(x)dx)$ as follows
\begin{align}
&\mathfrak{Q}_{q/r}(f,g)=\int_a^b \overline{f(x)}q(x)g(x) \, dx,     \lb{A.1}\\
& f,g\in {\rm dom}(\mathfrak{Q}_{q/r})=\big\{h\in L^2((a,b);r(x)dx)\, \big|\, 
(|q|/r)^{1/2}h\in L^2((a,b);r(x)dx) \big\}.\no
\end{align}
Evidently, $\mathfrak{Q}_{q/r}$ is densely defined and symmetric.  

In order to define other sesquilinear forms, we first define two families of operators indexed by 
$\alpha, \beta \in \{0,\infty\}$, in $L^2((a,b);r(x)dx)$, as follows
\begin{align}
&A_{\alpha,\beta}f=\upsilon f,    \no \\
& (\upsilon f)(x)=[p(x)r(x)]^{-1/2}f^{[1]}(x) \, \text{ for a.e.\ $x\in(a,b)$,}      \lb{A.3a} \\
&f\in \dom {A_{\alpha,\beta}}=\big\{g\in L^2((a,b);r(x)dx)\, \big|\,g\in AC([a,b]), \, 
\upsilon g\in L^2((a,b);r(x)dx),    \no \\
& \hspace*{6.65cm} g(a)=0\, \text{if $\alpha=\infty$},\, g(b)=0\, \text{if $\beta=\infty$}\big\},  \no\\
&A_{\alpha,\beta}^+f=\upsilon^+f,    \no \\
& (\upsilon^+f)(x)=-[p(x)r(x)]^{-1}\big([p(x)r(x)]^{1/2}f\big)^{\{1\}}(x) \, \text{ for a.e.\ $x\in (a,b)$,}  \lb{A.4a} \\
&f\in {\rm dom}\big(A_{\alpha,\beta}^+\big)=\big\{g\in L^2((a,b);r(x)dx)\, \big| \, (pr)^{1/2}g\in AC([a,b]),\no\\
& \hspace*{4mm} \upsilon^+g\in L^2((a,b);r(x)dx),\, \big((pr)^{1/2}g\big)(a)=0\, \text{if $\alpha=0$},\, \big((pr)^{1/2}g\big)(b)=0\, \text{if $\beta=0$}  \big\}.   \no
\end{align}
Here we recall that 
\begin{equation}
f^{[1]}(x) = p(x) \big[f'(x) + \foco(x) f(x)\big] \, \text{ for a.e.\ $x\in (a,b)$,} \; f\in AC([a,b]),    \lb{A.5A}
\end{equation}
denotes the {\it first quasi-derivative} of $f$, whereas the superscript $\{1\}$ denotes the modified 
quasi-derivative of functions in $AC([a,b])$,
\begin{equation}\lb{A.5aa}
f^{\{1\}}(x) = p(x) \big[f'(x) - \foco(x)f(x)\big] \, \text{ for a.e.\ $x\in (a,b)$}, \; f\in AC([a,b]).
\end{equation}

\begin{lemma} \lb{lA.2} 
Assume Hypothesis \ref{hA.1} with $q=0$ a.e.\ in $(a,b)$.  Then the following items $(i)$--$(iv)$ hold:\\
$(i)$ $A_{\alpha,\beta}$ and $A_{\alpha,\beta}^+$ are densely defined in $L^2((a,b);r(x)dx)$ for all $\alpha$, $\beta\in\lbrace 0,\infty\rbrace$.\\ 
$(ii)$ $A_{\alpha,\beta}^*=A_{\alpha,\beta}^+$ and $A_{\alpha,\beta}=(A_{\alpha,\beta}^+)^*$ for 
all $\alpha, \beta \in \{0,\infty\}$.  In particular, $A_{\alpha,\beta}$ and $A_{\alpha,\beta}^+$ are 
closed in $L^2((a,b);r(x)dx)$ for all $\alpha, \beta \in \{0,\infty\}$.\\
$(iii)$ $A_{\alpha,\beta}^*A_{\alpha,\beta}=S_{\alpha,\beta}^{(0)}$, $\alpha, \beta\in \{0,\infty\}$, 
where $S_{\alpha,\beta}^{(0)}$ in $L^2((a,b);r(x)dx)$ denotes the operator defined by
\begin{align}
&S_{\alpha,\beta}^{(0)}f = \tau^{(0)} f, \quad \alpha,\, \beta \in \{0,\infty\}, \no\\
&  f\in {\rm dom}\big(S_{\alpha,\beta}^{(0)}\big)=\big\{g\in L^2((a,b);r(x)dx) \, \big| \, g,g^{[1]}\in AC([a,b]),\lb{A.5a}\\
& \tau^{(0)} g\in L^2((a,b);r(x)dx), \, \big(g^{[1]} \big)(a)+\alpha g(a)=\big(g^{[1]} \big)(b)+\beta g(b)=0 \big\},\no
\end{align}
where, by convention, $\alpha=\infty$ $($resp., $\beta=\infty$$)$ corresponds to the Dirichlet 
boundary condition $g(a)=0$ $($resp., $g(b)=0$$)$ and $\tau^{(0)}$ is given by 
\begin{align}
& (\tau^{(0)} f)(x) = \frac{1}{r(x)} \left( - \big(p(x)[f'(x) + \foco(x) f(x)]\big)' + \foco(x) p(x)[f'(x) 
+ \foco(x) f(x)]\right)   \no \\ 
& \hspace*{5.4cm} \text{for a.e.\ $x\in(a,b)$,} \; f, f^{[1]}\in AC([a,b]). 
\end{align}
$(iv)$ The operator $S_{\alpha,\beta}^{(0)}$ is a self-adjoint restriction of $T_{\max}$ 
$($equivalently, a self-adjoint extension of $T_{\min}$$)$ for all 
$\alpha, \beta \in \{0,\infty\}$ for $q=0$ a.e.\ on $(a,b)$. 
In particular, $S_{\infty,\infty}^{(0)}$ is the Friedrichs extension of $T_{\min}$ 
for $q=0$ a.e.\ on $(a,b)$. 
\end{lemma}
\begin{proof}
First of all, define operators $K$ and $\hatt K$ as follows
\begin{align}
&K:L^2((a,b);r(x)dx)\rightarrow \dom{A_{\infty,0}},\no\\
&\quad  g\mapsto \E^{-\int_a^x \foco(t)dt}\int_a^x 
\frac{g(x')\E^{\int_a^{x'} \foco(t)dt}}{[p(x')r(x')]^{1/2}} \, r(x') dx',\lb{A.6aa}\\
&\hatt K:L^2((a,b);r(x)dx)\rightarrow \dom{A_{0,\infty}^+},\no\\
&\quad  g\mapsto -[p(x)r(x)]^{-1/2}\E^{\int_a^x \foco(t)dt}\int_a^x g(x') \E^{-\int_a^{x'}\foco(t)dt} \, r(x') dx'.
\lb{A.7aa}
\end{align}
With these definitions, one readily verifies by direct computation that 
\begin{equation}\lb{A.8aa}
\begin{split}
&(Kg)(a)=0,\quad\quad\quad\quad\ \  \, \upsilon Kg=g,\\
&\big((pr)^{1/2} \hatt Kg\big)(a)=0,\quad \upsilon^+ \hatt K g=g,
\end{split}
\qquad g\in L^2((a,b);r(x)dx).
\end{equation}
Furthermore, we denote by $T_0^{(0)}$ the minimal operator introduced in \eqref{3.3} with $q=0$ a.e.\ in $(a,b)$. 
Then 
\begin{equation}
{\rm dom}\big(T_0^{(0)}\big) \subset \dom {A_{\alpha,\beta}}, \quad \alpha, \beta \in \{0,\infty\}, 
\end{equation}
rendering $A_{\alpha,\beta}$ densely defined, since for $f \in {\rm dom}\big(T_0^{(0)}\big)$, 
\begin{equation}
\|v f\|_{2,r} =  \max_{x\in[a,b]} \big|f^{[1]}(x)\big|^2 \int_a^b \frac{dx}{p(x)} < \infty, 
\end{equation} 
employing $f^{[1]} \in AC([a,b])$. In order to prove that $A_{\alpha,\beta}^+$ is densely defined 
as well, let $f\in \dom{A_{0,0}^+}^\bot$ and set $g = K f$. 
Because of $v^+ \widehat{K} g_0 = g_0$ for all $g_0\in L^2((a,b);r(x)dx)$, one concludes that 
\begin{equation} 
g_0\in\ran(A_{0,0}^+) \, \text{ if and only if } \, \big((pr)^{1/2} \widehat{K}g_0\big)(b)=0.
\end{equation}  
As a result, one infers that  
\begin{equation}
 \ran(A_{0,0}^+) = \left\lbrace \E^{-\int_a^x s(t)dt} \right\rbrace^\bot. 
\end{equation}
Next, one computes for arbitrary $ h\in{\rm dom}(A_{0,0}^+)$,
\begin{align}\label{eqnorth}
\langle g, A^+_{0,0} h \rangle_r &= \int_a^b g(x) \overline{(A_{0,0}^+ h)(x)} \, r(x)dx  \no \\
&= - \int_a^b g(x) \ol{\big[\big([p(x) r(x)]^{1/2} h(x)\big)' - s(x) [p(x) r(x)]^{1/2} h(x)\big]} \, dx  \no \\
&= - g(x) \ol{\big((pr)^{1/2} h\big)(x)}\bigg|_a^b + \int_a^b g'(x) [p(x) r(x)]^{1/2} \ol{h(x)} \, dx  \no \\
& \quad + \int_a^b g(x) s(x) [p(x) r(x)]^{1/2} \ol{h(x)} \, dx   \no \\ 
&= \int_a^b [p(x) r(x)]^{-1/2} p(x)[g'(x) + s(x) g(x)] \ol{h(x)} \, r(x)dx   \no \\ 
&= \int_a^b (v g)(x) \ol{h(x)} \, r(x) dx = \int_a^b (v K f)(x) \ol{h(x)} \, r(x) dx   \no \\
&= \int_a^b f(x) \ol{h(x)} \, r(x)dx = \langle f, h\rangle_r = 0, 
\end{align}
since by hypothesis, $f\in \dom{A_{0,0}^+}^\bot$. Thus, we have $g \in  \ran(A_{0,0}^+)^{\bot}$, 
implying that $g= c \, \E^{-\int_a^x s(t)dt}$ for some constant $c \in \bbC$.
By the definition \eqref{A.3a} of $v$, it is readily verified that $f = v g = 0$ a.e.\ on $[a,b]$. 
Thus, ${\rm dom} (A_{0,0}^+)$, and hence 
${\rm dom}(A_{\alpha,\beta}^+) \supseteq {\rm dom} (A_{0,0}^+)$, 
$\alpha, \beta \in \{0, \infty\}$, is dense in $L^2((a,b); r(x)dx)$, completing the proof of item $(i)$. 

Regarding item $(ii)$, we only show $A_{\alpha,\beta}^*=A_{\alpha,\beta}^+$ as the case 
$\big(A_{\alpha,\beta}^+ \big)^*=A_{\alpha,\beta}$ is handled analogously.  Moreover, 
since $A_{\infty,\infty}\subseteq A_{\alpha,\beta}$ (this follows by definition of the operators) 
implies $A_{\alpha,\beta}^*\subseteq A_{\infty,\infty}^*$, we only prove 
$A_{\infty,\infty}^*=A_{\infty,\infty}^+$, the other cases follow from an additional integration by parts.  Therefore, first note that $A_{\infty,\infty}^+\subseteq A_{\infty,\infty}^*$ as an integration by parts shows
\begin{align}
\big\langle f, A_{\infty,\infty}^+g\big\rangle_r&=\int_a^b f(x)\overline{(\upsilon^+g)(x)} r(x)dx \no \\
&=-\int_a^b p(x)^{-1}f(x)\overline{\big((pr)^{1/2}g \big)^{\{1\}}(x)}dx\no\\
&=-\int_a^b f(x) \bigg[\overline{\big((pr)^{1/2}g\big)'(x)-\foco(x)\big((pr)^{1/2}g\big)(x)}\bigg]dx\no\\
&=-f(x)\overline{\big((pr)^{1/2}g\big)(x)}\big|_a^b+\int_a^b [p(x)r(x)]^{1/2}[f'(x)+\foco(x)f(x)]\overline{g(x)}dx\no\\
&=\int_a^b \frac{[p(x)r(x)]^{1/2}}{p(x)r(x)}p(x)[f'(x)+\foco(x)f(x)]\overline{g(x)}r(x)dx\no\\
&=\int_a^b [p(x)r(x)]^{-1/2}f^{[1]}(x)\overline{g(x)} r(x)dx\no\\
&=\big\langle A_{\infty,\infty}f,g \big\rangle_r,\quad f\in {\rm dom}\big(A_{\infty,\infty}\big),~ g\in {\rm dom}\big(A_{\infty,\infty}^+\big).   \lb{A.9aa}
\end{align}
Hence it remains to show ${\rm dom}\big(A_{\infty,\infty}^*\big)\subseteq {\rm dom}\big(A_{\infty,\infty}^+\big)$.  To this end, let $f\in {\rm dom}\big(A_{\infty,\infty}^*\big)$, and set $g= \hatt K A_{\infty,\infty}^*f$.  Then one computes
\begin{align}
\begin{split} 
&\int_a^b \overline{(f(x)-g(x))}(A_{\infty,\infty}h)(x)r(x)dx   \\
&\quad =\int_a^b \big[\big(A_{\infty,\infty}^*\overline{f}\big)(x)-(\upsilon^+\overline{g})(x) \big]h(x)r(x)dx=0, \quad h\in {\rm dom}\big(A_{\infty,\infty}\big).   \lb{A.10aa} 
\end{split} 
\end{align}
Consequently, $\ran\big(A_{\infty,\infty}\big)$ is contained in the kernel of the linear functional $k\mapsto \langle k , f-g \rangle_r$, $k\in L^2((a,b);r(x)dx)$.  On the other hand, since $\upsilon Kg_0=g_0$ for all $g_0\in L^2((a,b);r(x)dx)$, one infers that $g_0\in \ran\big(A_{\infty,\infty} \big)$ if and only if $(Kg_0)(b)=0$. 
As a result, 
\begin{equation}\lb{A.11aa}
\ran\big(A_{\infty,\infty}\big)=\Big\{(pr)^{-1/2}\E^{\int_a^x \foco(t)dt} \Big\}^{\perp}.
\end{equation}
On the other hand, \eqref{A.10aa} shows that $f-g$ is orthogonal to $\ran\big(A_{\infty,\infty}\big)$, and because of \eqref{A.11aa}, there exists a constant $c$ such that $f=g+c(pr)^{-1/2}\E^{\int_a^x \foco(t)dt}$.  It is a simple matter to check that $(pr)^{-1/2}\E^{\int_a^x \foco(t)dt}\in \dom {A_{\infty,\infty}^+}$  (in fact, $\upsilon^+$ applied to $(pr)^{-1/2}\E^{\int_a^x \foco(t)dt}$ is zero).  Therefore, by \eqref{A.11aa}, $f\in {\rm dom}\big(A_{\infty,\infty}^+\big)$, completing the proof of item $(ii)$.

To prove item $(iii)$, one notes that by item $(ii)$, 
\begin{equation}\lb{A.12aa}
{\rm dom}\big(A_{\alpha,\beta}^*A_{\alpha,\beta}\big)=\big\{g\in {\rm dom}\big(A_{\alpha,\beta}\big)\, 
\big| \, \upsilon g\in {\rm dom}\big(A_{\alpha,\beta}^+\big)\big\},
\end{equation}
so that, by inspection, one obtains ${\rm dom}\big(A_{\alpha,\beta}^*A_{\alpha,\beta}\big)
={\rm dom}\big(S_{\alpha,\beta}^{(0)}\big)$, $\alpha, \beta\in \{0,\infty\}$.  Then for 
$f\in {\rm dom}\big(S_{\alpha,\beta}^{(0)}\big)$, a simple computation shows 
$A_{\alpha,\beta}^*A_{\alpha,\beta}f=\upsilon^+(\upsilon f)=S_{\alpha,\beta}^{(0)}f$, $\alpha$, 
$\beta \in \{0,\infty\}$.  This completes the proof of item $(iii)$.

Since $A_{\alpha,\beta}$ is densely defined and closed for all $\alpha, \beta \in \{0,\infty\}$, the 
operator $S_{\alpha,\beta}^{(0)}=A_{\alpha,\beta}^*A_{\alpha,\beta}$ is self-adjoint and nonnegative 
(cf., e.g., \cite[Theorem V.3.24]{Ka80}). In addition, $S_{\alpha,\beta}^{(0)}$ is a restriction of 
$T_{\max}$, and that $S_{\infty,\infty}^{(0)}$ is the Friedrichs extension of $T_{\min}$ 
(for $q=0$ a.e.\ on $(a,b)$) follows from \eqref{10.101} and the assumed regularity of $\tau$ 
on $(a,b)$, proving item $(iv)$.
\end{proof}

With the operators $A_{\alpha,\beta}$, $\alpha, \beta\in \{0,\infty\}$, in hand, we define the densely defined, closed, nonnegative sesquilinear form by \\
\begin{equation}\lb{A.13aa}
\begin{split}
\mathfrak{Q}^{(0)}_{\alpha,\beta}(f,g)=\langle A_{\alpha,\beta}f,A_{\alpha,\beta}g\rangle_r,\quad f,\, g\in {\rm dom}\big(\mathfrak{Q}^{(0)}_{\alpha,\beta}\big) 
={\rm dom}\big(A_{\alpha,\beta}\big),&\\
 \alpha,\, \beta \in \{0,\infty\}.&
\end{split}
\end{equation}
The self-adjoint and nonnegative operator in $L^2((a,b);r(x)dx)$ uniquely associated 
with the sesquilinear form $\mathfrak{Q}^{(0)}_{\alpha,\beta}$, $\alpha, \beta \in \{0,\infty\}$,  
is then given by 
\begin{equation}\lb{A.14aa}
A_{\alpha,\beta}^* A_{\alpha,\beta} = S_{\alpha,\beta}^{(0)},  \quad \alpha,\, \beta \in \{0,\infty\},
\end{equation}
where $S_{\alpha,\beta}^{(0)}$ is the operator defined in \eqref{A.5a}.

Since functions in ${\rm dom}\big(\mathfrak{Q}^{(0)}_{\alpha,\beta}\big)$, 
$\alpha, \beta \in \{0,\infty\}$, are absolutely continuous on $[a,b]$, one infers
\begin{equation}\lb{A.4}
{\rm dom}\big(\mathfrak{Q}^{(0)}_{\alpha,\beta}\big) \subset \dom{\mathfrak{Q}_{q/r}}, \quad \alpha,\, \beta \in \{0,\infty\}.
\end{equation}
Finally, we define a family of sesquilinear forms, indexed by pairs of real numbers $\gamma, \nu \in \R$, 
as follows
\begin{equation}\lb{A.5}
\begin{split}
\mathfrak{Q}^{a,b}_{\gamma,\nu} (f,g)=\nu\overline{f(a)}g(a) - \gamma\overline{f(b)}g(b),
\quad f,\, g\in \dom{\mathfrak{Q}^{a,b}_{\gamma,\nu}}=AC([a,b]).\\
\end{split}
\end{equation}
In addition, we set 
\begin{equation} 
\mathfrak{Q}^{a,b}_{\infty,\nu} (f,g)=\mathfrak{Q}^{a,b}_{0,\nu} (f,g), \quad 
\mathfrak{Q}^{a,b}_{\gamma,\infty} (f,g)=\mathfrak{Q}^{a,b}_{\gamma,0} (f,g), \quad 
\mathfrak{Q}^{a,b}_{\infty,\infty} (f,g)=0.
\end{equation} 

\begin{lemma} \lb{lA.3}
Assume Hypothesis \ref{hA.1}.  Then the following items $(i)$ and $(ii)$ hold: \\
$(i)\,$ $\mathfrak{Q}_{q/r}$ and $\mathfrak{Q}_{|\foco|/r}$ are relatively form compact $($and hence infinitesimally bounded\,$)$ with respect to $\mathfrak{Q}^{(0)}_{\alpha,\beta}$ for all 
$\alpha, \beta \in \{0,\infty\}$, that is, 
\begin{equation}
|q/r|^{1/2} \big(S_{\alpha,\beta}^{(0)} + I_r \big)^{-1/2}, \, 
|s/r|^{1/2} \big(S_{\alpha,\beta}^{(0)} + I_r \big)^{-1/2}  \in \cB_{\infty} \big(L^2((a,b); r(x) dx)\big). 
\lb{A.5Aa} 
\end{equation}
In fact, compactness in \eqref{A.5Aa} can be replaced by the Hilbert--Schmidt property 
$($cf.\ \eqref{A.6}$)$. \\
$(ii)$ For each $\gamma$, $\nu\in \bbR$, the sesquilinear form $\mathfrak{Q}^{a,b}_{\gamma,\nu}$ is infinitesimally bounded with respect to $\mathfrak{Q}^{(0)}_{\alpha, \beta}$ for all $\alpha$, 
$\beta \in \{0,\infty\}$.
\end{lemma}
\begin{proof}
In item $(i)$, it clearly suffices to prove the claim for $\mathfrak{Q}_{q/r}$ only since $|\foco|$ and $q$ satisfy the same assumptions.  Let $G^{(0)}_{z,\alpha,\beta}(\cdot , \cdot)$, $z\in\bbC\backslash \bbR$ and 
$\alpha, \beta \in \{0,\infty\}$, denote the Green's function for the operator $S^{(0)}_{\alpha,\beta}$ in \eqref{A.14aa} (known to exist by Theorem \ref{thmSResolLCLC}).  Then 
\begin{equation}\lb{A.6}
\begin{split}
|q/r|^{1/2}\big(S^{(0)}_{\alpha,\beta}-zI_r\big)^{-1}|q/r|^{1/2}\in \cB_2\big(L^2((a,b);r(x)dx)\big), 
\quad z\in \bbC\backslash\bbR,&\\ 
\alpha,\, \beta \in\{0,\infty\},&
\end{split}
\end{equation}
since
\begin{align}
\int_a^b \int_a^b \frac{|q(x)|}{r(x)}\big|G^{(0)}_{z,\alpha,\beta}(x,x')\big|^2\frac{|q(x')|}{r(x')}r(x)dx\,r(x')dx'
\leq C(z,\alpha,\beta)\|q\|_{L^1((a,b);dx)}^2,   \lb{A.7}
\end{align}
for some constant $C(z,\alpha,\beta)$, because $G^{(0)}_{z,\alpha,\beta}(\cdot,\cdot)$ is uniformly bounded on $(a,b) \times (a,b)$ for all $\alpha, \beta \in \{0,\infty\}$ by \eqref{7.0} or \eqref{7.4}. This completes the 
proof of item $(i)$.

In order to prove item $(ii)$, fix $\alpha, \beta \in \{0,\infty\}$, and note that for arbitrary 
$c \in [a,b]$ and any function $f\in {\rm dom}\big(\mathfrak{Q}^{(0)}_{\alpha,\beta}\big)\subset 
\dom{\mathfrak{Q}_{\gamma,\nu}}$,
\begin{align}
|f(c)|^2 & = \biggl| f(x)^2 - 2 \int_c^x f(t) f'(t)dt\bigg|   \no \\
&\leq |f(x)|^2 + 2 \int_a^b \big|f(t) f'(t)+\foco(t)f(t)^2\big|dt\no\\
&\quad +2\int_a^b |\foco(t)||f(t)|^2 dt, \quad f \in {\rm dom}\big(\mathfrak{Q}^{(0)}_{\alpha,\beta}\big).   \lb{A.8}
\end{align}
One infers (after taking the supremum over all $c \in [a,b]$, multiplying by $r$, and integrating w.r.t. $x$ from $a$ to $b$) for 
any $\varepsilon>0$,
\begin{align}
&\| f \|_{L^\infty((a,b); dx)}^2 \no\\
&\quad \leq \|r\|^{-1}_{L^1((a,b); dx)} \| f \|^2_{2,r} + 2 \int_a^b
\frac{|f(t)|}{(\varepsilon p(t)/2)^{1/2}}\, (\varepsilon p(t)/2)^{1/2}
|f'(t)+\foco(t)f(t)| dt \no\\
&\quad\quad +2\mathfrak{Q}_{|\foco|/r}(f,f)    \no \\
&\quad \leq \|r\|^{-1}_{L^1((a,b); dx)} \| f \|^2_{2,r} + \int_a^b \biggl(\frac{2}
{\varepsilon}\, \frac{|f(t)|^2}{p(t)} + \frac{\varepsilon}{2}\, 
\frac{\big|f^{[1]}(t)\big|^2}{p(t)} \bigg)dt     \no \\
& \qquad +2\mathfrak{Q}_{|\foco|/r}(f,f), \quad f \in {\rm dom}\big(\mathfrak{Q}^{(0)}_{\alpha,\beta}\big). 
\lb{A.9}  
\end{align}
Since $0<p^{-1} \in L^1((a,b);dx)$, there exists a $\delta_1(\varepsilon) > 0$ such
that $\int_{I_1(\varepsilon)} p(t)^{-1}\, dt \leq
\frac{\varepsilon}{8}$ with $I_1(\varepsilon) = \{x\in (a,b)\, | \,
p(x)< \delta_1(\varepsilon)\}$. Thus,
\begin{equation}
\begin{split}
\int_a^b \frac{|f(t)|^2}{p(t)}dt &=
\int_{I_1(\varepsilon)} \frac{|f(t)|^2}{p(t)}dt +
\int_{(a,b)\backslash I_1(\varepsilon)}
\frac{|f(t)|^2}{p(t)}dt \\
&\leq \frac{\varepsilon}{8}\,
 \| f \|_{L^\infty((a,b);dx)}^2 + \frac{1}{\delta_1(\varepsilon)}\,  \int_a^b |f(t)|^2 dt, 
 \quad f \in {\rm dom}\big(\mathfrak{Q}^{(0)}_{\alpha,\beta}\big).
 \end{split}
\end{equation}
In addition, since $r > 0$ a.e.\ on $(a,b)$, there exists a $\delta_2(\varepsilon)>0$ such
that $|I_2(\varepsilon)| \leq \frac{\varepsilon \delta_1(\varepsilon)}{8}$ with $I_2(\varepsilon) = \{x\in (a,b)\, | \,
r(x)< \delta_2(\varepsilon)\}$. Thus,
\begin{equation}
\begin{split}
\int_a^b |f(t)|^2\, dt &=
\int_{I_2(\varepsilon)}|f(t)|^2\, dt +
\int_{(a,b)\backslash I_2(\varepsilon)}
|f(t)|^2\, dt \\
&\leq \frac{\varepsilon \delta_1(\varepsilon)}{8}\,
 \| f \|_{L^\infty((a,b);dx)}^2 + \frac{1}{\delta_2(\varepsilon)}\, \|f\|^2_{2,r}, 
 \quad f \in {\rm dom}\big(\mathfrak{Q}^{(0)}_{\alpha,\beta}\big).
 \end{split}
\end{equation}
Consequently, one obtains from \eqref{A.9},
\begin{align} 
\begin{split} 
 \| f \|_{L^\infty((a,b);dx)}^2 &\leq 2\big\{\|r\|^{-1}_{L^1((a,b); dx)} +
2[\varepsilon\delta_1(\varepsilon)\delta_2(\varepsilon)]^{-1}\big\} \| f \|_{2,r}^2     \\
&\quad + \varepsilon \mathfrak{Q}^{(0)}_{\alpha,\beta}(f,f) + 4\mathfrak{Q}_{|\foco|/r}(f,f), 
\quad f \in {\rm dom}\big(\mathfrak{Q}^{(0)}_{\alpha,\beta}\big).   \lb{A.12}
\end{split} 
\end{align}
By part $(i)$, $\mathfrak{Q}_{|\foco|/r}$ is infinitesimally bounded with respect to $\mathfrak{Q}^{(0)}_{\alpha,\beta}$.  Hence, there exists $\eta(\varepsilon)>0$ such that
\begin{equation}
\mathfrak{Q}_{|\foco|/r}(f,f)\leq \frac{\varepsilon}{4}\mathfrak{Q}^{(0)}_{\alpha,\beta}(f,f)+\eta(\varepsilon)\|f\|_{2,r}^2, 
\quad f \in {\rm dom}\big(\mathfrak{Q}^{(0)}_{\alpha,\beta}\big).
\end{equation}
As a result, \eqref{A.12} implies
\begin{align}
\begin{split} 
 \| f \|_{L^\infty((a,b);dx)}^2 & \leq 2\big\{\|r\|^{-1}_{L^1((a,b); dx)} +
2[\varepsilon\delta_1(\varepsilon)\delta_2(\varepsilon)]^{-1}+2\eta(\varepsilon)\big\} \| f \|_{2,r}^2   \\
& \quad + 2\varepsilon \mathfrak{Q}^{(0)}_{\alpha,\beta}(f,f), \quad f \in {\rm dom}\big(\mathfrak{Q}^{(0)}_{\alpha,\beta}\big).  \lb{A.15}
\end{split} 
\end{align}
Infinitesimal boundedness of $\mathfrak{Q}^{a,b}_{\gamma,\nu}$ with respect to 
$\mathfrak{Q}^{(0)}_{\alpha,\beta}$ follows since $\eps>0$ and 
$f\in {\rm dom}\big(\mathfrak{Q}^{(0)}_{\alpha,\beta}\big)$ were arbitrary. 
\end{proof}

Finally, introducing the densely defined, closed, and lower semibounded sesquilinear forms in 
$L^2((a,b);r(x)dx)$
\begin{align} 
& \mathfrak{Q}_{\alpha,\beta}(f,g) = \mathfrak{Q}^{(0)}_{0,0}(f,g)
+ \mathfrak{Q}_{q/r}(f,g) + \mathfrak{Q}^{a,b}_{\alpha,\beta} (f,g),  \lb{Q} \\
& \hspace*{10mm} f,\, g\in {\rm dom}\big(\mathfrak{Q}^{(0)}_{0,0}\big) 
={\rm dom}\big(A_{0,0}\big), \;  \alpha, \beta \in \R,     \no  \\
& \mathfrak{Q}_{\alpha,\infty}(f,g) = \mathfrak{Q}^{(0)}_{0,\infty}(f,g)
+ \mathfrak{Q}_{q/r}(f,g) + \mathfrak{Q}^{a,b}_{\alpha,0} (f,g),    \lb{Qainfty} \\
& \hspace*{15.5mm} f,\, g\in {\rm dom}\big(\mathfrak{Q}^{(0)}_{0,\infty}\big) 
={\rm dom}\big(A_{0,\infty}\big), \;  \alpha \in \R,     \no \\
& \mathfrak{Q}_{\infty,\beta}(f,g) = \mathfrak{Q}^{(0)}_{\infty,0}(f,g)
+ \mathfrak{Q}_{q/r}(f,g) + \mathfrak{Q}^{a,b}_{0,\beta} (f,g),    \lb{Qinftyb} \\
& \hspace*{15.3mm} f,\, g\in {\rm dom}\big(\mathfrak{Q}^{(0)}_{\infty,0}\big) 
={\rm dom}\big(A_{\infty,0}\big), \;  \beta \in \R,    \no  \\
& \mathfrak{Q}_{\infty,\infty}(f,g) = \mathfrak{Q}^{(0)}_{\infty,\infty}(f,g)
+ \mathfrak{Q}_{q/r}(f,g),    \lb{Qinfty} \\
& \hspace*{6.5mm} f,\, g\in {\rm dom}\big(\mathfrak{Q}^{(0)}_{\infty,\infty}\big) 
={\rm dom}\big(A_{\infty,\infty}\big),   \no
\end{align} 
and denoting the uniquely associated self-adjoint, and lower semibounded operator by 
$S_{\alpha,\beta}$, $\alpha, \beta \in \R \cup \{\infty\}$, the latter can be explicitly described as follows:

\begin{theorem} \lb{tA.4}
Define $ \mathfrak{Q}_{\alpha,\beta}$, $\alpha, \beta \in \R \cup \{\infty\}$, by   
\eqref{Q}--\eqref{Qinfty}. Then the uniquely associated 
self-adjoint, lower semibounded operator $S_{\alpha,\beta}$ in $L^2((a,b);r(x)dx)$ is given by 
\begin{align}
&S_{\alpha,\beta} f = \tau f, \quad \alpha,\, \beta \in \R\cup \{\infty\}, \no\\
&  f\in {\rm dom}\big(S_{\alpha,\beta}\big)=\big\{g\in L^2((a,b);r(x)dx) \, \big| \, g,g^{[1]}\in AC([a,b]), 
\lb{A.5b}\\
& g^{[1]}(a)+\alpha g(a) = g^{[1]}(b)+\beta g(b)=0, \, \tau g\in L^2((a,b);r(x)dx) \big\},\no
\end{align}
where, by convention, $\alpha=\infty$ $($resp., $\beta=\infty$$)$ corresponds to the Dirichlet 
boundary condition $g(a)=0$ $($resp., $g(b)=0$$)$. Moreover, the operator $S_{\alpha,\beta}$ is a 
self-adjoint restriction of $T_{\max}$ $($equivalently, a self-adjoint extension of $T_{\min}$$)$, in particular, $S_{\infty,\infty}$ is the Friedrichs extension $S_F$ of $T_{\min}$. 
\end{theorem}
\begin{proof}
It suffices to consider the Dirichlet case $\alpha=\beta=\infty$, the other cases being similar. 
We denote by $\widehat{S}_{\infty,\infty}$ the operator defined in \eqref{A.5b} for
$\alpha=\beta=\infty$ and by $S_{\infty,\infty}$ the unique operator
associated with $\mathfrak{Q}_{\infty,\infty}$. Choose $u \in \dom{\mathfrak{Q}_{\infty,\infty}}$ and 
$v \in {\rm dom}\big(\widehat{S}_{\infty,\infty}\big)$. Then an integration by parts yields
\begin{equation} 
\mathfrak{Q}_{\infty,\infty}(u,v) = \big\langle u, \widehat{S}_{\infty,\infty} v \big\rangle_r.
\end{equation}
Thus $\widehat{S}_{\infty,\infty} \subseteq S_{\infty,\infty}$ by
\cite[Corollary\ VI.2.4]{Ka80} and hence $\widehat{S}_{\infty,\infty} =
S_{\infty,\infty}$ since $\widehat{S}_{\infty,\infty} = S_F$ is self-adjoint. 
\end{proof}

\end{appendix}

\medskip
\noindent

\noindent 
{\bf Acknowledgments.} We are indebted to Rostyk Hryniv and Alexander Sakhnovich for very 
helpful discussions. We also sincerely thank the anonymous referee for the extraordinary 
efforts exerted in refereeing our manuscript, and for the numerous comments and suggestions  
kindly provided to us. G.T. gratefully acknowledges the stimulating
atmosphere at the Isaac Newton Institute for Mathematical Sciences in Cambridge
during October 2011 where parts of this paper were written as part of  the international
research program on Inverse Problems.


\end{document}